\documentclass{article}

\usepackage{graphicx,url}
\usepackage[all]{xy}
\usepackage{float}
\usepackage{appendix}
\usepackage{amsthm}
\usepackage{enumerate}
\usepackage{tikz-cd} 

\newtheorem{ourthm}{Theorem}

\newtheorem*{addC}{Addendum to Proposition~\ref{propAreaP}}

\newtheorem{lemma}{Lemma}[section]

\newtheorem{theorem}[lemma]{Theorem}%[section]

\newtheorem{proposition}[lemma]{Proposition}
\newtheorem{corollary}[lemma]{Corollary}

\theoremstyle{definition}

\newtheorem{definition}[lemma]{Definition}

\newtheorem{remark}[lemma]{Remark}

\renewcommand{\Re}{\operatorname{Re}}
 \renewcommand{\Im}{\operatorname{Im}}

\usepackage{amsfonts,amssymb,amsmath}
\usepackage{url}
\usepackage{tikz-cd}
\usepackage{mathtools,mathrsfs}
\usepackage[normalem]{ulem}

\numberwithin{equation}{section}

\newcommand{\C}{\ensuremath{\mathbb{C}}}

\newcommand{\Q}{\ensuremath{\mathbb{Q}}}
\newcommand{\R}{\ensuremath{\mathbb{R}}}

\newcommand{\Z}{\ensuremath{\mathbb{Z}}}

\newcommand{\Zp}{\Z_{>0}}
\newcommand{\Rp}{\R_{>0}}

\def\blue#1{{\color{blue} #1}}

\date{}
\usepackage{datetime2}
\newcommand{\datestamp}{{\small{File:\;\hbox{\tt\jobname.tex}
\; \DTMnow}}}
\newcommand\extrafootertext[1]{%
    \bgroup
    \renewcommand\thefootnote{\fnsymbol{footnote}}%
    \renewcommand\thempfootnote{\fnsymbol{mpfootnote}}%
    \footnotetext[0]{#1}%
    \egroup
  }

  \newcommand{\ie}{{\emph{i.e.}}\ }
  \newcommand{\ii}{^{-1}}
  \newcommand{\ti}{\tilde}
  \newcommand{\bg}{\bigskip}

  \newcommand{\gD}{\ensuremath{\mathcal{D}}}
  \newcommand{\cF}{\ensuremath{\mathcal{F}}}
  \newcommand{\gF}{\ensuremath{\mathscr{F}}}
  \newcommand{\gL}{\ensuremath{\mathscr{L}}}
  \newcommand{\gG}{\ensuremath{\mathscr{G}}}

\newcommand{\al}{\alpha}
\newcommand{\be}{\beta}

\newcommand{\De}{\Delta}
\newcommand{\om}{\omega}
\newcommand{\Ga}{\Gamma}
\newcommand{\la}{\lambda}
\newcommand{\La}{\Lambda}
\newcommand{\sig}{\sigma}
\newcommand{\ze}{\zeta}
\newcommand{\pa}{\partial}

\newcommand{\defeq}{\coloneqq}           % \newcommand{\defeq}{:=}
\newcommand{\col}{\colon\thinspace}     %% for a map f \col A \to B
\newcommand{\ov}{\overline}

\newcommand{\Fom}{\gF_{\!\omega}}
\newcommand{\Fomcv}{\Fom^{\hspace{-.07em}\De}}
\newcommand{\FomJcv}{\gF_{\om,J}^\De}
\newcommand{\Fomodd}{\gF_{\textnormal{odd},\om}}
\newcommand{\Fomoddcv}{\gF_{\textnormal{odd},\om}^\De}
\newcommand{\Sodd}{S_{\textnormal{odd}}}

\newcommand{\ID}{\mathop{\hbox{{\rm Id}}}\nolimits}

\newcommand{\begla}{\begin{equation}}
  \newcommand{\beglab}[1]{\begin{equation} \label{#1}}
    \newcommand{\edla}{\end{equation}}

  \newcommand{\ens}{\enspace}

  \newcommand{\cG}{\mathcal{G}}
  \newcommand{\cO}{\mathcal{O}}
  \newcommand{\gP}{\mathscr{P}}

  \newcommand{\Imp}{\quad\Rightarrow\quad}

  \newcommand{\imp}{\;\Rightarrow\;}

  \DeclarePairedDelimiter\abs{\lvert}{\rvert}

  \newcommand{\bz}{\bar z}

  \newcommand{\wrt}{{with respect to}}

  \newcommand{\ph}{\varphi}

  \newcommand{\eps}{\varepsilon}

  \newcommand{\lhs}{{left-hand side}}
\newcommand{\rhs}{{right-hand side}}

\newcommand{\dd}{{\mathrm d}}

\newcommand{\nidt}{\ \vspace{-1.4
    ex}

  \noindent --\ }

\usepackage[hyperindex=true, 
					hidelinks]{hyperref}
\begin{document} 

\title{Geometric normalization}

\author{Alain Chenciner\footnote{IMCCE, Observatoire de Paris and
    Universit\'e Paris Cit\'e}, David Sauzin\footnote{Capital Normal
    University, Beijing, on leave from CNRS--IMCCE, Observatoire de
    Paris}, Qiaoling Wei\footnote{Capital Normal University, Beijing}}
\maketitle

\vspace{2cm}

\begin{abstract}
For an analytic local diffeomorphism of $(\R^2,0)\equiv (\C,0)$ of the form\extrafootertext{\datestamp} 
\[
  F(z)=\lambda\bigl( z+\sum_{j+k\ge 2}F_{kl}z^j\bar z^k\bigr)
  \quad \text{with}\enspace  \lambda=e^{2\pi i\omega},
  \enspace \text{$\omega$ irrational},
\]
we introduce the notion of ``geometric normalization'', which includes
the classical formal normalizations as a special case: it is a formal
conjugacy to a formal diffeomorphism which preserves the foliation by
circles centered at $0$.  We show that geometric normalizations,
despite of non-uniqueness, correspond in a natural way to a unique
formal invariant foliation. We show, in various contexts, generic
  results of divergence for the geometric normalizations, which amount
  to the generic non-existence of any analytic invariant foliation.

\end{abstract}

%%%%%%%%%%%%%%%%%%%%%%%%%%%%%%%%%%%%%%%%%%%%%%%%%%%%%%%%
%%%%%%%%%%%%%%%%%%%%%%%%%%%%%%%%%%%%%%%%%%%%%%%%%%%%%%%%

\newpage

\tableofcontents

\bigskip % \vspace{1cm}

\begin{center} \rule{3cm}{.1pt} \end{center}

\bigskip % \vspace{1cm}

%%%%%%%%%%%%%%%%%%%%%%%%%%%%%%%%%%%%%%%%%%%%%%%%%%%%%%%%
%%%%%%%%%%%%%%%%%%%%%%%%%%%%%%%%%%%%%%%%%%%%%%%%%%%%%%%%

\section{Decoupling actions and angles}   \label{secdecoupl}

%%%%%%%%%%%%%%%%%%%%%%%%%%%%%%%%%%%%%%%%%%%%%%%%%%%%%%%%
%%%%%%%%%%%%%%%%%%%%%%%%%%%%%%%%%%%%%%%%%%%%%%%%%%%%%%%%

For $\om\in\R$, we denote by~$\Fom$ the set of all 
% local analytic
formal diffeomorphisms~$F$ of~$\R^2$ for which the origin is a fixed
point and the linear part $DF(0)$ is the rotation of angle $2\pi\om$.
%
% \marginlabel{There was inconsistency about CV in notations}
%
Identifying~$\R^2$ and~$\C$, the set $\Fom$ can be written as 
\[
 \Fom=\Big\{\, F(z) = \la z+\sum_{j+k\ge 2} F_{j,k} z^j \bar z^k
\mid F_{j,k}\in \mathbb{C} \,\Big\},\quad \text{with}\enspace \la \defeq e^{2\pi i\om}.
\]
Throughout the paper we make the ``non-resonance'' assumption
\begin{equation}
  \om \notin \Q,
\quad \text{\ie \,\! $\la$ is not a root of unity.}
\end{equation}
%
% We will always assume that $\omega$ is irrational throughout the paper without further indicated. 
%
We introduce a notion of normalization
% of a germ $F$ of analytic (resp. formal) diffeomorphism of
% $(\R^2,0)\equiv(\C,0)$
%
that, as will be shown later, addresses only the radial dynamical behaviour of~$F$:

\begin{definition}
  Given $F\in\Fom$, we call ``formal geometric normalization'' of~$F$
  any formal tangent-to-identity diffeomorphism
  \[
    \Phi \col z \mapsto \zeta=\Phi(z)=z+\sum_{j+k\ge2}
    \phi_{j,k}z^j\bar z^k
    \qquad (\phi_{j,k}\in\C)
  \]
  that conjugates~$F$ to a formal diffeomorphism
  $\ze \mapsto G(\ze)=\Phi\circ F\circ \Phi^{-1}(\ze)$
%
% $$=\lambda \zeta+\sum_{k+l\ge 2}G_{kl}\zeta^k\bar \zeta^l$$
%
such that
%which preserves the foliation by circles centered at 0.  This means precisely that
$|G(\ze)|^2$ depends only on $|\ze|^2$:
\begin{equation}   \label{eqcharacterisG}
  |G(\zeta)|^2=\Ga(|\zeta|^2)=
  |\zeta|^2 + \sum_{n\ge2}\Ga_n|\zeta|^{2n},
  \quad \text{with}\enspace \Ga_n\in\R.
\end{equation}
Such a formal diffeomorphism~$G$ is then called a ``formal geometric normal form'' for~$F$.
%
%The inverse image by $\Phi$ of the foliation by circles will be
%called a formal invariant foliation of $F$.
%
\end{definition}

Notice that, as particular cases of formal geometric normal
  forms, we have formal normal forms in the usual sense, \ie formal diffeomorphisms
\begin{equation}
  \zeta \mapsto
  N(\zeta)=\lambda \zeta\Bigl(1+\sum_{k\ge 1}c_k
  |\zeta|^{2k}\Bigr)
  \qquad (c_k\in\C)
\end{equation}
that  commute with the group of rotations.
Since~$\om$ is irrational, any $F\in\Fom$ admits formal normalizations, \ie
formal tangent-to-identity diffeomorphisms~$\Phi$ such that $\Phi\circ
F\circ \Phi^{-1}$ is a formal normal form,
and these are formal geometric normalizations of~$F$ as well.

% Since $\om$ is irrational, the existence of a formal geometric
% normalization is granted by the existence of a formal conjugacy to a
% normal form~$N$ in the usual sense, \ie to a formal diffeomorphism

Assuming~$F$ analytic, we will be interested in questions of convergence.\footnote{\label{footnCV} As
  usual in the theory of two-variable power series or two-variable
  analytic functions,
  we say that a series $\varphi(z)=\sum a_{j,k} z^j \bar z^k$ is convergent if
  there exist $\rho_1,\rho_2>0$ such that
  \begin{equation} \label{eqACV}
  \sum |a_{j,k}| \rho_1^j
  \rho_2^k < \infty.
  \end{equation}
  The domain of convergence~$\gD$ of $\varphi(z,w)=\sum
  a_{j,k} z^j w^k$ is then defined as the set of all $(z,w)\in\C^2$ such that
  there exist $\rho_1>|z]$ and $\rho_2>|w|$ such that~\eqref{eqACV}
  holds,
  and the sum of the series defines an analytic function of two variables in~$\gD$.
  By restriction to $\{w=\bar z\}$, we get a real-analytic function
  % of one variable
  in $\{z\in\C\mid (z,\bar z)\in\gD\}$
  (complex-valued in general, real-valued if $a_{j,k}=\bar a_{k,j}$
  for all $j,k$).
  }
The question of the existence of an analytic conjugacy to a
normal form in the usual sense has been well studied for area-preserving maps,
in which case it is natural to require that the formal normalization
be area-preserving too---the formal normal form is then unique (and called Birkhoff normal form).
For example, E.~Zehnder \cite{Z} and C.~Genecand \cite{C} proved that
a generic local analytic area-preserving map is not integrable
and hence admits no convergent area-preserving normalization.
In the parallel realm of Hamiltonian flows,
C.~L.~Siegel \cite{Si} proved that a generic analytic Hamiltonian has
no convergent normalization.
R.~P\'erez-Marco \cite{PM2} showed that, for a fixed
quadratic part, either any Hamiltonian has a convergent normal form,
%\blue{(resp.\ normalization)},
or a generic Hamiltonian has a divergent
Birkhoff normal form. % \blue{(resp.\ no convergent normalization)}.
% \David{[Suppress the normalization statement that I've put in blue?]}
%
Recently, R. Krikorian \cite{K} showed which alternative holds: for a fixed quadratic part, a generic
Hamiltonian has a divergent Birkhoff normal form.

%\smallskip
For general analytic $F\in\Fom$, not necessarily area-preserving, normal forms and
{\it a fortiori} geometric normal forms are not unique, with the only
exception of the case $F(\ze)=\la\ze$ (this follows from \cite[Lemma~4]{CSSW}).
This complicates the study of the convergence of
normalizations and normal forms, as well as that of
geometric normalizations and geometric normal forms.

One can see that a formal geometric
  normal form, resp.\ a formal normal form, is any formal diffeomorphism of the form
  %
%  \David{}
  %
\begin{equation}   \label{eqpolarG}
    G(\ze) = \lambda \ze(1+f(|\ze|^2)e^{2\pi i \beta(\ze)}
  \end{equation}
resp.
\begin{equation}   \label{eqpolarN}
  N(\zeta)=\lambda \zeta(1+f(|\zeta|^2))e^{2\pi i n(|\zeta|^2)}
  \end{equation}
where $f(R)$ and~$n(R)$ are real formal series in one indeterminate
and the formal series $\beta(\zeta) = \sum_{j+k\ge1} \beta_{j,k}\zeta^j\bar \zeta^k$ is real
in the sense that $\beta_{j,k}=\ov{\beta_{k,j}}$
(thinking of polar coordinates, these facts are obvious when~$G$
  or~$N$ are convergent; they can be proved in all generality by applying
  Lemma~\ref{lemformG} below to $\la\ii G$ and $\la\ii N$).
Thus, the convergence problem for the normalizations of a given analytic
  $F\in\Fom$ naturally splits into two independent problems, a radial
  one and an angular one:

\begin{enumerate}[(i)]
\item % {\it i)} 
  (radial problem)  existence of a convergent % an analytic
  geometric normalization of~$F$, \ie
  an analytic conjugacy to a local diffeomorphism~$G$ that preserves the
  foliation by circles centered at 0;
\item % {\it ii)}
  (angular problem) existence of an analytic conjugacy to a normal
  form~$N$ once the existence of an analytic geometric
  normalization is taken for granted.
 \end{enumerate}

% \smallskip
% \noindent In other words, we split the normalization problem into a radial and an angular one.
% \smallskip
 
The article \cite{CSSW} addressed problem~(ii): it considered
  those $F\in\Fom$ that are of the form~\eqref{eqpolarG}
  with~$f$ and~$\beta$ convergent, \ie analytic geometric normal
  forms,
  and showed the divergence of all formal normalizations
  %
  % the non-existence of an analytic normalization
  %
  for generic~$f$ and~$\beta$ when~$\omega$ is fixed irrational,
  as well as for generic~$\beta$ when~$f$ is fixed and~$\omega$ is fixed non-Bruno. 
%
% gave a criterion insuring the non-existence of an analytic conjugacy in case $\omega$ is not a Bruno number.
% \smallskip
 
In this paper, we address problem~(i) and give generic results of
divergence for all formal geometric normalizations of an
analytic $F\in\Fom$.
We will also see that the formal geometric normalizations define in a
natural way a unique \emph{invariant formal foliation}.
The idea is that, when we have a convergent geometric
normalization~$\Phi$,
equation~\eqref{eqcharacterisG} 
says that the foliation by circles centered
at~$0$ is invariant by the geometric normal form~$G$:
\begin{center}
  $G$ maps the circle $\big\{ |\ze|^2 = R \big\}$
  to the circle $\big\{ |\ze|^2 = \Ga(R) \big\}$
  \end{center}
  ($\Ga$ is automatically convergent in that case),
thus the levels of $L\defeq |\Phi|^2$ define a
singular foliation %(formal because $|\Phi|^2$ is only a formal series,
(singular because of the exceptional level $\{0\}$) that is invariant
by~$F$:
\begin{equation}   \label{eqLGa}
\text{$F$ maps the level $\{ L(z) = R \}$ to the level $\{ L(z) = \Ga(R)
  \}$.}
  \end{equation}
  When $\Phi$ is not assumed to be convergent,
$L=|\Phi|^2$ belongs to the affine space
  \begin{equation}   \label{eqdefgL}
    \gL \defeq \Big\{\, L(z) = z\bar z + \sum_{r+s\ge3} L_{rs} z^r
    \bar z^s \mid
    L_{rs} \in\C \;\text{with}\; L_{rs}=\ov{L_{sr}}\;\text{for
      every $(r,s)$} \,\Big\}
    \end{equation}
  and $\Ga$ belongs to the group~$\gG$ of formal tangent-to-identity 
 and $\mathbb{R}$-preserving diffeomorphisms, 
  \begin{equation}  \label{eqdefgG}
    \gG \defeq \Big\{\, \Ga(R) = R + \sum_{n\ge2}\Ga_n R^n \mid
    \Ga_n \in\R \,\Big\},
    \end{equation}
  but in that case too the
  ``formal singular foliation defined by~$L$'' 
%  and equation~\eqref{eqLGa} 
will be given a natural interpretation
as the class of $L$ in the quotient  % \Alain{}\marginlabel{}
$\gG\backslash \gL$ of $\gL$ by the left action of $\gG$.
  In general, $\Phi$, $L$ and~$\Ga$ are not unique, however
  we will see that
  the invariant formal singular foliation defined by the levels of~$L$ is unique
  in the sense that changing~$L$ simply amounts to relabelling the
  levels.
  We will also see that the conjugacy class of~$\Ga$ in~$\gG$
is uniquely determined, which motivates the following definition:
\begin{definition}   \label{defformcons}
  We call ``formally conservative'' an $F\in\Fom$ for which $\Ga=\ID$.
\end{definition}

% \red{We will show that, despite
%   of the non-uniqueness of normalizations and normal forms,
%   % for non area-preserving mappings,
%   one can define a natural unique ``formal
%   invariant foliation" which is given by the equivalent class of the
%   square norm $L=|\Phi|^2$ (section 2) of geometric normalizations ,
%   which encodes the radial part of geometric normalizations.}

Our main result is that, generically, one cannot find any convergent
geometric normalization.
%
% \David{}\marginlabel{cf. \cite[p.3]{PM2}}
%
In order to speak about generic properties, for any $\De>0$ we consider
\begin{equation}
\Fomcv \defeq \{\, \text{$F\in\Fom$ analytic with a complex
    extension~$\hat F$ holomorphic in~$B_\De$} \,\} 
\end{equation}
where $B_\De\defeq \{\, (z,w)\in\C^2 \mid |z|<\De, \; |w|<\De \,\}$
  % (one could take any open  ball of~$\C^2$ centred at~$(0,0)$ as well)
  %
(with reference to~\eqref{notahatF} for the notation~$\hat F$);
  we endow this space with the compact-open topology
  and call generic a property that is true on a dense
  $G_{\delta}$ subset of~$\Fomcv$.
%
% set for the compact-open topology on analytic functions.
% %
% Let 
% %
% \begin{equation}
%   %
% \Fomcv \defeq \{\, \text{analytic $F\in\Fom$} \,\}.
%   %
% \end{equation}
%
The main theorem can be stated as follows:

\goodbreak

%%%%%%%%%%%%%%%%%%%%%%%%
%%%%%%%%%%%%%%%%%%%%%%%%
\begin{ourthm}  \label{main}
  Let $\De>0$. \medskip
  
  \noindent (i) For generic $F\in\Fomcv$, all formal geometric
  normalizations of $F$ are divergent.
  % solutions to \eqref{conj} have $L$ divergent;
  %
  \medskip % \vspace{0.95ex}

%%%%%%%%%%%%%%%%%%%%%%%%
\noindent (i') The same conclusion holds for generic
$F\in\Fomcv\cap\{\Gamma\neq\ID\}$.

  \medskip % \vspace{0.95ex}

%%%%%%%%%%%%%%%%%%%%%%%%
\noindent (ii) Given an integer $N\ge 2$ and a finite jet
$J(z)=\lambda z+\sum_{2\le j+k\le N}J_{jk}z^j\bar z^k$,
the same conclusion holds for generic $F\in\FomJcv$ with
\begin{equation}
\FomJcv \defeq \{F\in\Fomcv \mid F=J+O(|z|^{N+1})\}
\end{equation}
under the assumption that $\omega$ satisfies
%
% belongs to the dense subset of the unit interval defined by 
a super-Liouville type condition, namely~\eqref{eqdefsuperL} below.
\smallskip % \medskip %   \vspace{0.95ex}

%%%%%%%%%%%%%%%%%%%%%%%%
\noindent (iii) Given % an integer $N\ge 2$ and
$J(z)=\lambda z+\sum_{2\le k\le N} J_{k} z^k$
(a \emph{holomorphic} $N$-jet), there exists $\De_0>0$ such that
the same conclusion holds for generic $F\in\FomJcv$ if $\De\le\De_0$ and
under the weaker classical assumption that~$\omega$ is a non-Bruno number,
namely~\eqref{eqdefnonB} below.
\medskip %   \vspace{0.95ex}

%%%%%%%%%%%%%%%%%%%%%%%%
  \noindent (iv) Given a polynomial $J(z)=\lambda
  z+\sum_{2\le j+k\le N}J_{jk}z^j\bar z^k$ that is the $N$-jet of an
  area-preserving map,
  %
  % \blue{having non-trivial Birkhoff normal form at
  %  order~$N$ and belonging to~$\Fomcvp$ for a $\Dep>\De$}
  the same conclusion holds for generic $F\in\FomJcv$
  without any other assumption than $\om\in\R-\Q$.
\end{ourthm}
%%%%%%%%%%%%%%%%%%%%%%%%
%%%%%%%%%%%%%%%%%%%%%%%%

Notice that all area-preserving maps are formally
conservative in the sense of Definition~\ref{defformcons},
and so are holomorphic maps---see Proposition~\ref{integ} below.
%
% Notice that, clearly, all area-preserving maps are formally
% conservative in the sense of Definition~\ref{defformcons}.
% %
% This is also the case of holomorphic maps.

\medskip

There are similar results when our dynamical system belongs to
\begin{equation}
  \Fomoddcv \defeq \{F\in \Fomcv \mid F(-z)=-F(z)\},
\end{equation}
\ie when the dynamics is symmetric with respect to the origin
of~$\R^2$, a feature that is often encountered in practice, however
that feature will affect the proof.

\begin{ourthm}   \label{thmB}
  The statements (i), (i'), (ii), (iii) and~(iv) of Theorem~\ref{main}
  remain valid if one replaces~$\Fomcv$ with $\Fomoddcv$,
  if one takes~$J$ odd in the analogue of~(ii)--(iv),
  and if, in the analogue of~(ii), one replaces the super-Liouville
  condition~\eqref{eqdefsuperL} for~$\om$ by the same condition
  on~$2\om$.
%``even super-Liouville'' condition~\eqref{eqdefsuperLev} below.
  %
\end{ourthm}
%%%%%%%%%%%%%%%%%%%%%%%%
%%%%%%%%%%%%%%%%%%%%%%%%

In fact, in all the cases covered by Theorems~\ref{main}
and~\ref{thmB}, not only is any geometric normalization~$\Phi$
divergent, but also $L=|\Phi|^2$ is, as will be seen in
Section~\ref{secpfthmmain}. % \marginlabel{[Alain: ref. made precise]}
%Section~\ref{secgendiv} \David{[give a more precise ref?]}.
%
We can thus say that \emph{the invariant formal foliation of~$F$ is
generically divergent.}

%%%%%%%%%%%%%%%%%%%%%%%%
The aforementioned super-Liouville type condition on~$\omega$ is
\begin{equation}   \label{eqdefsuperL}
\text{there exist infinitely many }\; k\in\Zp \;\text{such that}\;
  \bigl|\lambda^k-1\bigr|^{-1}\ge k!.
\end{equation}
We need this condition in Point~(ii) to mimic Siegel-Moser's
construction of an example of a non linearizable homomorphic map $F(z)=\lambda z+O(|z|^2)$.
Condition~\eqref{eqdefsuperL} follows from the arithmetic
condition~(11) in Chapter III, \S 25 of \cite{SM} (in which $k!^2$ is used instead of~$k!$,
actually any positive power of~$k!$ would do).
It is proved in \cite[pp.~189--190]{SM} that \emph{the set of such $\om$'s is
dense in~$\R$.} % the unit interval.
%
% The ``even'' variant
% %
% \begin{equation}   \label{eqdefsuperLev}
%   %
% \text{there exist infinitely many even}\; k\in\Zp \;\text{such that}\;
%   \bigl|\lambda^k-1\bigr|^{-1}\ge k!
%   %
% \end{equation}
%   %
% will be used in a similar construction in $\Fomodd$ with prescribed $N$-jet~$J$.

The non-Bruno condition used in Point~(iii) is
\begin{equation}   \label{eqdefnonB}
\sum_{k\ge0} \frac{\ln q_{k+1}}{q_k} = +\infty,
  \end{equation}
where we denote by $(p_k/q_k)_{k\ge0}$ the sequence of the
convergents of the continued fraction expansion of~$\om$.
It is well-known that \emph{non-Bruno numbers form a dense subset of~$\R$.}

Actually, the non-Bruno set contains the super-Liouville set.
Indeed, we have $|\la^q-1| \ge 4 \operatorname{dist}(q\om,\Z)$ for all
$q\in\Z$ (\cite[p.~189]{SM}), hence~\eqref{eqdefsuperL} implies that
for infinitely many $q\in\Zp$ one can find $p\in\Z$ such that
$|q\om-p|\le \frac{1}{4q!}$.
For any such large enough~$q$, take $k\ge1$ such that $q_k\le q <
q_{k+1}$;
the law of the best approximation (\cite[Theorem~182 and its
proof]{HW}) yields $|q\om-p| \ge |q_k\om-p_k|$.
Since we always have $2q_{k+1}>|q_k\om-p_k|^{-1}$, we get
$q_{k+1} > \frac12 |q\om-p|\ii \ge 2 q! \ge 2 q_k!$ for infinitely
many $k$'s,
whence~\eqref{eqdefnonB}.
% \David{}

%%%%%%%%%%%%%%%%%%%%%%%%
\subsubsection*{Strategy of proof and plan of the article.}

\begin{enumerate}[--]
\item
In Section~\ref{secformfol} we describe, for a given $F\in\Fom$, the
set of all formal geometric normalizations~$\Phi$ and show how to
reduce the problem to the question of finding a pair $(L,\Ga)$ as
alluded to in \eqref{eqLGa}--\eqref{eqdefgG}.
We also formalize the notion of invariant formal foliation.
\item
  In Section~\ref{secselect}, we use the unique invariant formal
  foliation to define a coordinate-dependent formal
  involution~$\tau_F$ and then to single out a ``balanced''
  % geometric normalization,
  formal series~$L_F$,
  whose divergence implies the divergence of all geometric
  normalizations.
\item
In Section~\ref{secgendiv}, we prove the generic divergence of the
balanced formal series~$L_F$ % geometric normalization
by exploiting the so-called
Ilyashenko--P\'erez-Marco alternative
and thus prove Theorem~\ref{main}.
In fact, we even prove the generic divergence of the invariant formal 
foliation itself (\ie of any~$L$).
We can also prove the generic divergence of the formal
involution~$\tau_F$, with an extra arithmetic condition in the case of
a prescribed $N$-jet~$J$: we then require~$\omega$ to fulfill an ``odd
super-Liouville'' assumption---see Theorem~\ref{thmtaud}.
\item
  In Section~\ref{secodddiv}, we consider the case of dynamical
  systems $F\in\Fomoddcv$.
Such symmetric system necessarily have $L(-z)=L(z)$ and $\tau_F(z)=-z$, and some
modifications are thus needed in the proofs.
\item
Section~\ref{secquest} raises questions for future research.
\item
The paper ends with four technical appendices and a list of minor
errata to the article \cite{CSSW}, to which the present paper is a
kind of prequel.
\end{enumerate}

%%%%%%%%%%%%%%%%%%%%%%%%%%%%%%%%%%%%%%%%%%%%%%%%
%%%%%%%%%%%%%%%%%%%%%%%%%%%%%%%%%%%%%%%%%%%%%%%%

\section{Formal geometric normalizations and formal invariant
  foliations}
\label{secformfol}

%%%%%%%%%%%%%%%%%%%%%%%%%%%%%%%%%%%%%%%%%%%%%%%%
%%%%%%%%%%%%%%%%%%%%%%%%%%%%%%%%%%%%%%%%%%%%%%%%

\subsection{Reduction to $F$-admissible pairs}

%%%%%%%%%%%%%%%%%%%%%%%%%%%%%%%%%%%%%%%%%%%%%%%%
%%%%%%%%%%%%%%%%%%%%%%%%%%%%%%%%%%%%%%%%%%%%%%%%

It is convenient to introduce a notation for the 
square modulus function: % square norm operator
%
% \David{}
%
\begin{equation}    \label{eqdefnu}
\nu(z)\defeq z \bar z. % |z|^2.
\end{equation}
%
% that we shall use throughout the whole paper.
%
Let us start with the easy
(obvious in case of convergence) %  \Alain{}\marginlabel{}
\begin{lemma}   \label{lemformG}
  Suppose that
  $\Psi(z) = z+\sum_{j+k\ge2}
  \psi_{j,k}z^j\bar z^k$
  is a formal tangent-to-identity diffeomorphism such that
  \begin{equation}   \label{eqnuPsiGanu}
    \nu\circ \Psi = \Ga \circ \nu
  \end{equation}
  for some univariate formal series $\Ga(R) \in \C[[R]]$.
  Then necessarily~$\Ga$ belongs to the group~$\gG$ defined by
  formula~\eqref{eqdefgG} and~$\Psi$ is of the form
  \begin{equation}   \label{eqformPsi}
    \Psi(z) = z \big( 1 + f(z \bar z) \big) e^{2\pi i\be(z)}
  \end{equation}
  with a univariate real formal series without constant term,
  $f(R)\in R\,\R[[R]]$, determined by
  \begin{equation}   \label{eqformGa}
    \Ga(R) = R \big( 1 + f(R) \big)^{\!2}
  \end{equation}
  and a uniquely determined formal series
  $\beta(z) = \sum_{j+k\ge1} \beta_{j,k}z^j\bar z^k\in\C[[z,\bar z]]$
  that is real in the sense that $\beta_{j,k}=\ov{\beta_{k,j}}$ for
  all $(j,k)$, \ie $\be = \ti\be$ if we use
  notation~\eqref{notatifzw}.  % \David{}
  \smallskip
  
  Moreover, if $\Psi(z)$ is convergent, then $f$ and~$\be$ are convergent.
  \end{lemma}
Note that, in general, the left-hand side of~\eqref{eqnuPsiGanu} is a formal series in
$(z,\bar z)$.

\begin{proof}
  Expanding in powers of~$z$, we can write
\[
  \Psi(z) = \psi_0(\bar z) + z\, \psi_*(z,\bar z), \qquad
  \ov{\Psi(z)} = \chi_0(\bar z) + z\, \chi_*(z,\bar z)
\]
with $\psi_0(\bar z), \chi_0(\bar z)\in \C[[\bar z]]$ and
$\psi_*(z,\bar z), \chi_*(z,\bar z)\in \C[[z,\bar z]]$.
On the other hand, $\Ga(R) = \Ga_0 + R\,\Ga_*(R)$ with $\Ga_0\in\C$
and $\Ga_*(R) \in \C[[R]]$.
From~\eqref{eqnuPsiGanu} we get
\[
  \Psi\ov\Psi=\Ga_0 + z\bar z\,\Ga_*(z\bar z).
\]
In particular $\psi_0(\bar z)\chi_0(\bar z)=\Ga_0$.
The absence of constant term in~$\psi_0(\bar z)$ implies that
$\Ga_0=0$ and $\psi_0(\bar z)\chi_0(\bar z)=0$;
but the linear part of~$\ov{\Psi(z)}$ is~$\bar z$, hence $\chi_0(\bar
z)=\bar z+O(\bar z^2)$ is nonzero and we must have $\psi_0(\bar z)=0$.

Thus $\Psi(z) =z\, \psi_*(z,\bar z)$ with $\psi_*(z,\bar z)=1+O(z,\bar
z)$ satisfying
\[
  \psi_*\ov{\psi_*}=\Ga_*\circ\nu.
\]
Clearly, $\Ga_*(R) = 1+O(R)\in \R[[R]]$.
Substituting $\Ga_*-1$ in the formal series $\sqrt{1+T} = 1 + \frac12 T
+ \cdots$,
we get a real formal series $\sqrt{1+(\Ga_*-1)} = 1+f(R)$ such that
$\psi_*\ov{\psi_*}=( 1 + f\circ\nu)^2$,
and substituting $\frac{\psi_*}{1+f\circ\nu}-1$ in the formal series
$\log(1+T) = T - \frac12 T^2 + \cdots$,
we get a formal series
\[
  \be(z,\bar z) \defeq \frac{1}{2\pi i}\log\!\Big(1 +
  \Big(\frac{\psi_*}{1+f\circ\nu}-1\Big)\Big) \in\C[[z,\bar z]]
\]
without constant term
such that $\psi_* = (1+f\circ\nu) e^{2\pi i \be}$.
We have thus obtained~\eqref{eqformPsi}--\eqref{eqformGa}.
Using notations \eqref{notatifzw}--\eqref{eqnormfsq} we see that the equation
$1 = \frac{\psi_* \ti\psi_*}{(1+f\circ\nu)^2} = e^{2\pi i(\be-\ti \be)}$
implies $\be=\ti \be$.

If~$\Psi(z)$ is convergent, then so is $\psi_*$, and thus also $\Ga_*$,
$f$ and~$\be$.
\end{proof}

The condition~\eqref{eqcharacterisG} that characterises formal geometric
normal forms~$G$ can be rewritten $\nu\circ G=\Ga\circ \nu$;
when applied to~$\la\ii G$, Lemma~\ref{lemformG} thus justifies the claim made earlier that any
geometric normal form is of the form~\eqref{eqpolarG}.

Now, we see that, given $F\in\Fom$,
\begin{multline*} %   \label{equivGN}
  \text{$\Phi$ is a formal geometric normalization}
  \;\Longleftrightarrow\; %\\[1ex]
  \exists \Ga\in\gG \enspace \text{such that}\;
  \nu\circ\Phi\circ F = \Ga\circ \nu \circ \Phi \\[1ex]
  \;\Longleftrightarrow\; 
  \exists (L,\Ga)\in\gL\times\gG \enspace \text{such that}\;
  \left\{\begin{aligned}
      & L = \nu\circ \Phi \\[1ex]
      & L\circ F = \Ga \circ L
    \end{aligned}\right.
\end{multline*}
%
%
%    (recall that~$\gL$ was defined in formula~\eqref{eqdefgL} of Section~\ref{secdecoupl}).
    (recall that~$\gL$ and~$\gG$ were defined in \eqref{eqdefgL}--\eqref{eqdefgG}).
%    Section~\ref{secdecoupl},
    %
%with appropriate formal series~$L$ and~$\Ga$.
%

\begin{figure}[h]
  \centering
  \includegraphics[scale=0.65]{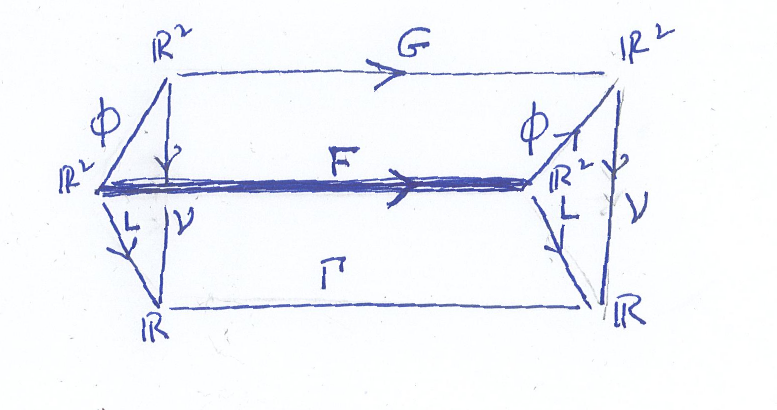}
  \vspace{-3ex}
  
  Commutative Diagram~1: Geometric normal form
  
  \end{figure}
% \marginlabel{ \Qiaoling{Figure move here}}

\begin{definition}\label{F-ad}
    We call ``$F$-admissible pair''  any $(L,\Ga)\in\gL\times\gG$
    such that
    \begin{equation}    \label{conj}
      L\circ F=\Ga\circ L
    \end{equation}
 (see Commutative Diagram~1).
    We say that~ $L$ is ``$F$-admissible'', or~$L$ is admissible for~$F$,
    if there exists~$\Ga$ such that $(\Ga, L)$ is an $F$-admissible pair.
  \end{definition}

% \Qiaoling{{\it \noindent
%
    The following two lemmas imply that the existence of a geometric
    normalization of~$F$ is equivalent to the existence of an
    $F$-admissible pair $(\Gamma,L)$.
    
\begin{lemma}\label{Phi}
  Given $L\in\gL$,
\medskip
  
  \noindent (i) there always exist formal tangent-to-identity
  diffeomorphisms~$\Phi$ such that $L=\nu\circ\Phi$;
\medskip
  
  \noindent (ii) two such formal diffeomorphisms~$\Phi$
  differ from one another by the multiplication by
  $e^{2\pi i\be(z)}$, where
$\be(z)=\sum_{j+k\ge 1}\be_{j,k}z^j\bar z^k$ is a real formal
  series (\ie $\be = \ti\be$ if we use
  notation~\eqref{notatifzw}); % $\be_{j,k}=\ov{\be_{k,j}}$);
\medskip
  
  \noindent (iii)  if~$L$ is convergent, then one can find a convergent
  solution~$\Phi$, and all other convergent solutions are obtained by
  multiplying it by the exponential of an arbitrary imaginary convergent
  $2\pi i\be$.
\end{lemma}

\begin{proof}
  (i)
The existence of $\Phi$ is nothing but a formal version of the classical {\it Morse lemma} which, in this simple case, can be proved directly as follows.
Using the notations $z=x+iy$ and (slight abuse) $L(x,y)=L(z)$, 
we have the following identity in the space of formal (resp.\
convergent) series in two
variables:
%
% \Alain{}\marginlabel{\red{Peut-être changer la notation $\tau$ ?}}
% \David{}\marginlabel{\red{J'essaie avec $t$...}}
%
\begin{equation}   \label{eqTAF}
L(x,y)=x\int_0^1\frac{\partial L}{\partial x}( t  x, t  y)d t +
y\int_0^1\frac{\partial L}{\partial y}( t  x, t  y)d t .
\end{equation}
Applying
this successively to $L(x,y),\, \frac{\partial L}{\partial x}$ and
$\frac{\partial L}{\partial y}$, and recalling that these three formal
series vanish at $(0,0)$, we get three real formal series in two % real
variables, % $x,y$,
which we write
$A(x,y),B(x,y),C(x,y)\in \mathbb{R}[[x,y]]$, such that
\begin{equation*}
\left\{
\begin{split}
&L(x,y)=x^2A(x,y)+2xy B(x,y)+y^2C(x,y),\\[1ex]
& A(0,0) =C(0,0)=1,\enspace B(0,0)=0.
\end{split}
\right.
\end{equation*}
The formal series~$A$ and~$C$ have multiplicative inverses, and the
formal series $A-1$, $C-1$ and $\frac{B^2}{AC}$ can be substituted in
the formal series $\sqrt{1+T} = 1 + \frac12 T + \cdots$.
% Considering  $\sqrt{1+x}$ as a formal series,
We thus define $a \defeq \sqrt{1+(A-1)}$,
$c \defeq\sqrt{1+(C-1)}$
and $d \defeq \sqrt{1 - \frac{B^2}{AC}}$; these are elements of $\R[[x,y]]$ such
that
$a^2=A$, $c^2=C$, $d^2=1 - \frac{B^2}{AC}$ and $a(0,0) =c(0,0)=d(0,0)=1$.
Defining~$\Phi$ by the formula 
\begin{equation}\label{square root}
  \Phi(x+iy) \defeq \frac{x A(z)+y B(z)}{a(z)} + i\, y \, c(z) d(z),
  % x a(z)+y\frac{B(z)}{a(z)} +i\, y c(z)\sqrt{1-\frac{B^2(z)}{a^2(z)c^2(z)}} % \; \cdot 
\end{equation}
one can check that $\Phi$ is a formal tangent-to-identity
diffeomorphism and $\nu\circ\Phi=L$.

\medskip

(ii)
%
% \David{}
%
% \Qiaoling{ Now suppose $\Phi_1(z)$ and $\Phi_2(z)$ are such two formal
%   series such that $|\Phi_1|^2=|\Phi_2|^2$. Since $\Phi_i$, $i=1,2$
%   are tangent to identity, they have multiplicative inverse \emph{[David:
%   NO! I will correct this, using composition inverse...]}. Let $\alpha(z):=\Phi_1(z)\Phi_2^{-1}(z)$, then $\alpha(z)=1/\overline{\alpha(z)}$, hence $|\alpha(z)|=1$, that is, $\Phi_1$ and $\Phi_2$ differ by an multiplication by $e^{2\pi i \beta(z)}$ with $\beta(z)$ real.}
%
Suppose that~$\Phi_1(z)$ and~$\Phi_2(z)$ are two formal
tangent-to-identity diffeomorphisms such that $\nu\circ\Phi_2=\nu\circ\Phi_1$
and let $\Psi \defeq \Phi_2 \circ \Phi_1\ii$:
this is a formal tangent-to-identity diffeomorphism satisfying
$\Psi(z) \ov{\Psi(z)} = z \bar z$,
Lemma~\ref{lemformG} thus implies that
$\Psi(z) = z\,\exp\!\big(2\pi i B(z,\bar z)\big)$
with $\ti B= B$,
hence $\Phi_2 = \Phi_1 \exp(2\pi i B\circ\hat\Phi_1)$ with notation~\eqref{notahatF}.
We thus get the desired conclusion with
$\be \defeq B\circ \hat\Phi_1$,
since $\ti\be = \widetilde{B\circ\hat\Phi_1} = \ti B \circ \hat\Phi_1 = \be$.

\medskip

(iii)
If $L$ is convergent, then the explicit formulas for $A(z)$, $B(z)$
and $C(z)$ derived from~\eqref{eqTAF} show that they are convergent,
% since the coefficients of each of them are parts of those of $L$.
as well as $a$, $c$ and~$d$,
hence the formal diffeomorphism~$\Phi$ constructed in~\eqref{square
  root} is convergent
and the conclusion follows.
\end{proof}

\begin{lemma}\label{exist}
  Given $L$ admissible for~$F\in \Fom$,
  %
  % Given an admissible pair $(L,\Gamma)$ for~$F\in \Fom$,
  %
  any formal tangent-to-identity diffeomorphisms~$\Phi$ such that
  $L=\nu\circ\Phi$ is a formal geometric normalization of~$F$.
\end{lemma}

\begin{proof}
 Choose $\Ga\in\gG$ such that $(L,\Ga)$ is $F$-admissible. According
    to Commutative Diagram~1, $G= \Phi\circ F\circ \Phi^{-1}$ 
    satisfies $\nu\circ G=\Gamma\circ\nu$.
%
  % \red{It follows directly from the above commutative diagram, where
  %   we recall that $\nu(z)=|z|^2$:
  %   $ L\circ F=\nu\circ\Phi\circ F=\Gamma\circ L=\Gamma\circ
  %   \nu\circ\Phi$, let $G= \Phi\circ F\circ \Phi^{-1}$, hence
  %   $\nu\circ G=\nu\circ\Phi\circ F\circ \Phi^{-1}=\Gamma\circ\nu$,}
  % \David{[English and presentation to be improved]} that is
  % $|G(\zeta)|^2=\Gamma(|\zeta|^2)$.
  %
\end{proof}

%\smallskip

Therefore, the formal geometric normalizations~$\Phi$ of any
$F\in\Fom$ can be retrieved from its admissible formal series~$L$.

%%%%%%%%%%%%%%%%%%%%%%%%%%%%%%%%%%%%%%%%%%%%%%%%
%%%%%%%%%%%%%%%%%%%%%%%%%%%%%%%%%%%%%%%%%%%%%%%%

\subsection{The unique invariant formal foliation $\cF_F$}   \label{secforminvfol}
%%%%%%%%%%%%%%%%%%%%%%%%%%%%%%%%%%%%%%%%%%%%%%%%
%%%%%%%%%%%%%%%%%%%%%%%%%%%%%%%%%%%%%%%%%%%%%%%%

We now show that, although formal geometric normalizations are never
unique, one can define a unique ``invariant formal foliation''.
Observe that the group~$\gG$ acts from the left on the set~$\gL$:
\begin{equation}    \label{eqleftaction}
(g,L) \in \gG \times \gL \mapsto g\circ L \in \gL.
\end{equation}
\begin{definition}\label{deffol}
  We call \emph{formal foliation} any element of the quotient set
  $\gG\backslash \gL$.
  We use the notation
\begin{equation}    \label{eqleftquotient}
  L\in\gL \mapsto [L]=\{g\circ L \mid g\in \gG\} \in\gG\backslash \gL.
\end{equation}
for the canonical projection.
We say that a formal foliation is \emph{analytic} if at least one of
  its representatives is convergent (see footnote~\ref{footnCV}).
\end{definition}

Our motivation is that, if % by the analogy with the convergent case:
%
% suppose indeed that
%
$L\in\gL$ is convergent for $|z|$ small enough,
then we get an analytic singular foliation by considering the
family of the levels $L\ii(c)$, $c>0$ small,
but we want to identify this foliation with the one induced by $g\circ
L$ for any convergent $g\in\gG$,
since this amounts to a simple reparametrisation of the levels
($c\mapsto g\ii(c)$).

%\bigskip

\smallskip

\noindent\textit{Example:}
The standard foliation $\{|z|=c \mid c\in
\Rp\}$ by round circles is $[\nu]$, with~$\nu$ as in~\eqref{eqdefnu}.

\smallskip

Note that to prove that a formal foliation is divergent, \ie not
analytic, one must prove that \emph{all} its representatives are divergent.

\smallskip

The group~$\gF$ of all formal local diffeomorphisms of $(\R^2,0)$
whose linear part is a rotation
(of which~$\Fom$ is a subset)
acts from the right on~$\gL$, and thus also on
$\gG\backslash \gL$.
We shall use the notation
\[
  [L]\circ F \defeq [L\circ F]
  \quad \text{for $F\in\gF$ and $L\in\gL$.}
\]
If $[L]\circ F=[L]$, we say that the formal foliation~$[L]$ is
$F$-invariant, or that~$F$ preserves~$[L]$;
this precisely means that $L\circ F$ is of the form $g\circ L$ for
some $g\in\gG$.

\medskip

\noindent\textit{Example (case of the standard foliation $[\nu]$):}

\begin{itemize}
\item A formal diffeomorphism $G\in\Fom$   % \David{}
  %
  % of the form $\lambda z +\text{\textit{higher order terms in $(z,\bar z)$}}$
  %
  preserves~$[\nu]$ if and only if it is a formal
  geometric normal form.

\item According to Lemma~\ref{lemformG}, a formal tangent-to-identity
diffeomorphism~$\Phi$ preserves $[\nu]$ if and only if it is of the
form
\begin{equation}   \label{eqPhialbet}
\Phi(z)=z\big(1+\alpha(|z|^2)\big)e^{2\pi i\beta(z)}
  \end{equation}
with a real univariate formal series $\alpha(R)$ and a formal series
$\beta(z) = \sum_{j+k\ge1} \beta_{j,k}z^j\bar z^k$ that is real
in the sense that $\beta_{j,k}=\ov{\beta_{k,j}}$ for all $(j,k)$,
\ie $\be = \ti\be$ if we use notation~\eqref{notatifzw}.  % \David{}

\item The equality of two formal foliations can be described
  geometrically: for $L_1,L_2\in \gL$, $[L_1]=[L_2]$ if and only if,
  for any $\Phi_1, \Phi_2$ such that $\nu\circ \Phi_i=L_i$ as in Lemma~\ref{Phi}, 
  $\Phi_1\circ\Phi_2^{-1}$ preserves~$[\nu]$.
  % $i=1,2$.

\end{itemize}

The following lemma is a direct consequence of Definitions~\ref{F-ad}
and~\ref{deffol} and Lemma~\ref{exist}:
\begin{lemma}   \label{propequivstatements}
  Given $L\in\gL$ and $F\in \Fom$, the following statements are 
equivalent:
%
% \Alain{}\marginlabel{\red {I reordered this page and the following}}
%
\begin{alignat*}{2}
  & \text{(i)}\; && \text{$L$ is $F$-admissible}\\
  \ArrowBetweenLines
  & \text{(ii)}\;  && \text{the formal foliation $[L]$ is $F$-invariant}\\
  \ArrowBetweenLines
  & \text{(iii)} \; && \text{any formal tangent-to-identity diffeomorphism~$\Phi$ such that
      $L=\nu\circ \Phi$}\\
  &&& \text{is a formal geometric normalization.}
\end{alignat*}
%
% \begin{enumerate}
%   %
% \item $L$ is $F$-admissible;
% \item The formal foliation $[L]$ is $F$-invariant;
% \item  Any formal tangent-to-identity diffeomorphism~$\Phi$ such that
%   $L=\nu\circ \Phi$ is a formal geometric normalization.
% %
% \end{enumerate}
%
%Now, given $L\in\gL$ and $F\in \Fom$,
%
%\[
  %  \text{$L$ is $F$-admissible}
  %
%  \enspace\Longleftrightarrow\enspace
  %
%\text{the formal foliation $[L]$ is $F$-invariant.}
  %
% \]
%
% \smallskip
%
\end{lemma}

We can now show 
    \begin{proposition}[Uniqueness of the invariant foliation]\label{uniqueL}
      Any $F\in\Fom$ admits a unique invariant formal foliation
      $\cF_F\in \gG\backslash \gL$:
      \begin{align*}
        \cF_F &=[L] && \hspace{-5em}\text{for any $F$-admissible $L$}\\[1ex]
      &=[\nu]\circ \Phi &&\hspace{-5em}\text{for any formal geometric normalization $\Phi$ of $F$}.
      \end{align*}
      Moreover, the invariant formal foliation of~$F$ is intrinsic
      (natural), in the sense that if one conjugates~$F$ by
      %
      % an analytic
      a formal
      local change of coordinates $\Psi\in\gF$ (note that the linear
      part of~$F$ does not change), then $\cF_F$ changes accordingly:
  \[
    \cF_{\Psi\ii\circ F\circ \Psi} = \cF_F \circ \Psi.
    \]
    In particular, if $\cF_F$ % the invariant formal foliation of an analytic $F\in\Fomcv$
      is analytic, then so is $\cF_{\Psi\ii\circ F\circ \Psi}$ for any analytic
      $\Psi\in\gF$.
  \end{proposition}

   \begin{proof}
     Take any two $F$-admissible formal series $L_1$ and $L_2$.
     Write $L_i=\nu\circ\Phi_i$, $i=1,2$: by Lemma~\ref{exist}, $\Phi_1$ and $\Phi_2$ are formal geometric normalizations. We need show that 
     $[L_1]=[L_2]$, that is $[\nu\circ\Phi_1] = [\nu\circ\Phi_2]$, or
     $[\nu]\circ\Phi_1=[\nu]\circ\Phi_2 $ or 
     $[\nu]\circ\Phi_1\circ\Phi_2^{-1}=[\nu]$.
     Hence uniqueness of the invariant foliation will follow from
     Lemma~\ref{preserve} below.
   \end{proof}
   
  \begin{lemma}\label{preserve}
   Let $F\in\Fom$.
   For any two geometric normalizations, $\Phi_1$ and~$\Phi_2$, the composition
$\Psi\defeq\Phi_1\circ \Phi_2^{-1}$ preserves the foliation
$[\nu]$;
this means that there exists $g\in\gG$ as in Commutative Diagram~2.
Conversely, given a geometric normalization~$\Phi_2$ of~$F$ and any
formal diffeomorphism~$\Psi$ preserving the foliation $[\nu]$,
$\Phi_1\defeq\Psi\circ \Phi_2$ is also a geometric normalization of $F$.
\end{lemma}

\vspace{-2.5ex}

\begin{center}
% \begin{figure}[h]
%  \centering
  \includegraphics[scale=0.6]{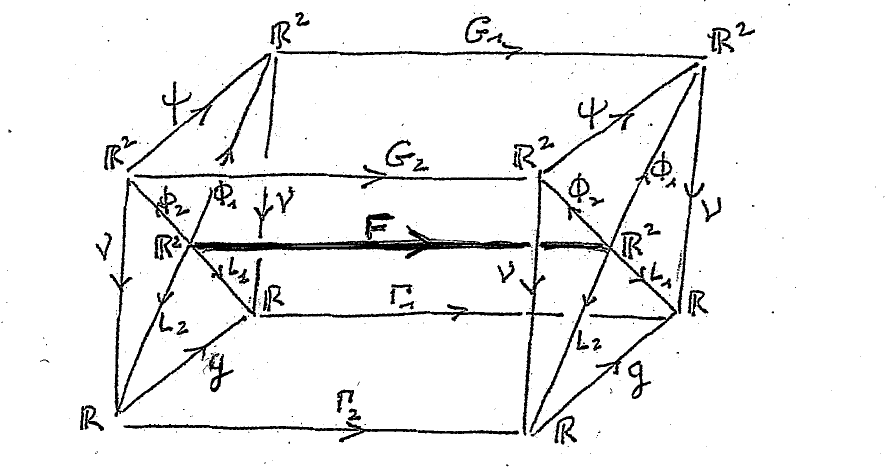}
  \vspace{-1ex}
  
  Commutative Diagram~2: Comparison between two geometric normalizations
  
  % \end{figure}
\end{center}

\begin{proof}[Proof of Lemma~\ref{preserve}]
  The second statement is obvious.  % \David{}
  For the first one, we consider the corresponding geometric normal forms $G_i\defeq\Phi_i\circ F\circ
\Phi_i^{-1}$, $i=1,2$. % preserve the foliation $[\nu]$
We can view~$\Psi$ as a conjugacy between~$G_2$ and~$G_1$.
The following Lemma allows us to replace~$G_1$ with a normal form~$N$
by means of a formal change of coordinate~$\Psi_1$ that preserves~$[\nu]$.
Now, in this new coordinate, $\Psi$ appears as a formal normalization
of~$G_2$, and thus, again by the same lemma, preserves~$[\nu]$
(or, if one prefers, $\Psi_2\defeq \Psi_1\circ\Psi$ satisfies
$\Psi_2\circ G_2 = N\circ\Psi_2$, thus $\Psi_2$ preserves~$[\nu]$ by
Lemma~\ref{cssw3},
and so does~$\Psi$).
\end{proof}

\begin{lemma}[\cite{CSSW}, Lemma 3 and Corollary 1] \label{cssw3}
Suppose that~$G$ is a formal geometric normal form.
  Then any formal normalization %~$\Psi$
  of~$G$ preserves~$[\nu]$.
 \end{lemma}

%%%%%%%%%%%%%%%%%%%%%%%%%%%%%%%
%%%%%%%%%%%%%%%%%%%%%%%%%%%%%%%

\begin{remark}
  Lemma~\ref{preserve} says that, for any two geometric normalizations
  of $F\in\Fom$,
  \[
    \Phi_1\circ \Phi_2^{-1}(z) = z(1+\alpha(|z|^2))e^{2\pi i \beta(z)}
  \]
  with $\alpha$ and~$\beta$ as in~\eqref{eqPhialbet}. This is to be compared with
Lemma~3 or Corollary~1 of \cite{CSSW}, % \blue{check ref},
  which says that, for any two formal normalizations
  of a geometric normal form,
  \[
    \Phi_1\circ \Phi_2^{-1}(z) = z\big(1+a(|z|^2)\big)e^{2\pi i b(|z|^2)}
  \]
  with suitable real univariate formal series $a(R)$ and $b(R)$.
  \end{remark}

    As a consequence of the uniqueness of the invariant foliation, we
    get a complete description of all admissible pairs for a given
    $F\in \Fom$:
\begin{corollary}\label{setF}
  Let $F\in \Fom$ and pick an $F$-admissible pair $(L_0,\Gamma_0)$.
  The set of all $F$-admissible pairs coincides with the orbit of
  $(L_0,\Gamma_0)$ under the natural left action of~$\gG$ on
  $\gL\times\gG$, \ie
  \[
    \{\, (g\circ L_0,g\circ\Ga_0\circ g\ii) \mid
    g\in\gG \,\},
  \]
  and the set of all $F$-admissible formal series is $\{\, g\circ L_0 \mid
    g\in\gG \,\}$ (\ie it coincides with $[L_0]=\cF_F$).
  \end{corollary}

  \begin{proof}
    Putting together Proposition~\ref{uniqueL} and
    Lemma~\ref{propequivstatements}, we see that the unique
    $F$-invariant formal foliation is $\cF_F=[L_0]$ and the set of
    $F$-admissible formal series is the orbit of~$L_0$ under~$\gG$.
   
    Now, if $L=g\circ L_0$ with $g\in\gG$, it follows directly from the
  equation $L_0\circ F=\Gamma_0\circ L_0$ that $(g\circ L_0,
  g\circ \Gamma_0\circ g^{-1})$ is $F$-admissible and, according
  to Lemma~\ref{lemLdeterminesGa}, it is the only $F$-admissible pair whose
  first component is~$L$.
\end{proof}

Finally, given an $F$-admissible~$L$, the following lemma, which
describes the unique~$\Gamma$ such that the pair $(L,\Gamma)$ is
admissible, shows in particular its analyticity when~$L$ is analytic:
\begin{lemma} \label{lemLdeterminesGa}
  Let $L\in\gL$. Then there exists a unique $\ell\in\gG$ such that the
  complex extension $L(z,w)\in\C[[z,w]]$ (as defined
  in~\eqref{eqfCtwotoC})\footnote{No
    confusion should be made with the ephemerous notation $L(x,y)$
    used only in the proof of Lemma~\ref{Phi}.} satisfies % \David{}
\begin{equation}
  L(z,z) = \big( \ell(z) \big)^2.
\end{equation}

Suppose furthermore that $L$ is admissible for $F\in\Fom$. Then there is only one $\Ga \in
\gG$ such that $(L,\Ga)$ is an $F$-admissible pair, and it is explicitely
determined by the equation
\begin{equation}   \label{eqdeterminGa}
  \Ga(z^2) = L\circ \hat F\big( \ell\ii(z),\ell\ii(z) \big),
\end{equation}
where~$\hat F$ is the complex extension of~$F$ as in~\eqref{notahatF}.
\smallskip

Moreover, if $F\in\Fomcv$ and~$L$ is convergent, then~$\ell$ and~$\Ga$ are
  convergent.
\end{lemma}

\begin{proof}
Considering an arbitrary $L\in \gL$ as in~\eqref{eqdefgL} and
  composing its complex extension by
  \begin{equation}   \label{eqdefiota}
    \iota \col \C \to \C^2, \qquad \iota(z) \defeq (z,z),
  \end{equation}
 we see that
  \begin{equation}   \label{eqLcirciota}
    L\circ\iota(z) = z^2 + \sum_{n\ge3} \bigg( \sum_{r+s=n} L_{rs}
    \bigg) z^n
  \end{equation}
  has real coefficients and is of the form $L\circ\iota(z) = z^2\big(
  1 + O(z) \Big) \in \R[[z]]$.
  Substituting $z^{-2}L\circ\iota(z)-1$ into $\sqrt{1+T} = 1 + \frac12
  T + \cdots \in\R[[T]]$ and multiplying the result by~$z$, we get
  $\ell(z)\in\R[[z]]$ of the form $\ell(z)=z\big(1+O(z)\big)$ such
  that $L\circ\iota = \ell^2$, as desired, and it is clearly unique
  (and convergent if~$L$ is convergent).  
  
  We thus have
  \begin{equation}   \label{eqdefSz}
    L\circ\iota = S\circ \ell
    \qquad \text{with} \quad S(z)\defeq z^2.
  \end{equation}
  We now suppose that $L$ is $F$-admissible, \ie that there is a
  $\Ga\in\gG$ such that $L\circ F = \Ga \circ L$.
  For the complex extensions of~$L$ and~$F$, this yields (see also~\eqref{conj1})
  \begin{equation}   \label{eqLhatFGaLzw}
    L\circ \hat F(z,w) = \Ga\circ L(z,w),
  \end{equation}
  whence $L\circ\hat F\circ\iota = \Ga \circ L\circ\iota = \Ga \circ S
  \circ\ell$, or
  \[
    \Ga \circ S = L\circ\hat F\circ\iota\circ\ell\ii,
  \]
  which is exactly~\eqref{eqdeterminGa}.
  The right-hand side is a univariate formal series that is guaranteed
  to be even, of the form $z^2+O(z^4)$, with real coefficients, because we assumed that
  there is a solution $\Ga\in\gG$; now we see that this solution is
  unique. Moreover, if~$F$ and~$L$ are convergent, then the right-hand side is
  convergent, and so is~$\Ga$.
  \end{proof}

In particular, as announced in Section~\ref{secdecoupl}, the conjugacy
class of~$\Ga$ is uniquely determined and Definition~\ref{defformcons} makes
sense.
%
%%%%%%%%%%%%%%%%%%%%%%%%%%%%%%%%%%%%%%%%%%%%%%%%
%%%%%%%%%%%%%%%%%%%%%%%%%%%%%%%%%%%%%%%%%%%%%%%%
%
%  \subsection{Two cases of formal conservativity}
  %
%%%%%%%%%%%%%%%%%%%%%%%%%%%%%%%%%%%%%%%%%%%%%%%%
%%%%%%%%%%%%%%%%%%%%%%%%%%%%%%%%%%%%%%%%%%%%%%%%
%
Note that, if $F$ is ``formally conservative'' as in
Definition~\ref{defformcons}, then
\[
  \text{$L$ is $F$-admissible}
  \enspace\Longleftrightarrow\enspace
  L\circ F = L
  \]
  for any $L\in\gL$,
  \ie $F$ is formally integrable and the $F$-admissible series are
  first integrals of it.
%
% amounts to the existence of a
% formal first integral (since $\Ga=\ID$ precisely means that $L\circ
% F=L$).
%  
Here are two instances of this property:

\begin{proposition}\label{integ} Let $F\in\Fom$.
  \medskip
  
  \noindent (i) If $F$ is a complex holomorphic map,
  % with an analytic geometric normalization,
%
  then $F$ is formally conservative.
  \medskip
  
  \noindent (ii) If $F$ is area-preserving,
%
%with an analytic geometric normalization,
%
then $F$ is formally conservative.
% \\ Conversely, if $F$ is analytically integrable, then $F$ has an analytic geometric normalization.
\end{proposition}

\begin{proof}
  %
%   By Corollary \ref{setF},  if there exists a $F$-admissible pair $(L_*,\Gamma_*)$ with $\Gamma_*(R)=R$, then for all $F$-admissible pairs $(L,\Gamma)$,  we have
% \[L\circ F=\Gamma\circ L=g\circ \Gamma_*\circ g^{-1}\circ L=L,\]
% in particular, the existence of an analytic $F$-admissible $L$ implies that $L$ is an analytic integral of $F$.
% \smallskip

  $(1)$ If $F(z)=\lambda z+\sum_{j\geq 2}f_j \, z^j$, then there exists a
  (unique) formal series of the form 
  $\phi(z)=z+\sum_{j\geq 2} \phi_jz^j$ that linearizes~$F$, \ie
  $\phi\circ F(z)=\lambda \,\phi(z)$. % $h\circ F\circ h^{-1}(z)=\lambda z$.
  This~$\phi$ can be viewed as a particular kind of tangent-to-identity
  diffeomorphism
  and the rotation $z\mapsto \la z$ as a particular kind of normal
  form,   % \David{}
  we thus see that
  $(|\phi|^2,\ID)$ is an $F$-admissible pair
  (see Appendix~\ref{appaltpf} \eqref{eqFholom}--\eqref{eqnotahF}).
  \smallskip

  $(2)$ If $F$ is area-preserving, then there exists an area-preserving
  normalization~$\Phi$ that takes~$F$ to an % a unique
  area-preserving normal form~$N$ (the Birkhoff normal
  form). Since~$N$ preserves the foliation $[\nu]$ by circles, one
  must have $|N(z)|=|z|$, hence
  $(|\Phi|^2,\ID)$ is an $F$-admissible pair.
  % one gets a $F$-admissible $\Gamma_*(R)=R$.
%
\end{proof}

\section{Selecting a geometric normalization}   \label{secselect}
%%%%%%%%%%%%%%%%%%%%%%%%%%%%%%%%%%%%%%%%%%%%%%%%%%%%%%%%

Our aim is to show that, for a generic $F\in\Fomcv$, the formal
invariant foliation is divergent: all admissible~$L$ and thus all
formal geometric normalizations are divergent as claimed in Theorem~\ref{main}.

The main difficulty encountered when trying to prove generic
divergence of geometric normalizations is the non-uniqueness of
admissible series.
%
% Indeed, if a selection process restoring uniqueness is found which
% is such 
%
We thus need a selection process so as to single out a particular
solution, but in such a way
that the divergence of the selected $F$-admissible series implies the
divergence of all other $F$-admissible series; then the Ilyashenko -
P\'erez-Marco alternative (see Section~\ref{general}) will reduce the
question to the construction of one example where divergence occurs.

We describe two such processes: the first one, analogous to the
one which had been used in \cite{CSSW} under the name of \emph{basic}
normalization, is very natural but does not suffice % work
here and
must be supplemented % had to be replaced
by the second one, more sophisticated.

\subsection{Selection by the resonant part}\label{respart}

\begin{definition}\label{res}
  Given $L(z)=|z|^2+\sum_{r+s\ge 3}L_{rs}{z^r\bar z^s}\in\gL$,
  the series
  %
%  we use the notation
%  \[ L=<L>+\{L\},\quad \text{where}\quad  <L>(z)
%  =|z|^2+\sum_{n\ge 2}L_{nn}|z|^{2n}=|z|^2+\rho_L (|z|^2)\]
%
  \begla
  \rho_L(R) \defeq\sum_{n\ge 2}L_{nn}R^n\in R^2\R[[R]]
  \edla
  is called the \emph{resonant part} of $L$.
\end{definition}

\begin{lemma}\label{couple}
  Given $F\in\Fom$, % $F(z)=e^{2\pi i\omega}z+O(|z|^2)$ with $\omega$ irrational,
  the map
  \[
    (L,\Gamma) \in \{\text{$F$-admissible pairs}\} \mapsto \rho_L \in
      R^2\R[[R]]
    \]
    is bijective.
%
  % the correspondence $(L,\Gamma)\mapsto \red{\rho_L}$, which to an
  % admissible pair $(L,\Gamma)$ for $F$ associates the resonant part
  % $\red{\rho_L}(R)\in R^2\R[[R]]$ of $L$, is bijective.
  %
    In other words, for any $\rho(R)\in\ R^2\R[[R]]$, there exists a
    unique $(L,\Gamma)\in\gL\times\gG$ satisfying
  \begin{equation}   \label{eqGaLwithrho}
    \Gamma\circ L=L\circ F\quad\text{and}\quad \rho_L=\rho. % <L>=\nu+\rho\circ\nu.
  \end{equation}
\end{lemma}

\begin{proof}
  Let us write % \David{}
\beglab{eqFFstLLstGGst}
    F=F^{(1)}+F_*, \qquad L=\nu+L_*, \qquad \Gamma=\ID+\Gamma_*
\edla
    with $F^{(1)}(z)\defeq\la z$ and
    \beglab{eqFstFjketc}
      F_*(z)=\sum_{j+k\ge 2}F_{jk}z^j\bar z^k, \qquad
      L_*(z)=\sum_{r+s\ge 3}L_{rs}z^r\bar z^s, \qquad
      \Ga_*(R)=\sum_{n\ge 2}\Ga_n R^n.
    \edla
    We consider the coefficients~$F_{jk}$ and those of~$\rho(R)$ as given constants, and we
    want to solve~\eqref{eqGaLwithrho} for the unknowns~$L_{rs}$ and~$\Ga_n$.
%
%\marginlabel{\color{blue}$\Lambda, A,\gamma,f$ respectively replaced by $F^{(1)},L_*,\Gamma_*,F_*$}
%
    The equation $\Gamma\circ L=L\circ F$ can be rewritten % in the form
    \[
      L_*+\Gamma_*\circ(\nu+L_*)=\nu\circ
      F-\nu+L_*\circ(F^{(1)}+F_*),
    \]
    that is, applying the Taylor formula,
    \beglab{eqdefAlowBlow}
      L_*-L_*\circ F^{(1)} + \Ga_*\circ\nu = \nu\circ F-\nu
      \, + \, \sum_{a+b\ge 1} \frac{1}{a!b!} \bigl(
      (\pa^a_z\pa^b_{\bar z}L_*)\circ F^{(1)} \bigr) F_*^a
      \ov F_*^{\raisebox{-.5ex}{$\scriptstyle b$}}
      \, - \,
      \sum_{c\ge 1}\frac{1}{c!}(\Ga_*^{(c)}\circ\nu)L_*^c.
    \edla
Let us denote by~$A_{rs}$ the coefficient of $z^r\bar z^s$ in the
right-hand side of~\eqref{eqdefAlowBlow}.
%    
% Identifying terms in $z^r\bar z^s$ of both sides,
%
We get
\begin{equation}\label{dominant}
 L \circ F = \Ga\circ L  \quad\Longleftrightarrow\quad
  \left\{ \begin{alignedat}{2}  % \begin{split}
(1-\la^{r-s})L_{rs}&= % \la \ov{F_{s,r-1}}+\bar\la F_{r,s-1}+
A_{rs} % B_{rs}^{low} 
\quad && \text{when $r\neq s$, $r+s\ge3$,} \\[1ex]
\Ga_n&= % \la\ov{F_{n,n-1}}+\bar\la F_{n,n-1}+
A_{nn} % A_{2n}^{low} 
\quad && \text{when $r=s=n\ge2$.}
\end{alignedat} \right. % \end{split} 
\end{equation}
We observe that $A_{rs}$ is a polynomial expression
% $A_{2n}^{low}$ and $B_{rs}^{low}$ are polynomial expressions
in certain coefficients of~$\Ga$ and~$L$, namely 
%
% (and also in coefficients of~$F$, but the latter ones are treated as
% constants and the former ones as unknowns),
  %
\beglab{eqArspolyn}
\text{$A_{rs}$ is a polynomial in $\big(L_{r's'}\big)_{r'+s'<r+s}$ and
  $\big(\Ga_{n'}\big)_{2n'<r+s}$.}
\edla
Indeed, the sum over~$c$ contains 
$\Ga_{n'}L_{r'_1s'_1}\cdots L_{r'_c s'_c}$ in the coefficient of
$z^r\bar z^s$ only if
$n'-c + \sum r'_i=r$ and $n'-c+\sum s'_i=s$,
which implies \[2n' + \sum(r'_i+s'_i-3) = r+s-c<r+s\]
with $n'\ge2$ and each $r'_i+s'_i\ge3$,
whence $2n'<r+s$ ,
and the sum over $(a,b)$ contains $L_{r's'}$ in the coefficient of
$z^r\bar z^s$ only if
$r'-a + j=r$ and $s'-b+k=s$ where $z^j\bar z^k$ stems from
$F_*^a \ov F_*^{\raisebox{-.5ex}{$\scriptstyle b$}}$,
which implies \[ r'+s'= r+s +a+b-(j+k) < r+s\] because $j+k\ge 2(a+b)$
(see Lemma~\ref{lemrevisitgrs} below %  Section~\ref{general}
for a more detailed computation, but the above
argument is enough for now).

% \noindent
%
Since $\la^m\neq1$ for any $m\in\Z^*$, the upshot is
\[
  \text{\eqref{eqGaLwithrho}} \enspace\Longleftrightarrow\enspace
  \left\{ \begin{alignedat}{2}  
& L_{rs} = \frac{A_{rs}}{1-\la^{r-s}} % \frac{B_{rs}^{low}}{1-\la^{r-s}}
\quad && \text{when $r\neq s$, $r+s\ge3$} \\[1ex]
\text{$\Ga_n= A_{nn}$}\enspace \text{and} & \enspace \text{$L_{nn} = $ coefficient of~$R^n$ in~$\rho(R)$}
\quad && \text{when $r=s=n\ge2$}
\end{alignedat} \right.
\]
and~\eqref{eqArspolyn} shows that our system of equations can be solved by induction
on $r+s$ and uniquely determines~$L$ and~$\Gamma$.
\end{proof}

\begin{corollary}    \label{determined}
  Given $F\in\Fom$,  the map
  \begin{equation}   \label{eqbijLrhoL}
    L \in \{\text{$F$-admissible formal series}\} \mapsto \rho_L \in
    R^2\R[[R]]
  \end{equation}
  %
  % which to an $F$-admissible~$L$ associates its resonant
  % part~$\rho_L$,
  %
  is bijective.
 \end{corollary}

 \begin{proof}
  By Lemma~\ref{couple} the pair $(L,\Gamma)$ and thus the series~$L$ are uniquely determined by~$\rho_L$;
%
%the conclusion follows because $\rho$, being a part of $L$,
%
conversely, $\rho_L$ is uniquely determined by $L$.
\end{proof}

\begin{remark} \textbf{Alternative proof of the uniqueness of the formal
    $F$-invariant foliation:}
Given
\[
  g(R)=R+\sum_{c\geq 2}g_cR^c\in\gG
  \quad\text{and}\quad
  L(z)=\sum_{r+s\geq 2}L_{rs}z^r\bar{z}^{s}\in \gL,
  \]
one can expand $L'\defeq g\circ L = L + \sum_{c\ge2} g_c L^c =
\sum_{r+s\geq 2}L'_{rs}z^r\bar{z}^{s}$ and extract its resonant
part $\rho_{L'}(R)=\sum_{n\geq 2}L_{nn}' R^n$: for each $n\ge2$,
\begin{align}   \label{resonant0}
  L_{nn}' &= L_{nn} +  \sum_{c\ge2} \,
  \sum_{r_1+\cdots +r_c=s_1+\cdots+s_c=n \atop r_i+s_i\ge 2}
  g_c \,L_{r_1s_1}\cdots L_{r_ns_n}\\%[1ex]
  \label{resonant1}
  &= L_{nn} + g_n+ \sum_{c=2}^{n-1} \,
  \sum_{r_1+\cdots +r_c=s_1+\cdots+s_c=n \atop r_i+s_i\ge 2}
  g_c \,L_{r_1s_1}\cdots L_{r_ns_n}
\end{align}
because~\eqref{resonant0} had a summation with $2c \le \sum(r_i+s_i) =2n$, whence
$c\le n$, and the only possibility with $c=n$ was $r_i+s_i=2$ for
each~$i$, whence $(r_i,s_i)=(1,1)$ and $L_{r_i s_i}=1$ in that case.
%
% In view of Corollary~\ref{determined}, w

We can thus provide an alternative self-contained proof of the
uniqueness of the invariant formal foliation obtained in
Corollary~\ref{uniqueL} (without resorting to Lemma~3 and Corollary~1
of~\cite{CSSW} as we did in Section~\ref{secforminvfol}):
% \medskip

Suppose $L_1$ and $L_2$ are $F$-admissible.
Let~$L_1$ play the role of~$L$ and~$L_2$ that of~$L'$
in~\eqref{resonant1} and solve for the $g_n$'s:
we thus obtain $g\in\gG$ such that $\rho_{g\circ L_1}$ and~$\rho_{L_2}$
coincide.
Now, $g\circ L_1$ and~$L_2$ being two $F$-admissible series with
the same resonant part, Corollary~\ref{determined} says that they must
coincide.
\end{remark}

\begin{definition}   \label{defresfreeLFst}
  We define the ``resonant-free $F$-admissible series'' to be the
  unique $F$-admissible series~$L_F^*$ whose resonant part is~$0$.
%
  % A formal geometric normalization $\Phi$ such that the resonant part
  % of $L=|\Phi|^2$ is $\rho$ is called a $\rho$-geometric
  % normalization.
%
\end{definition}

% \noindent
%
Corollary~\ref{setF} implies an interesting property of the inverse of
the bijection~\eqref{eqbijLrhoL} described in Corollary~\ref{determined}:
\begin{proposition}
  The inverse $\rho\mapsto L$ of
the bijection~\eqref{eqbijLrhoL} %  described in Corollary~\ref{determined}
  % such that $L$ determines an admissible pair $(L,\Gamma)$ for $F$
  %
  preserves affine combinations.
\end{proposition}

\begin{proof}
Given $\rho_0,\rho_1 \in R^2\R[[R]]$ and $t\in\R$, we call $L_0$
and~$L_1$ their images by the inverse of~\eqref{eqbijLrhoL}.
Corollary~\ref{setF} yields $g\in\gG$ such that $L_1=g\circ L_0$.
Now, the affine combination $g_t=t\ID+(1-t)g$ also belongs to~$\gG$, this implies that the affine combination 
\[ L_t \defeq tL_0+(1-t)L_1=\bigl[t\ID+(1-t)g)\bigr]\circ L_0=g_t\circ
  L_0 \]
%
% determines an admissible pair $(L_t,\Gamma_t)$ for $F$, where
% $\Gamma_t=g_t\circ \Gamma\circ g_t^{-1}$.
%
is $F$-admissible. Since $\rho_{L_t} = t\rho_0+(1-t)\rho_1$,
we see that $L_t$ is the image of $ t\rho_0+(1-t)\rho_1$ by the
  inverse of the bijection~\eqref{eqbijLrhoL}, as desired.
%
% the proof is complete.
%
% As the resonant part of
% $L_t$ is the affine combination
% $t\rho_0+(1-t)\rho_1$, we conclude that $\rho_t$ is mapped to
% $L_t$ since the map $\rho\mapsto L$ is bijective.}
%
\end{proof}
\medskip

With Lemma~\ref{couple} we have seen how to find a unique $F$-admissible
series with prescribed resonant part~$\rho$.
As a variant, one may prescribe the even part of the univariate series
$L\circ\iota(z)$ that we already used in~\eqref{eqLcirciota}
(recall that $L\circ\iota(z)=L(z,z)$, where $L(z,w)$ is the complex
extension of~$L(z)$ as in~\eqref{eqfCtwotoC}).
We thus introduce the notation
\begin{multline}
  L\in\gL \mapsto \chi_L\in\gG \enspace \text{determined by} \\
  \chi_L(z^2)\defeq\frac{1}{2}\bigl(L(z,z)+L(-z,-z)\bigr)
   = z^2 + \sum_{m\,\text{even}\ge4} \bigg( \sum_{r+s=m} L_{rs}
   \bigg) z^m.
   \label{eqdefsigL}
\end{multline}
%
% Another selection of unique admissible pair $(L,\Gamma)$ equivalent to that by the resonant part $\rho$ of $L=L(z,\bar z)$ is by the even part of $L(z,z)$ (where $\bar z$ in $L$ is replaced by $z$). 
% More precisely, let  
%
% be the even part of $L(z,z)$, we have the following selection Lemma:
%
\begin{lemma}\label{even}
  Given $F\in \Fom$ and $\rho(R)=R+\sum_{n\ge 2}\rho_nR^n\in\gG$,
  there exists a unique $F$-admissible~$L$
  % pair $(L,\Gamma)$ for $F$
  %
  such that $\chi_L=\rho$.
\end{lemma}

\begin{proof}
In order to find an admissible pair $(L,\Ga)$ such that
  $\chi_L=\rho$, we proceed as in the proof of Lemma~\ref{couple}.
  With the notation~$A_{rs}$ for the coefficient of $z^r\bar z^s$ in the
right-hand side of~\eqref{eqdefAlowBlow},
  the unique solution to our
  problem is obtained by induction on $r+s\ge3$:
\[
  \left\{ \begin{alignedat}{2}  
& L_{rs} = \frac{A_{rs}}{1-\la^{r-s}} % \frac{B_{rs}^{low}}{1-\la^{r-s}}
\quad && \text{when $r\neq s$, $r+s\ge3$} \\[1ex]
\text{$\Ga_n= A_{nn}$}\enspace \text{and} & \enspace
L_{nn} = \rho_n - \displaystyle\sum_{r+s=2n \atop r\neq s} L_{rs}
\quad && \text{when $r=s=n\ge2$.}
\end{alignedat} \right.
\]
%
% \[
%   %
%   r\neq s \enspace\Longrightarrow \enspace  L_{rs} =
%   %
% \frac{A_{rs}}{1-\la^{r-s}} % \frac{B_{rs}^{low}}{1-\la^{r-s}}
%   %
% \]
% %
% and $r=s=n\ge2 \;\Longrightarrow\; \Ga_n = A_{nn}$ % A_{2n}^{low}
% %
% and $L_{nn} = \rho_n - \displaystyle\sum_{r+s=2n \atop r\neq s} L_{rs}$.
%
% 
%   Denote $L(z,z)=\sum L_{mn}z^{m+n}$, we have 
% \[\sigma_L(R)=\sum_{m+n=2r}L_{mn}R^r,\]
% Given $\rho(R)=R+\sum_{n\geq 2}\rho_n R^n$,  the equation $\sigma_L=\rho$ writes as
% $$\rho_n=L_{nn}+\sum_{m'+n'=2n, m'\not = n'}L_{m'n'},$$
% Hence $\rho_n$ uniquely determines $L_{nn}$, while $L_{m'n'},m'\neq n'$ are determined inductively by (\ref{dominant}). By Lemma \ref{couple}, a unique admissible pair $(L,\Gamma)$ is thus determined.
%
%By induction on $m+n$ one gets that
%if $m\neq n$, $L_{mn}$ is expressed in terms of the $L_{m'n'}$ for $m'+n'<m+n$ and the $\rho_{n'}=L_{n'n'}$ for $n'<n$.
%if $m=n$, $L_{nn}$ is determined by the equation
%$$\sum_{m'+n'=2n, m'\not = n'}L_{m'n'}+L_{nn}=\rho_n.$$
\end{proof}

%Their equivalence can be seen from the following observation:
%$$\Lambda_{2r}=\sum_{r'+s'=2r}L_{r's'} =L_{rr}+\sum_{r'+s'=2r\atop\scriptstyle (r',s')\neq(r,r)}L_{r's'}$$ 
%where at each step of the induction in the determination of an admissible pair in Lemma \ref{couple}, one has the freedom to choose either $\Lambda_{2r}$ or $L_{rr}$. This implies that a possible choice warranting uniqueness is to fix a priori the even part of $\Lambda$ .

%The equation 
%$$\Lambda_{2r}\defeq\sum_{r'+s'=2r}L_{r's'} =L_{rr}+\sum_{r'+s'=2r\atop\scriptstyle (r',s')\neq(r,r)}L_{r's'}$$ 
%shows the equivalence at each step of the induction of choosing $\Lambda_{2r}$ or $L_{rr}$. 
%$\sum_{d\ge 2\atop\scriptstyle d\, even}\Lambda_dz^d$ 

% \noindent
%

\subsection{Selection induced by the foliation involution}\label{realaxis}
The resonant-free $F$-admissible series might seem a natural choice for a``good'' geometric normalization
among the others,
%
% A natural first attempt to select a``good'' geometric normalization among the
% others would be to choose the resonant part $\rho_L(R)$ of $L$ equal to 0,
%
but numerical simulations suggest that this choice
does not share the property of the \emph{basic} normalization in
\cite{CSSW}, namely its divergence does not seem to imply divergence
of all normalizations.

Instead, we shall consider the restriction of~$L$ to the real
  axis and use it to define a formal involution~$\tau$, which will then
  help us to define a new selection procedure.
If we use again the complex extension $L(z,w)$ of $L(z)$ as
in~\eqref{eqfCtwotoC}, restricting to~$z$ real amounts to
replacing~$w$ by~$z$ and considering $L(z,z)=L\circ\iota(z)$
as in~\eqref{eqdefiota}--\eqref{eqLcirciota}.
%
% As in Lemma~\ref{lemLdeterminesGa}, we denote by $L(z,w)$ the
% two-variable  ``complex extension'' obtained from
% $L(z)\in\gL\subset\C[[z,\bar{z}]]$ by replacing $\bar z$ by an
% independent variable $w$ (see Appendix~\ref{appA}).
% %
% The restriction of~$L$ to the real axis $z\in\R$ is thus nothing but
% the real-coefficient formal series $L(z,z)$.

\begin{lemma}  \label{lemdefinvol}
(i) Given a formal foliation $\cF \in \gG\backslash \gL$, there is a
  unique formal involution of the form
  \begin{equation}   \label{eqformtauz}
    \tau(z)=-z+O(z^2) \in \R[[z]]
  \end{equation}
such that, for any representative~$L$ of~$\cF$,
  \begin{equation}   \label{eqLauztauz}
    L(\tau(z),\tau(z))=L(z,z).
  \end{equation}

 (ii) For each $L\in\gL$, equation~\eqref{eqLauztauz} determines a
    unique formal involution of the form~\eqref{eqformtauz}, that can be obtained as
  \begin{equation}   \label{eqtauelliisigell}
    \tau=\ell^{-1}\circ\sigma\circ\ell,
    \qquad \text{where}\quad \sig \defeq -\ID
  \end{equation}
and $\ell(z)\in\gG$ is the square root of $L(z,z)$ introduced in Lemma~\ref{lemLdeterminesGa}.
\end{lemma}

We emphasize that~$\tau$ depends only on the foliation~$\cF$, not on
any particular representative~$L$ (the solution to~\eqref{eqLauztauz}
is invariant under the action
\eqref{eqleftaction}--\eqref{eqleftquotient} of~$\gG$ on~$\gL$).
%
% \David{}

\begin{proof}
 Geometrically, if~$L$ is convergent, the whole lemma is clear.
Indeed, equation~\eqref{eqLauztauz} determines a unique real-analytic
  function~$\tau$ of the form~\eqref{eqformtauz} and it must be an
  involution (for $z\in\R^*$ close enough to~$0$, the curve $L\ii\big(L(z)\big)$
  intersects the real line at two points,~$z$ and~$\tau(z)$, see the figure below);
the function~$\tau$ is clearly unchanged upon replacing~$L$ by $g\circ
L$ for a convergent $g\in\gG$.

\hskip-0.0cm
\includegraphics[scale=0.7]{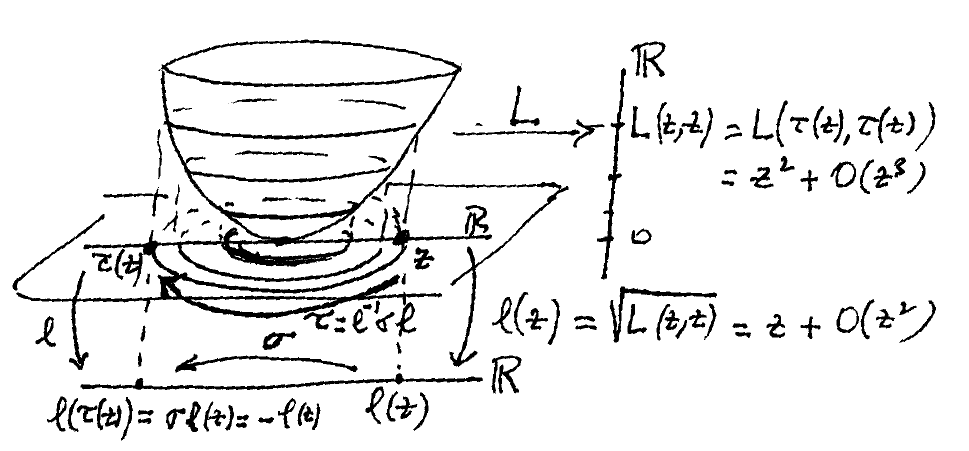}

\smallskip

 To prove the lemma in all generality, we start with the second statement.
  For a general formal series $L\in\gL$, we can rewrite~\eqref{eqLauztauz} as
\[
  L\circ\iota\circ\tau = L\circ\iota.
\]
 Since $L\circ\iota = \ell^2$, \eqref{eqLauztauz} $\,\Longleftrightarrow\,
\ell\circ\tau = \pm\ell$, which is equivalent to $\ell\circ\tau = -
\ell$ if we assume $\tau\neq\ID$.
This yields~\eqref{eqtauelliisigell} as desired because
$-\ell=\sig\circ\ell$.  

 The first statement follows because any representative of~$[L]$ is of
the form $g\circ L$ with $g\in \gG$, and \eqref{eqLauztauz} $\,\Longleftrightarrow\,
g\circ L\circ\iota\circ\tau = g\circ L\circ\iota$. 
\end{proof}
  
\begin{remark}\label{allconj}   % \Alain{}
  We see from Lemma~\ref{lemdefinvol}(ii) that the formal
  involution~$\tau$ is formally conjugate to $\sigma=-\hbox{Id}$.
  That property is shared by all formal involutions in one variable,
  real or complex---see Appendix~\ref{appforminvol}.
%
% All the formal involutions stemming from formal foliations as in
% Lemma~\ref{lemdefinvol} are formally conjugated to $\sigma=-\ID$ (we
% have $\tau=\ell^{-1}\circ \sigma\circ \ell$).  In fact, this is true
% for all formal involutions, real or complex---see
% Appendix~\ref{appforminvol}.
%
\end{remark}

% \begin{lemma} \label{canon}
% The involution $\tau$ is canonically attached to $F$, \red {precisely to the invariant foliation attached to $F$}, that is, it is independent of the choice of a $F$-admissible $L$.
% %a solution $L(z,w)$ of equation \eqref{conj}.
% \end{lemma}
% \begin{proof} 
% Geometrically, this follows from uniqueness of the invariant foliation canonically attached to $F$. \marginlabel{\red {Alain. I changed the sentence in order not to cut formulas}}
% Formally, this follows from the equivalence for any $g\in \gG$ of the formul\ae  $\; L(\tau(z),\tau(z))=L(z,z)$ and $g\circ L(\tau(z),\tau(z))=g\circ L(z,z)$.
% \end{proof}

\smallskip

\begin{definition} Given $F\in\Fom$, the formal involution defined by
  Lemma~\ref{lemdefinvol} applied to~$\cF_F$ (the formal $F$-invariant
  foliation) is called the ``foliation involution'' of~$F$ and denoted
  by~$\tau_F$.
\end{definition}
%\marginlabel{\color{blue}replacement of $\tau$ by $\tau_F$}

\begin{proposition}\label{unique} Given $F\in\Fom$, there is a unique $L_F\in\gL$ admissible for $F$ such that 
$$L_F(z,z)=h(z),\quad\hbox{where}\quad h(z)\defeq-z\,\tau_F(z)=z^2+O(z^3).$$
\end{proposition}

\begin{proof}
Let~$L$ be an arbitrary $F$-admissible series. According to
  Corollary~\ref{setF}, we just need to show the existence of a unique $g\in\gG$ such that
  $g\circ L(z,z)=h(z)$ (and then $L_F=g\circ L$).
  
Applying Lemma~\ref{lemLdeterminesGa} to~$L$ and using
notations~\eqref{eqdefiota} and~\eqref{eqdefSz}, we
have
$L \circ \iota = S \circ \ell$.
The equation $g\circ L\circ\iota = h$
is equivalent to $g\circ S = h\circ\ell\ii$, that is $g(z^2)=h\circ
\ell^{-1}(z)$.
Since $h(z)=z^2+O(z^3)$ and~$h$ and~$\ell$ have real coefficients, the existence and
uniqueness of a solution $g\in\gG$ is equivalent to the fact that
$h\circ \ell^{-1}$ is even, which can be checked as follows:
clearly \[ h = h\circ\tau_F \] because~$\tau_F$ is an involution,
hence $h\circ\ell^{-1} = h\circ\tau_F\circ\ell\ii = h\circ\ell\ii\circ\sig$
(using~\eqref{eqtauelliisigell} at the last step).
\end{proof}

%\noindent Starting from an arbitrary solution $L$ of equation \eqref{conj} one looks for a formal $M(z)=z+O(z^2)$ such that $M(L(z,z))=h(z)$, 
%that is $MS\ell=h$, where we denote by $S$ the map $S(z)=z^2$. As by (ii), $\ell\tau_F=-\ell=\sigma\ell$, 
%$h\ell^{-1}$ is $\sigma$-invariant: indeed, by (i), $h\ell^{-1}\sigma=h\tau_F\ell^{-1}=h\ell^{-1}$;  in other
%words $h\ell^{-1}$ is even. Moreover $h\ell^{-1}(z)=z^2+h.o.t.$ thus we can find
%$M=z+h.o.t.$ such that $MS=h\ell^{-1}$.

\begin{definition} The $F$-admissible formal series~$L_F$ defined by
  Proposition~\ref{unique} is called the balanced admissible series of~$F$.
\end{definition}

The terminology ``balanced'' refers to the following fundamental property:

\begin{proposition} \label{LFbasic}
  Given an analytic $F\in\Fom$,
    if $L_F$ is divergent, then all $F$-admissible series are
    divergent
    and, consequently, the formal $F$-invariant foliation is divergent
    and all formal geometric normalizations of~$F$ are divergent.
\end{proposition}

\begin{proof}
  Suppose that $F\in\Fom$ is analytic and $L$ is $F$-admissible and
    convergent.
    We must prove that $L_F$ is convergent.
    We know that $L_F = g \circ L$ for some $g\in\gG$, we will show
    that~$g$ is convergent.  

Using notations~\eqref{eqdefiota} and~\eqref{eqdefSz},
  Lemma~\ref{lemLdeterminesGa} implies that~$\ell$ is convergent,
    hence so is~$\tau_F$ by~\eqref{eqtauelliisigell},
    and so is~$h$ in Proposition~\ref{unique}.
    Now, $h = L_F \circ \iota = g\circ L \circ \iota = g \circ S \circ
    \ell$, \ie
    $g\circ S = h\circ \ell\ii$.
    The latter formal series is convergent
    % (\red{because so is $g\circ S$)}, \Alain{}
    (because~$h$ and~$\ell$ are), % \David{}
%    and known to be even),
    hence~$g$ is convergent.  
%
%   By Lemma \ref{canon}, existence of  a convergent solution $L(z,w)$ implies the convergence of $\tau_F=\ell^{-1}\circ \sigma\circ \ell$ and the one of $h(z)=-z\tau_F(z)$. Writing 
% $L_F=g\circ L$ we get $h=g\circ S\circ \ell$, that is $h\ell^{-1}(z)=g(z^2)$ .
% %or $z^2=Mh\ell^{-1}(z)$. 
% So, convergence of $h\ell^{-1}$, which is a series in $R=z^2$, implies convergence of $g(R)$  hence of $L_F $.
%
\end{proof}

In the proof of Proposition \ref{LFbasic} we noticed that~$\tau_F$ is
convergent as soon as there exists a convergent
$F$-admissible~$L$.
However, we should not expect the converse to hold.  The following
lemma indeed shows that~$\tau_F$ is convergent in a case for which we
will later show the generic divergence of all admissible series. 

\begin{lemma}\label{Fodd}
  If $F\in \Fom$ is odd, that is if $F(-z)=-F(z)$, then $L(-z)=L(z)$
  for any $F$-admissible $L$. In particular, one has $\tau_F(z)=-z$.
\end{lemma}

\begin{proof}
  %
%Denote $s:(z,w)\mapsto (-z,-w)$, then $F$ odd is equivalent to $\hat F\circ s=s\circ \hat F$. If $L$ is any solution of equation \eqref{conj}, or equivalently (\ref{conj1}) in the Appendix, one has 
%{\color{red}$$\Gamma\circ L\circ s=(L\circ s)\circ(s\circ \hat F\circ s)=(L\circ s)\circ \hat F,$$}
%which proves that $L\circ s$ is also a solution.
%Denote $s: z\mapsto -z$.
  %
  Suppose that $L\in\gL\subset\C[[z,\bar z]]$ is $F$-admissible.
  We thus have an $F$-admissible pair $(L,\Ga)$.
  The relation $\Ga\circ L = L\circ F$ implies
  $\Ga\circ L\circ\sig = L\circ F\circ\sig = L\circ\sig\circ F$,
  \ie $(L\circ\sig,\Ga)$ is $F$-admissible too.
  Since $\chi_{L\circ\sig}=\chi_L$ (notation~\eqref{eqdefsigL}),
  Lemma~\ref{even} shows that $L\circ\sig=L$. 
\end{proof}

%%%%%%%%%%%%%%%%%%%%%%%%%%%%%%%%%%%%%%%%%%%%%%
%%%%%%%%%%%%%%%%%%%%%%%%%%%%%%%%%%%%%%%%%%%%%%

\section{Genericity of the divergence of geometric
  normaliza\-tions}   \label{secgendiv}
  % --- the general case}   \label{secgendiv}

%%%%%%%%%%%%%%%%%%%%%%%%%%%%%%%%%%%%%%%%%%%%%%
%%%%%%%%%%%%%%%%%%%%%%%%%%%%%%%%%%%%%%%%%%%%%%

%Theorem \ref{gen} is completely unsatisfactory because it leaves unanswered the genericity with respect to $F$ of divergence of {\it all} geometric normalizations. 
%In order to appreciate the problem it is enough to start with $F$ which is itself already a geometric normal form: indeed in such a case, geometric normalizations are exactly arbitrary diffeomorphisms preserving the foliation by circles,
%that is local diffeomorphisms of the form 
%$z\mapsto z(1+f(|z|^2))e^{2\pi ig(z,\bar z)}$ 
%and divergent such diffeomorphisms do exist. Of course, knowing that for this $F$ there exist divergent geometric normalizations -- or even that geometric normalizations are generically divergent -- brings no information at all. 

%%%%%%%%%%%%%%%%%%%%%%%%%%%%%%%%%%%%%%%%%%%%%%
%%%%%%%%%%%%%%%%%%%%%%%%%%%%%%%%%%%%%%%%%%%%%%

\subsection{Two classical examples}   \label{secClassEx}

%%%%%%%%%%%%%%%%%%%%%%%%%%%%%%%%%%%%%%%%%%%%%%
%%%%%%%%%%%%%%%%%%%%%%%%%%%%%%%%%%%%%%%%%%%%%%

Consider an area-preserving map or a complex holomorphic map.
  In view of Proposition~\ref{integ},
  %
  % non-integrability (\ie the non-existence of analytic non-constant
  % first integrals) prevents the existence of any convergent geometric
  % normalization in the following cases:
%
if the map is not integrable (\ie if there exist no non-constant
analytic first integral), 
% then there exist no convergent geometric normalization.
then the unique formal invariant foliation is not analytic and the
balanced admissible series must be divergent.
This is what happens generically. Indeed:

\smallskip

\textbf{(1)} If $F$ is area-preserving, the existence of Birkhoff
zones of instability prevents the existence of an invariant foliation
because orbits exist which are asymptotic to both boundaries of the
zone.
%
%\color{blue}
%
C.~Genecand (\cite{G}), extending a result of E.~Zehnder (\cite{Z}),
%and proved generic existence of local non-integrable area-preserving
%maps for arbitrarily prescribed $N$-jet
%
proved the genericity of non-integrability for local area-preserving
diffeomorphisms with arbitrarily prescribed $(N-1)$-jet with a very fine
topology:
% on such local diffeomorphisms used to define genericity is the one
% in which
%
given an analytic area-preserving $F_0\in\Fom$ and an integer $N\ge4$,
consider the set $S(F_0,N)$ of all local analytic area-preserving
maps with the same $(N-1)$-jet as~$F_0$; these maps can be parametrised by
their generating functions as follows:
\beglab{eqdefgenu}
S(F_0,N) = \big\{\,  F_0\circ \Phi^u \mid
u(x,y') = \sum_{k+\ell>N} u_{k\ell} x^k y'^\ell \in\R\{x,y'\} \,\big\},
\edla
where $\Phi^u\col z=x+i y \mapsto z'=x'+i y'$ is the local analytic area-preserving
map of generating function~$u$, \ie implicitly defined by
\beglab{eqdefugen}
x' = x + \frac{\pa u}{\pa y'}(x,y'), \qquad
y = y' + \frac{\pa u}{\pa x}(x,y')
\edla
(note that the $(N-1)$-jet of~$\Phi^u$ is the identity);
an open neighborhood of $F_0\circ \Phi^u$ 
% $\phi(z)=\lambda z+\sum_{j+k\ge 1}c_{jk}z^j\bar z^k$ is
% the set of those $\psi$ whose Taylor coefficients $d_{jk}$ are equal to those of
% $\phi$ for $j+k\le N$ and
%
is by definition any set of the form $\{ F_0\circ\Phi^v \mid
k+\ell>N \imp |v_{k\ell}-u_{k\ell}|<\epsilon_{k\ell}\}$ for
some double sequence $(\epsilon_{k\ell})_{k+\ell>N}$ of positive reals.
%
% (technically the definition is better given in terms of generating functions but this is immaterial).

\begin{theorem} [\cite{G}]\label{area}
If $F_0$ has non-trivial Birkhoff normal form at order~$N-1$, then a
generic $F\in S(F_0,N)$ has non-degenerate homoclinic orbits and non-degenerate Mather
sets in every neighborhood of the origin. % a stable elliptic fixed point.
In particular, this implies local non-integrability.
  %
%  ``Generic" refers to a topology defined be means of the Taylor
%  coefficients of the mapping at the fixed point whose order is higher
%  than a prescribed arbitrarily integer, all lower coefficients being
%  fixed.
  %
\end{theorem}

In the course of the proof of Theorem~\ref{main}, we will require

\begin{corollary}\label{CorGenecand}
  Let $N\ge1$. 
  Given a polynomial
  $J(z)=\la z+\sum_{2\le j+k\le N}J_{jk}z^j\bar z^k$ that is the
  $N$-jet of a formal area-preserving map, for any $\De>0$ there
  exists $F\in\FomJcv$ area-preserving and non-integrable.
\end{corollary}

\begin{proof}
  By hypothesis, $J(z)$ is the $N$-jet of an area-preserving formal
  map of~$\Fom$, for which we can assume that the $(N+1)$-jet~$J_+(z)$
  is such that the Birkhoff normal form at order~$N+1$ is non-trivial
  (by changing the top-degree homogeneous part of~$J_+$ if necessary).
Now, let~$F_0$ be a polynomial area-preserving map of~$\Fom$ that has $(N+1)$-jet~$J_+$;
% and that has a complex extension holomorphic in all of~$\C^2$;
%
such~$F_0$ always exists by \cite[Theorem~1]{LPPW}---see
Proposition~\ref{propAreaP} in Appendix~\ref{appAreaP} for the details.
  
  Applying Theorem~\ref{area} with $M\defeq \max\{N+2,4\}$, for any $\eps>0$ we can find a non-integrable
  area-preserving map of the form $F_0\circ\Phi^u \in S(F_0,M)$ such that the
  coefficients of~$u$ satisfy $|u_{k\ell}| < \eps/k!\ell!$.
  We will have $F_0\circ\Phi^u\in\FomJcv$ if the complex
  extension of~$\Phi^u$ is holomorphic in the polydisc~$B_\De$,
  % and maps it in~$B_{\Dep}$,
  %
  and this is what happens for~$\eps$ small enough, as one can check
  by applying the contraction principle to solve the second equation
  in~\eqref{eqdefugen} to find~$y'$ in terms
  of $(x+iy,x-iy)\in B_\De$.
\end{proof}

%\color{black}

\medskip

\textbf{(2)} 
If $F:(\C,0)\to(\C,0)$ is complex holomorphic, the equivalence
between Lyapunov stability and analytic linearizability
(see \cite[Chap.~III \S~25]{SM} or \cite{PM3}) shows that no analytic invariant foliation (that is no
analytic geometric normalization) may exist when $F$ is not
analytically linearizable.
(See Appendix~\ref{appaltpf} for an alternative proof of that statement.)
%
%Non-linearizable holomorphic maps have been exhibited by J.-C. Yoccoz (\cite{Y}), L.~Geyer (\cite{Ge}), etc.
%
Non-linearizable holomorphic maps have been exhibited by J.-C. Yoccoz
\cite{Y} and L.~Geyer \cite{Ge}:
\begin{theorem}[\cite{Y}]\label{holo} 
The quadratic holomorphic map $F(z)=e^{2\pi i\omega}z(1-z)$ is not
analytically linearizable if and only if $\omega$ is non-Bruno, \ie
$\om$ satisfies~\eqref{eqdefnonB}.
\end{theorem}

%\begin{theorem}[\cite{Ge}]\label{Ge}
%\blue{  The holomorphic map $F(z)=e^{2\pi i\omega}z(1+z^d)$, $d\geq
%  2$, is not analytically linearizable if and only if $\omega$ is
%  non-Bruno.  }
%\end{theorem}
%
%\David{} \blue{  Actually, this is not Geyer's original example, but a
%  variant.
%  %
%  \cite{Ge} deals with $P_{\la,d}(z)=e^{2\pi i\omega}z(1+z/d)^d$ and
%  a little work is required to transfer Geyer's non-linearizability
%  result to the variant we need---see picture}
%{\verb|Variant_of_Geyer_s_example_IMG_8423|}
%\blue{  \textsc{ Put that little work in appendix...?} }
%
%%A direct proof is given in the appendix.\marginlabel{\color{red}David's text ``Holomorphic Case"}
%

\begin{theorem}[\cite{Ge}]\label{Ge}
 The holomorphic map $P_{\omega,d}(z)=e^{2\pi i\omega}z(1+z/d)^d$, $d\geq
  2$, is not analytically linearizable if and only if $\omega$ is
  non-Bruno.  
\end{theorem}

\begin{corollary}\label{CorGe}
 The holomorphic map $F(z)=e^{2\pi i\omega}z(1+z^d)$, $d\geq
  2$, is not analytically linearizable if and only if $\omega$ is
  non-Bruno.  
\end{corollary}

\begin{proof} % \Alain{}\marginlabel{\red {probably too elliptic for David ?}}
Again, this is an easy consequence of the equivalence
between Lyapunov stability and analytic linearizability recalled earlier. Indeed, the mapping 
$f(z)=e^{2\pi i\omega} z\left(1+\frac{z^d}{d}\right)$ is the pullback by the ramified covering
$Z=\rho(z)=z^d$ of the mapping 
$P_{d\omega,d}(Z)=e^{2\pi id\omega}Z\left(1+\frac{Z}{d}\right)^d$, that is
$\rho\circ f_d=P_{d\omega,d}\circ \rho$, and a finite ramified covering preserves stability around the point of ramification. Finally, on the one hand~$d\omega$ is non-Bruno if and only if~$\omega$ is non-Bruno, on the other hand the homothety $z\mapsto d^{1/d}z$ conjugates $f$ to the mapping $F$ in the corollary.
%{\verb|Variant_of_Geyer_s_example_IMG_8423|}
\end{proof}

%%%%%%%%%%%%%%%%%%%%%%%%%%%%%%%%%%%%%%%%%%%%%%
%%%%%%%%%%%%%%%%%%%%%%%%%%%%%%%%%%%%%%%%%%%%%%

\subsection{The IPM alternative}\label{general}

%%%%%%%%%%%%%%%%%%%%%%%%%%%%%%%%%%%%%%%%%%%%%%
%%%%%%%%%%%%%%%%%%%%%%%%%%%%%%%%%%%%%%%%%%%%%%

In the particular cases examined in Section~\ref{secClassEx} % up to now
(area-preserving or complex holomorphic),
no weak attraction can be present: the map is formally conservative in
the sense of Definition~\ref{defformcons}. Using
Proposition~\ref{LFbasic} we will prove Theorem~\ref{main}, which
covers all cases, even those where weak attraction may exist.
The method is an elaboration on Section~5.4 of \cite{CSSW},
which already relied on what we now call the
``Ilyashenko-P\'erez-Marco alternative'', reducing the proof of generic
divergence of~$L_F$ to exhibiting an example where~$L_F$ is divergent.
  
% \red{TO BE REMOVED: the existence of a balanced geometric normalization~$\Phi_F$
% associated to~$L_F$ (see Lemma~\ref{Phi}) allows us to use the
% Ilyashenko-P\'erez-Marco alternative and reduces the proof of generic
% divergence to the exhibition of an example.}

\medskip

\begin{definition}
Suppose that $(K_t)$ is a one-parameter family of formal series in one or several
indeterminates. We say that $(K_t)$ is ``regular-polynomial'' if there
exist real numbers~$a$ and~$b$ such that, for every
$N\in\Z_{\ge0}$,
the coefficient of every monomial of total degree~$N$ in~$K_t$
is a polynomial in~$t$ of degree $\le aN+b$.
\end{definition} 

The article \cite{PM2}, based on the Bernstein-Walsh inequality in
potential theory, leads to % proves

\begin{lemma}  \label{lemBW} % [Bernstein-Walsh,\cite{PM2}]
  Let $(K_t)$ be a regular-polynomial family of formal series.
%
  % polynomially parametric family of formal series of
  % parameter\footnote{The parameter can be of higher dimensional
  %   $t\in \C^{m}$ or $\R^m$.}  $t\in \R$ (or $t\in \C$), that is, the
  % coefficients of monomials of degree $j$ in $K_t$ is a polynomial of
  % $t$ with degree no larger than\footnote{Upper bound of the degree
  %   $j$ can be replaced by $Cj$ for some positive integer $C$.} $j$,
  %
  Then either $K_t$ is convergent for all $t\in\C$, or
  %
  % the set of $t$ such that $K_t$ is divergent has Lebesgue measure
  % zero.
  %
  $\{\, t\in\C\mid K_t \; \text{is divergent} \,\}$ has zero
  Lebesgue measure in~$\C$ 
  and $\{\, t\in\R\mid K_t \; \text{is divergent} \,\}$ has zero
  Lebesgue measure in~$\R$.
\end{lemma}

The reader can check Lemmas~15 and~16 of \cite{CSSW} for some details.
Building on Lemma~\ref{lemBW},
the arguments given in the proof of \cite[Theorem~1]{PM2}
(see also \cite[Sec.~5.4]{CSSW} and \cite{IYS}\footnote{The general idea of the
  alternative goes back to Ilyashenko but the recourse to potential
  theory, much more precise, is due to P\'erez-Marco.})
allow one to deduce
\begin{theorem}[Ilyashenko - P\'erez-Marco alternative] \label{thmIPM}
  Suppose that $\mathcal{F}$ is a closed affine subspace of~$\Fomcv$
  % local analytic functions
  %
  and $F\mapsto K(F)$ is a map from~$\cF$ to the space of formal
  series in one or several indeterminates such that,
  %
  % If a correspondence $F\mapsto K(F)$ from $\mathcal{F}$ to formal
  % series has the following property:
  %
  for any affine one-parameter family $F_t=(1-t)F_0+tF_1$ in~$\cF$,
  the corresponding family~$\big(K(F_t)\big)$ % $K_t=K_{F_t}$
  is regular-polynomial. % a polynomially parametric family of formal series,
  Then either $K(F)$ is convergent for all $F\in\cF$
  % all~$F$ in $\mathcal{F}$ have convergent~$K(F)$
  %
  or $K(F)$ is divergent for generic~$F$.
  %
  % or the generic~$F$ in~$\mathcal{F}$ has divergent~$K(F)$.
  %
  % , where genericity means the
  % property is true on a dense $G_{\delta}$-set for the compact-open
  % topology on analytic functions.
  %
\end{theorem}

  % \noindent
  %
% To $F\in\Fomcv$ one can associate its geometric normalizations $\Phi$,
% admissible pairs $(L,\Gamma)$ and canonical involution $\tau_F$. Let
% us denote by $K=K_F$ any one of these objects, selected according to
% some procedure so as to make the correspondence $F\mapsto K$ well
% defined.

To prove Theorem~\ref{main}, in view of Proposition~\ref{LFbasic},
in each of the cases (i)--(iv)
we just need to  \begin{enumerate}
  \item % (1)
    show that the Ilyashenko - P\'erez-Marco's principle can be
    applied
    with $\mathcal{F}=\Fomcv$ or $\FomJcv$
    to the map $F \mapsto L_F\in\C[[z,\bz]]$ (or to the map $F\mapsto \tau_F\in\C[[z]]$, since the
    divergence of~$\tau_F$ is stronger than the divergence of~$L_F$);
  \item % (2)
    find examples of $F\in\mathcal{F}$ % \David{} \red{for which the balanced normalization (???)}
    for which~$L_F$ is divergent.
  \end{enumerate}
In the rest of Section~\ref{general}, we will tackle the first point.

\bigskip

%%%%%%%%%%%%%%%%%%%%%%%%%%%%%%%%%%%%%%%%%%%%%%
%%%%%%%%%%%%%%%%%%%%%%%%%%%%%%%%%%%%%%%%%%%%%%

\noindent  {\bf  Polynomial dependence of the coefficients of the
  resonant-free admissible series} %$L_{F_t}$ on the parameter $t$}
\smallskip

%%%%%%%%%%%%%%%%%%%%%%%%%%%%%%%%%%%%%%%%%%%%%%
%%%%%%%%%%%%%%%%%%%%%%%%%%%%%%%%%%%%%%%%%%%%%%

With a view to making explicit the dependence of the coefficients~$L_{rs}$
of an $F$-admissible series~$L$ upon those of~$F$, we first revisit the
proof of Lemma~\ref{couple}.

\begin{lemma}   \label{lemrevisitgrs}
  With the notations~\eqref{eqFFstLLstGGst}--\eqref{eqFstFjketc},
  the coefficient of $z^r\bar z^s$ in the right-hand side of~\eqref{eqdefAlowBlow}
  is
  \[
    A_{rs} = \la \ov F_{s,r-1}+\bar\la F_{r,s-1}+ G_{rs}
    \qquad \text{for $r+s\ge3$,}
  \]
 where~$G_{rs}$ is the following polynomial in
  $\big(L_{jk}\big)_{j+k<r+s}$, $\big(\Ga_n\big)_{2n<r+s}$ and $\big(F_{jk}\big)_{j+k<r+s-1}:$
  \begin{multline}   \label{eqgrssumofthree}
  G_{rs} =
  \sum % \sum_{\substack{j+p=r \\[.5ex] k+q=s}}
  F_{jk} \ov F_{qp}
-\sum % -\sum_{n, c}\;\sum_{r_1,s_1,\cdots,r_c,s_c}
{n+c\choose c}\Gamma_{n+c}L_{r_1s_1}\cdots L_{r_cs_c} \\[1ex]
+\sum % +\sum_{a, b, \ell, m}\; \sum_{\scriptstyle j_1,k_1,..., j_a,k_a\atop\scriptstyle p_1,q_1,..., p_b,q_b} 
{\ell+a\choose a}{m+b\choose b} \la^{\ell-m} L_{\ell+a,m+b} F_{j_1k_1}\cdots
F_{j_ak_a}\ov F_{q_1p_1}\cdots \ov F_{q_bp_b}
\end{multline}
where the first sum runs over non-negative integers $j,k,p,q$ such
that
\beglab{eqfirstsum}
j+k\ge2, \quad  p+q\ge2, \quad  j+p=r, \quad  k+q=s,
\edla
the second sum over non-negative integers $n,c$ and
$r_i,s_i$ ($i=1,\ldots,c$) such that
\beglab{eqsecondsum}
c\ge 1, \quad
n+c\ge 2, \quad
r_i+s_i\ge 3, \quad
n+\sum_{i=1}^c r_i=r, \quad
n+\sum_{i=1}^c s_i=s
\edla
and the third sum runs over non-negative integers
$a,b,\ell,m,j_\al,k_\al$ ($\al=1,\ldots,a$) and $p_\be,q_\be$ ($\be=1,\ldots,b$) such that
\begin{multline}  \label{eqthirdsum}
  a+b\ge 1,\quad
\ell+a+m+b\ge 3, \quad
  j_\al+k_\al\ge 2, \quad
  p_\be+q_\be\ge 2,\\[1ex]
\ell+\sum_{\al=1}^a j_\alpha+\sum_{\be=1}^b p_\beta=r, \quad
m+\sum_{\al=1}^a k_\alpha+\sum_{\be=1}^b q_\beta=s.
\end{multline}
\end{lemma}

\begin{proof}
Expand the various constituents of~\eqref{eqdefAlowBlow}.
On the one hand, the identities
  \[
    \frac{1}{c!}\pa^c_R \big( R^{n+c} \big) = {n+c\choose c} R^n,
    \qquad  
    \frac{1}{a!}\frac{1}{b!}\pa^a_z\pa^b_{\bar z}
    \big( z^{\ell+a} \bar z^{m+b} \big) =
    {\ell+a\choose a}{m+b\choose b}z^\ell \bar z^m
\]
allow us to express $\frac{1}{c!}\Ga_*^{(c)}\circ\nu$ and
$\frac{1}{a!b!}\pa^a_z\pa^b_{\bar z}L_*$;
on the other hand, one can write
\[
  F_{*}^a \ov F_*^{\raisebox{-.5ex}{$\scriptstyle b$}}
  =\sum F_{j_1k_1}\cdots F_{j_ak_a}\ov
F_{q_1p_1}\cdots \ov
F_{q_bp_b}z^{j_1+\cdots+j_a+p_1+\cdots+p_b}\bar z^{k_1+\cdots+k_a+q_1+\cdots+q_b}
\]
with a sum over % the set of indices satisfying the inequalities
all non-negative integers
$j_\al,k_\al$ ($\al=1,\ldots,a$) and $p_\be,q_\be$ ($\be=1,\ldots,b$) such that
$  j_\al+k_\al\ge 2$ and $p_\be+q_\be\ge 2$,
%
% $$
% j_1+k_1\ge 2,\ldots, j_a+k_a\ge 2,\quad p_1+q_1\ge 2,\ldots, p_b+q_b\ge 2,
% $$
%
and similarly for $L_*^c$.
\end{proof}

\begin{proposition}\label{degL}
For $F=\la z + \sum_{j+k\ge 2}F_{jk}z^j\bar z^k\in \Fom$, consider the $F$-admissible pair $(L_F^*,\Ga_F^*)$
  where~$L_F^*$ is the resonant-free $F$-admissible series of Definition~\ref{defresfreeLFst}
  and use the notation
  \[
    L_F^*(z) = z\bar z + \sum_{r+s\ge3} L_{rs}^* z^r\bar z^s,
    \qquad
    \Ga_F^*(R) = R + \sum_{n\ge2} \Ga_n^* R^n.
    \]  

      \noindent (i)
    For each $r,s\ge0$ such that $r+s\ge3$ and $n\ge2$,
    $L_{rs}^*$ is a polynomial function of $(F_{jk})_{j+k<r+s}$
    and $\Ga_n^*$ is a polynomial function of $(F_{jk})_{j+k<2n}$.   

    \medskip
    
     \noindent (ii)
    For any affine one-parameter family $F_t=(1-t)F_0+tF_1$ in~$\Fom$,
  %
  % For any sequence $\big(\rho_n(t)\big)_{n\ge2}$ of real polynomials
  % such that $\deg(\rho_n) \le 2n-2$,
  % %
  % the $F_t$-admissible series~$L_t$ of resonant part
  % $\sum_{n\ge2} \rho_n(t) R^n$ determined by Corollary~\ref{determined} form a
  %
    the corresponding families $(L_{F_t}^*)$ and $(\Ga_{F_t}^*)$ are
    regular-polynomial, with % polynomially parametric family.
    \begla
    \deg_t(L_{rs}^*) \le r+s-2, \qquad \deg_t(\Ga_n^*)\le 2n-2.
    \edla  
%
%   Any solution $(\Gamma_t,L_t)$ of the equation
%   $\Gamma_t\circ L_t=L_t\circ F_t$ such that the arbitrary diagonal
%   coefficients $L_{rr}(t)$ are chosen to be polynomials in $t$ of
%   degree not higher than $2r-2$, is such that
% $$\forall (r,s),\; \hbox{deg}_tL_{rs}(t)\le r+s-2.$$
%
\end{proposition}

\begin{proof}
We will prove by induction on $N\ge3$ that~(i) and~(ii) hold for all
$r,s\ge0$ and $n\ge2$ such that $r+s\le N$ and $2n\le N$.
%
% \marginlabel{The term $\Ga_{n+c}^*$ had been overlooked in previous versions}
%
To that end, we use the formulas obtained at the end of the proof of
Lemma~\ref{couple} together with those of Lemma~\ref{lemrevisitgrs}.

For $N=3$, (i) and~(ii) hold because necessarily $r\neq s$ and
$L_{rs}^* = (1-\la^{r-s})\ii (\la \ov F_{s,r-1}+\bar\la F_{r,s-1})$
(and there is no $n\ge2$ such that $2n\le 3$).

Suppose that $N\ge4$ and $r+s=N$. We have
\beglab{eqLrsst}
r =
s \enspace\Longrightarrow \enspace
L_{rs}^* = 0, \qquad
r\neq s \enspace\Longrightarrow \enspace
L_{rs}^* = \frac{\la \ov F_{s,r-1}+\bar\la F_{r,s-1}+ G_{rs}^*}{1-\la^{r-s}}
\edla
and, if $N$ is even,
\beglab{eqGanst}
2n=N \enspace\Longrightarrow \enspace
\Ga_n^* = \la \ov F_{n,n-1}+\bar\la F_{n,n-1}+ G_{nn}^*
\edla
where $G_{rs}^*$ and~$G_{nn}^*$ are obtained by evaluating $G_{rs}$ and~$G_{nn}$ on
  $\big(L_{jk}^*\big)_{j+k<N}$, $\big(\Ga_n^*\big)_{2n<N}$ and
  $\big(F_{jk}\big)_{j+k<N-1}$,
  which according to~\eqref{eqgrssumofthree} yields sums of three terms.
  The induction assumption shows that each of these three terms is a
  polynomial in $\big(F_{jk}\big)_{j+k<N-1}$.
  Indeed, for the first one~\eqref{eqfirstsum} gives
  $j+k+p+q=N$, whence $j+k$ and $p+q\le N-2$.
  For the second one,~\eqref{eqsecondsum} gives
  $N=2n + \sum_{i=1}^c (r_i+s_i) \ge 2n+3c\ge 2n+2c+1$,
  whence $2(n+c)\le N-1$
  and each $r_i+s_i$ is at most $N-2n\le N-2$ if $c=1$ or $N-2n-3\le
  N-3$ if $c\ge2$.
  For the third one,~\eqref{eqthirdsum} gives
  $N=\ell+m + \sum_{\al=1}^a (j_\al+k_\al) + \sum_{\be=1}^b
  (p_\be+q_\be) \ge \ell+m+2a+2b\ge \ell+m+a+b+1$,
  whence $\ell+m+a+b\le N-1$
  and each $j_\al+k_\al$ is at most $N-\ell-m\le N-2$ if $(a,b)=(1,0)$ or $N-\ell-m-3\le
  N-3$ otherwise, and similarly for the $p_\be+q_\be$.

We thus get~(i) at rank~$N$ from~\eqref{eqLrsst} and~\eqref{eqGanst}.

As for~(ii), this results from an easy computation when specializing
the above polynomials to the case of the affine family~$(F_t)$, since
$\deg_t(F_{jk})\le1$ then.
%
% one deduces from the above relations that the
% degree in $t$ of $L_{rs}(t)$ is bounded on the one hand by
% %
% $$r_1+s_1-2+\cdots+r_c+s_c-2=r+s-2n-2c\le r+s-4,$$
% and on the other hand by 
% %
% \[
%   %
%   l+a+m+b-2+a+b=l+m+2a+2b-2\le
%   l+m+\sum{(j_\alpha+k_\alpha)}+\sum{(p_\beta+q_\beta)}-2=r+s-2.
%   %
% \]
%
\end{proof}

% This lemma applies immediately to the resonant-free $F_t$-admissible series.
%
% choice $L_{rr}=0$ (no resonant part).
%
% For the balanced series $L_{F_t}$, characterized by
% $L_{F_t}(z,z)=-z\tau_{F_t}(z)$,
% %
% we first need to study the dependence of $\tau_{F_t}$ on~$t$.  

\medskip

%%%%%%%%%%%%%%%%%%%%%%%%%%%%%%%%%%%%%%%%%%%%%%
%%%%%%%%%%%%%%%%%%%%%%%%%%%%%%%%%%%%%%%%%%%%%%

\noindent  {\bf  Polynomial dependence of the coefficients
  of the foliation involution} %  on the parameter $t$
\medskip
%%%%%%%%%%%%%%%%%%%%%%%%%%%%%%%%%%%%%%%%%%%%%%
%%%%%%%%%%%%%%%%%%%%%%%%%%%%%%%%%%%%%%%%%%%%%%

Let $F\in\Fom$. % Let $L$ be $F$-admissible.
In view of Corollary~\ref{uniqueL} and Lemma~\ref{lemdefinvol},
the involution~$\tau_F$ can be obtained from any $F$-admissible
series~$L$ by the equation $\Lambda(\tau_F(z))=\Lambda(z)$, where
$\Lambda(z) \defeq L(z,z)$.
We choose $L=L_F^*$ as in Proposition~\ref{degL} and % to be the resonant-free $F$-admissible series and
%
% Recall that by definition $\Lambda(\tau_F(z))=\Lambda(z)$ where,
% $\tau_F$ being in fact independent on the choice of $L$
% (Lemma~\ref{lemdefinvol}), we can choose the resonant part
% $\rho(R)=\sum_{k\ge 2}L_{kk}R^k$ to vanish identically.
%
introduce the notations
\beglab{eqnotaLa}
  \Lambda(z)=z^2+\Lambda_*(z)=z^2+\sum_{n\ge 3}\Lambda_nz^n,\qquad
  \tau_F(z)=-z+\tau_*(z)=-z+\sum_{n\ge 2}\tau_n z^n.
  \edla
  We thus get % Developping the identity defining $\tau_F$ we get
  $$\bigl(-z+\tau_*(z)\bigr)^2+\Lambda_*(-z+\tau_*(z))=z^2+\Lambda_*(z),$$
  that is
  %
%  \David{} \marginlabel{A few more details given}
  %
$$2z\tau_*(z)=\Lambda_*(-z+\tau_*(z))-\Lambda_*(z)+\tau_*(z)^2=\Lambda_*(-z)-\Lambda_*(z)+\tau_*(z)^2+
\sum_{r\ge 1}\frac{1}{r!}\Lambda_*^{(r)}(-z)\tau_*^r(z),$$
%
% \Alain{Give some details on how to get the formula ?}
%
whence we derive induction formulas: equating the
coefficients of~$z^n$ in the leftmost and rightmost sides of this
identity and dividing by~$2$, for every $n\ge3$, we get
\begin{equation}\label{5}
\begin{split}
  \tau_{n-1}=&-\epsilon_n\Lambda_n +
  \frac{1}{2}\sum_{ n_1,n_2\ge2 \atop n_1+n_2=n}\tau_{n_1}\tau_{n_2}\\[1ex]
  &+\frac{1}{2}\sum_{r\ge 1} \, \sum_{n_0\ge 0\atop n_0+r\ge 3} \,
  \sum_{n_1,\ldots,n_r\ge2 \atop n_0+\sum n_i=n}
  (-1)^{n_0}{n_0+r\choose r} \Lambda_{n_0+r}\tau_{n_1}\cdots\tau_{n_r}
\end{split}
\end{equation}
where $\epsilon_n\defeq0$ if $n$ is even, $\epsilon_n\defeq1$ if $n$ is odd.
Note that
\begla
\Lambda_n = \sum_{r+s=n} L_{rs}^*
\quad \text{for every $n\ge3$.}
\edla
We deduce a result parallel to Proposition~\ref{degL}:

% \noindent
%
%Notice that the fact that $\epsilon_n=0$ when $n$ is even forces the existence of non-zero coefficients $ F_{kl}$ with $k+l$ odd when constructing examples {\it \`a la Siegel} of $F$ for which $\tau_F$ diverges. In other words, such constructions lead necessarily to non-odd $F$. In fact, 
%Lemma \ref{Fodd} shows that $\tau_F$ always converges when $F$ is odd.

\begin{proposition} \label{degtaun}
 (i)
  The coefficient~$\tau_n$ of~$z^n$ in the involution~$\tau_F$ is a
  polynomial function of $(F_{jk})_{j+k\le n}$.
  
  \medskip

\noindent (ii)
  For any affine one-parameter family $F_t=(1-t)F_0+tF_1$ in~$\Fom$,
  the corresponding family of involutions $(\tau_{F_t})$ is
  regular-polynomial, with
%
  % The degree in the affine parameter $t$ of the coefficient $\tau_n(t)$ in
  %
  % $\tau_{F_t}(z)=-z+\sum_{n\ge 2}\tau_n(t)z^n$ are polynomials in~$t$
  % such that
  %
  \[
    \deg_t(\tau_n)\le n-1.
  \]   
\end{proposition}

\begin{proof}
  We prove (i) and~(ii) by induction on $n\ge2$.
  For $n=2$, this results from $\tau_2 =  -\Lambda_3=
  -\sum_{r+s=3} L_{rs}^*$ and Proposition~\ref{degL}.

For $n\ge3$, we use~\eqref{5} to write~$\tau_n$ in the form of a sum
of three terms, and we observe that each of them satisfies~(i) and~(ii).
%  
% By induction on $n$, let us suppose that the conclusion holds
%   for all integers $n'<n$ and look at each term of the
%   formula~\eqref{5} for~$\tau_{n-1}$.
%   %
%   Proposition~\ref{degL} shows that the coefficients~$\Lambda_n$ are
%   polynomials in~$t$ with
% %  which applies because of our choice of $L$ without resonant term,
%   $\deg(\Lambda_n)\le n-2$.
%   %
%   For the second term, the degree in~$t$ of the summands is at most $(n_1-1)+(n_2-1)=n-2$. For the last
%   term, this degree is at most $n_0+r-2+\sum{(n_i-1)}=n-2$.
  %
\end{proof}

\medskip

%%%%%%%%%%%%%%%%%%%%%%%%%%%%%%%%%%%%%%%%%%%%%%
%%%%%%%%%%%%%%%%%%%%%%%%%%%%%%%%%%%%%%%%%%%%%%

\noindent  {\bf  Polynomial dependence of the coefficients
  of the balanced admissible series} %  on the parameter $t$
\medskip
%%%%%%%%%%%%%%%%%%%%%%%%%%%%%%%%%%%%%%%%%%%%%%
%%%%%%%%%%%%%%%%%%%%%%%%%%%%%%%%%%%%%%%%%%%%%%

\begin{proposition}\label{degLF}
For $F=\la z + \sum_{j+k\ge 2}F_{jk}z^j\bar z^k\in \Fom$, consider the $F$-admissible pair $(L_F,\Ga_F)$
  where~$L_F$ is the balanced $F$-admissible series defined by
  Proposition~\ref{unique}
  and use the notation
  \[
    L_F(z) = z\bar z + \sum_{r+s\ge3} L_{rs} z^r\bar z^s,
    \qquad
    \Ga_F(R) = R + \sum_{n\ge2} \Ga_n R^n.
    \]   

  \noindent (i)
    For each $r,s\ge0$ such that $r+s\ge3$ and $n\ge2$,
    $L_{rs}$ is a polynomial function of $(F_{jk})_{j+k<r+s}$
    and $\Ga_n$ is a polynomial function of $(F_{jk})_{j+k<2n}$.
    Moreover,
    \beglab{eqdominantF}
    r\neq s \enspace\Longrightarrow\enspace
    (1-\la^{r-s}) L_{rs} = \la \ov F_{s,r-1}+\bar\la F_{r,s-1}+G_{rs},
    \edla
    where $G_{rs}$ is a polynomial function of $(F_{jk})_{j+k\le
      r+s-2}$.  

    \medskip
    
 \noindent (ii)
    For any affine one-parameter family $F_t=(1-t)F_0+tF_1$ in~$\Fom$,
    the corresponding families $(L_{F_t})$ and $(\Ga_{F_t})$ are
    regular-polynomial, with % polynomially parametric family.
    \begla
    \deg_t(L_{rs}) \le r+s-2, \qquad \deg_t(\Ga_n)\le 2n-2.
    \edla   
\end{proposition}

\begin{proof}
  Follow the lines of the proof of Proposition~\ref{degL}.
%\color{blue}
  The only difference is that,
  after dealing with the case $r\neq s$ (by the same kind of
  computations, based on the second part of~\eqref{eqLrsst}),
  one handles the case $r=s=n$ (which entails that~$N$ is even) by
  replacing the first part of~\eqref{eqLrsst} with the formula
\[
  L_{nn} = -\tau_{n-1} - \sum_{r+s=2n\atop r\neq s} L_{rs}
\]
and making use of Proposition~\ref{degtaun}. %\color{black}
\end{proof}

%\color{blue}

\begin{remark}
  Similarly,
  for any regular-polynomial family $(\rho_t)$ in $R^2\R[[R]]$, the
  family $(L_t)$ of $F$-admissible series obtained from
  Corollary~\ref{determined} is regular-polynomial.
\end{remark}

\subsection{Proof of Theorem~\ref{main}}\label{secpfthmmain}
Proposition~\ref{degLF} shows that the Ilyashenko - P\'erez-Marco
alternative holds true in our setting:
we can apply Theorem~\ref{thmIPM} with $K(F)=L_F$ and
  $\mathcal{F}=\Fomcv$ or $\FomJcv$.
It only remains to show that in each case, there exists at least one
example~$F$ such that~$L_F$ is divergent.
%
% (and thus also every $F$-admissible series and every geometric
% normalization of~$F$ by Proposition~\ref{LFbasic}).

\medskip

{\bf Proof of (i)}. Corollary~\ref{CorGenecand}
  with $J(z)=\la z$ ($N=1$)
  yields a non-integrable area-preserving $F\in\Fomcv$,
  whose balanced admissible series~$L_F$ is necessarily divergent
  (otherwise, it would be an analytic non-constant first integral
  of~$F$ by Proposition~\ref{integ}(ii)).
  Then, the IPM alternative entails that the balanced admissible
  series of the generic element of~$\Fomcv$ is divergent.

% To show that for generic $F\in \Fomcv$, the balanced $F$-admissible
% series~$L_F$ is divergent, we take a non-integrable area-preserving
% map~$F_0$ in~$\Fomcv$ as provided by Theorem~\ref{area}:
%  %
%  all $F_0$-admissible~$L$ must be divergent, since an analytic~$L$
%  would be a non-constant first integral by Proposition~\ref{integ}(ii).

 \medskip

 {\bf Proof of (i')}. We first remark that both area-preserving maps
and holomorphic maps in $\Fomcv$ belongs to the set $\{\Gamma=\ID\}$.
To show that the generic~$F$ in $\Fomcv\cap \{\Gamma\neq \ID\}$ has
divergent~$L_F$, in view of~(i), it is enough to show that $\Fomcv\cap
\{\Gamma\neq \ID\}$ is open and dense in~$\Fomcv$.
 This follows easily from the fact that, by~\eqref{dominant},
  the conjugacy equation $L_F\circ F = \Ga_F\circ L_F$ implies
  $\Ga_n = A_{nn}$ and, according to Lemma~\ref{lemrevisitgrs}, a
  perturbation of the coefficient $F_{n,n-1}$
%
% in $F(z)=\lambda z+\sum F_{jk}z^j\bar z^k$
will perturb the coefficient~$\Ga_n$.  
% in $\Gamma(z)=\sum \Gamma_nz^n$.

\medskip

{\bf Proof of (ii)}.
Assuming that~$\om$ is super-Liouville, we will % just need to
construct $F\in\Fomcv$ with prescribed $N$-jet~$J$ such that
$L_F=\sum L_{rs}z^r\bar z^s$ is divergent
% and the IPM alternative will give the conclusion.
%
by mimicking Siegel's method (see \cite{SM} page~189),
exploiting~\eqref{eqdominantF}.

Without loss of generality we assume $\De=1$.
Let~$(n_p)$ denote an increasing sequence of natural numbers such that
$|\lambda^{n_p}-1|^{-1}\geq n_p!$.
We construct the coefficients $F_{rs}$ by induction on $r+s$.
We take them as prescribed by~$J$ for $r+s\le N$, 
%
% 1. From the fact that $L_{rs}=\ov L_{sr}$ we can take into account the
% corresponding recurrence relation only for $r+s=n,\; r>s$, hence
% $F_{r,s-1}$ and $F_{s,r-1}$ appear each in only one of these
% relations.
%
% 2. Let $F_{rs}$ be given by the prescribed $N$-jet for $r+s\leq
% N$.
%
and for $r+s>N$ we choose $F_{rs}=0$ except for $(r,s)$ of the form
$(1,n_p)$ or $(n_p+1,0)$ with $n_p\ge N$, in which case the value will
be~$\pm1$.
We choose $F_{1,n_p}$ and $F_{n_p+1,0}$ by writing~\eqref{eqdominantF}, % (\ref{dominant2}),
\[(1-\lambda^{n_p})L_{n_p+1,1}=\lambda
  \ov{F}_{1,n_p}+\bar{\lambda}F_{n_p+1,0}+G_{n_p+1,1},\]
where $G_{n_p+1,1}$ is determined by coefficients~$F_{jk}$ with
$j+k\le n_p$ and thus already chosen,
%
% for $r+s<n_p+1$.  By induction, we freely choose
%
% We thus choose $F_{1,n_p}$ and $F_{n_p+1,0}$ in $\{-1,+1\}$ 
%$\lambda \bar{F}_{1,n_p}+\bar{\lambda}F_{n_p+1,0}$ to be~$\pm 1$
%
and arranging for the inequality
\begin{equation}
  |\lambda \ov{F}_{1,n_p}+\bar{\lambda}F_{n_p+1,0}+G_{n_p+1,1}|
  \ge 2 |\Re \la|
\end{equation}
to hold.
For that, % since $\Re\la=\cos(2\pi\om)\neq0$,
  we may take $F_{1,n_p} = F_{n_p+1,0} = 1$ if
  $(\Re\la)\Re(G_{n_p+1,1})\ge0$ and $F_{1,n_p} = F_{n_p+1,0} = -1$
  otherwise.
  %
% according to the sign of the real part of $G_{n_p+1,1}$, e.g. when
% $\Re(G_{n_p+1,1})\ge0$, take $F_{1,n_p} = F_{n_p+1,0} = 1$ if
% $\Re\la>0$ and $-1$ otherwise.

   With such~$F$, we get $|L_{n_p+1,1}|\geq 2 n_p! |\cos(2\pi\om)|$ for
   all~$p$ large enough, therefore~$L_F$ is divergent;
   by the IPM alternative, the balanced admissible
  series of the generic element of~$\FomJcv$ is divergent.

\medskip

{\bf Proof of (iii)}.
Assuming that $\omega$ is non-Bruno, and given
$J(z)=\lambda z+\sum_{2\leq k\leq N} J_kz^k$, we just need to
construct a non-linearizable~$F$ holomorphic for~$z$ in the
  disc $\{ |z|<\De\}$ with $N$-jet equal to~$J$.
%
% We will make use of \red{Geyer's criterion (???)
% (Theorem~\ref{Ge})}.
%
  We will make use of the polynomial example of Corollary~\ref{CorGe} that we
  have deduced from Geyer's result.

  Since~$\om$ is irrational, there exists a tangent-to-identity
  polynomial $\Phi(z)=z+\sum \Phi_j z^j$ such that
\beglab{eqlinJN}
  J^*(z) \defeq \Phi\circ J\circ\Phi\ii=\la z+O(z^{N+1}).
  %
%  \Phi\circ J\circ\Phi^{-1}=\lambda z+\epsilon(z)=:J^*(z)
  %
\edla
%
% where $\epsilon(z)$ is a polynomial of order $\geq N+1$.
%
% Let $\rho>0$ denote the radius of convergence of~$\Phi\ii$.
% %
% If $\rho<1$, then we replace~$\Phi$ by $\rho\ii\Phi$;
% %
% since $\rho\ii J^*(\rho z)= \la z+O(z^{N+1})$, we can thus assume
% without loss of generality that~\eqref{eqlinJN} holds with a
% polynomial~$\Phi$ whose inverse $\Psi\defeq\Phi\ii$ is holomorphic in
% the unit disc.
%
  Consider $F^*(z) \defeq \la z+\la z^{N+1}$,
  $F\defeq\Phi\ii\circ F^*\circ \Phi$
  and let~$\De_0$ denote the radius of convergence of~$F$.
By Corollary~\ref{CorGe}, $F^*$ is not linearizable, hence neither is~$F$,
which implies that, for any $\De\le\De_0$,~$F$ can be viewed as an
element of~$\Fomcv$ for which the unique invariant foliation~$\cF_F$ is
not analytic (by Proposition~\ref{integ}(i), as explained at the
beginning of Section~\ref{secClassEx}(2)) and~$L_F$ is divergent.

  On the other hand,
   one easily checks that when two maps coincide modulo $O(z^{N+1})$, as~$F^*$ and~$J^*$ do,
  the same is true for their compositions to the right by~$\Phi$ or to
  the left by~$\Phi\ii$,
  thus $F=\Phi\ii\circ F^*\circ \Phi$ and $J=\Phi\ii\circ J^*\circ \Phi$ coincide modulo $O(z^{N+1})$.
%
%   $F^* = J^* + O(z^{N+1})$
%   %
%   and this will allow us to check that the $N$-jet of~$F$ is
%   $\Phi\ii\circ J^*\circ \Phi = J$.
% %
% Indeed, $F^*\circ\Phi = f + \de$ with $f\defeq
% J^*\circ\Phi=\Phi\circ J$ and
% $\de(z)=O(z^{N+1})$,
% %
% whence
% %
% $F - J = \Phi\ii\circ(f + \de) - \Phi\ii\circ f =
% \sum_{r\ge1}\frac1{r!}(\Phi\ii)^{(r)}\de^r$ has order at least $N+1$,
%
    We thus have found, for each $\De\le\De_0$, an $F\in\FomJcv$
  % without any convergent geometric normalization
with divergent~$L_F$;
   the IPM alternative entails that the balanced admissible
  series of the generic element of~$\FomJcv$ is divergent.

% We conclude by showing that the $N$-jet of~$F$ is~$J$.
% %
% Indeed,
% \[
%   %
%   F^*(z)=J^*(z)-\epsilon(z)+\lambda z^{N+1}=:J^*(z)+\delta(z),
%   %
% \]
% %
% where $\delta(z)$ is of order $\geq N+1$, hence
% %
% \[
%   %
%   F=\Phi^{-1}\circ (J^*+\delta)\circ \Phi=\Phi^{-1}\circ J^*\circ
%   \Phi+R=J+R
%   %
% \]
% where the remainder $R=(\Phi^{-1})'\delta\circ \Phi+\cdots$ is of order $\geq N+1$.

\medskip

{\bf Proof of (iv)}.    For any such prescribed $N$-jet~$J$,
% Theorem~\ref{area}
Corollary~\ref{CorGenecand} allows us to find a non-integrable
  area-preserving map~$F$ in~$\FomJcv$, thus with~$L_F$ divergent,
and the IPM alternative then entails that the balanced admissible
  series of the generic element of~$\FomJcv$ is divergent.

\medskip

%%%%%%%%%%%%%%%%%%%%%%%%%%%%%%%%%%%%%%%%%%%%%%
%%%%%%%%%%%%%%%%%%%%%%%%%%%%%%%%%%%%%%%%%%%%%%

\subsection{Generic divergence of the foliation involution~$\tau_F$}
\label{sectaud}

We have seen in Lemma~\ref{Fodd} that the involution~$\tau_F$ is
  automatically convergent in a case for which we will later show the
  generic divergence of all admissible series.
%
% So, the divergence of all $F$-admissible series has no reason to
% imply the divergence of~$\tau_F$.
%
Nevertheless, we can also prove

\begin{ourthm}   \label{thmtaud}
  Assume that $\omega$  satisfies the following ``odd super-Liouville condition'':
  \begin{equation}  \label{eqoddsuperL}
    \text{there exist infinitely many odd}\; k\in\Zp \;\text{such that}\;
    \bigl|\lambda^k-1\bigr|^{-1}\ge k!
\end{equation}
(where $\la=e^{2\pi i\omega}$ as usual).
Then, for any $\De>0$, the generic~$F$ in~$\Fomcv$ has a divergent~$\tau_F$.
\end{ourthm}

An adaptation of the proof of the density of the super-Liouville
set~\eqref{eqdefsuperL} will show that the smaller ``odd
super-Liouville set'' defined by~\eqref{eqoddsuperL} is dense in~$\R$
too:

\begin{proposition}  \label{propoddsuperLdense}
  The set of odd super-Liouville numbers, \ie the irrational $\om$'s
  satisfying~\eqref{eqoddsuperL}, is dense in~$\R$.
\end{proposition}

\begin{proof}[Proof of Theorem~\ref{thmtaud}]
% Let $\delta_n:=1-\lambda^n$ and $\theta_n=\arg (\delta_n)$. 
%
  %\color{blue}
  Assume that $\omega$ is odd super-Liouville and pick an increasing
  % that is, there exists a
  %
  sequence of positive odd integers $(n_p)$ such that
  \beglab{ineqlanpnp}
  |1-\la^{n_p}|\le 1/n_p! \quad \text{for all $p\ge1$.}
  \edla
  %
  % sequence of odd integers $\{2m_k-1\}$ such that $|1-\lambda^{2m_k-1}|\leq 1/(2m_k-1)!$.
  %
  We will construct an example of $F\in\Fomcv$ for which the foliation
  involution~$\tau_F$ is divergent. Without loss of generality we
  assume $\De=1$.
%
  % Since $\lambda$ and $1$ lie on the unit circle, it follows that the
  % argument $\theta_k:=\arg(1-\lambda^{2m_k-1})$ goes to zero as $k$
  % goes to infinity. 
 %$|\delta_{2m_k-1}|<1/(2m_k-1)!$, from which $\theta_{2m_k-1}\to 0$. 
%
The coefficients~$F_{jk}$ of~$F$ will all be zero
except for $(j,k)$ of the form $(1,n_p)$ or $(n_p+1,0)$, in which case
we will choose the value to be~$\pm i$.

For arbitrary $F\in\Fom$, using the notations~\eqref{eqnotaLa} with
$\La(z)\defeq L_F^*(z,z)$,
the coefficients of~$\tau_F$ are related to those of~$F$
according to formula~\eqref{5}, which involves the coefficients
$\La_n = \sum_{r+s=n} L^*_{rs}$, $n\ge3$.
In particular, for $n\ge1$ odd, $\eps_{n+2}=1$ and
\beglab{eqdeftaunp}
  \tau_{n+1}=-\La_{n+2}+ H_{n+1}, \qquad H_{n+1} = \text{polynomial expression in
    $\tau_2,\ldots,\tau_n,\La_3, \ldots, \La_{n+1}$.} 
\edla
%
% where $H_n$ is a polynomial of $(F_{jk})_{j+k<n-1}$.
%
Recall that $\tau_{n+1}\in\R$ and, since $L_F\in\gL$, $\La_{n+2}\in\R$ too.
We will arrange for
\beglab{ineqaim}
|\tau_{n_p+1}| \ge 2 n_p! |\cos(2\pi\om)|
\edla
to hold for all~$p$ large enough
by choosing inductively the $F_{jk}$'s with $j+k\le n_p+1$.
% $F_{1,n_p}$ and $F_{n_p+1,0}$ as follows.

For $2\le j+k\le n_1$, we choose $F_{jk}\defeq 0$.
Suppose now that $p\ge1$ and the coefficients $F_{jk}$, $2\le j+k\le n_p$, have already
been chosen.
Proposition~\ref{degL} and its proof (especially~\eqref{eqLrsst}) show
that
\begla
\La_{n_p+2} =
\sum_{r+s = n_p+2}
\frac{\la \ov F_{s,r-1}+\bar\la F_{r,s-1}}{1-\la^{r-s}}
+ K_p,
\qquad K_p \defeq \sum_{r+s = n_p+2}
\frac{G_{rs}^*}{1-\la^{r-s}},
\edla
where~$K_p$ is determined by the $F_{jk}$'s, $j+k\le n_p$.
The sum over $(r,s)$ in the \rhs\ of the equation giving~$\La_{n_p+2}$ is real, due to the
symmetry $(r,s)\mapsto (s,r)$,
thus $K_p \in \R$ too and
\begla
\La_{n_p+2} = 
2\Re\bigg( \sum_{r+s = n_p+2\atop r>s}
\frac{\la \ov F_{s,r-1}+\bar\la F_{r,s-1}}{1-\la^{r-s}}
\bigg)
+ K_p.
\edla
We choose $F_{jk}\defeq0$ if
$j+k=n_p+1$ and $(j,k)\notin\{(1,n_p),(n_p+1,0)\}$, thus the
terms with $(r,s)\neq(n_p+1,1)$ do not contribute and
\begla
\La_{n_p+2} =
2\Re\Big( 
\frac{\la \ov F_{1,n_p}+\bar\la F_{n_p+1,0}}{1-\la^{n_p}}
\Big)
+ K_p.
\edla
Inserting this in~\eqref{eqdeftaunp}, we get
%
%\color{black}
%
\begla
\tau_{n_p+1} = 2\Re\Big( 
\frac{\la \ov F_{1,n_p}+\bar\la F_{n_p+1,0}}{\xi_p}
\Big)
+ I_p
\quad\text{with} \ens
\xi_p \defeq \la^{n_p}-1
\edla
and $I_p \defeq H_{n_p+1}-K_p\in\R$ is determined by the $F_{jk}$'s, $j+k\le n_p$.
%
% \Alain{where the $G_{kl}$ depend on the $F_{i,j}$ with $i+j\le n-2$
% and the $\tilde H_n$ depend on the $F_{i,j}$ with $i+j\le n-1$ with
% the exception of $F_{1,n-2}$ and $F_{n-1,0}$ ...GIVE DETAILS ?}
%

% We construct $F$ as follows: let $F_{kl}=0$, except for
% $(k,l)=(2m_k,0)$ and \Alain{$(2m_k-1,1)$}, in which case
%
% the value will be $\pm 1$ according to the sign of the real part of $\tilde{H}_{2m_k+1}$ so that 
% \[|\Re\{\la \ov F_{2m_k,0}+\bar{\la}F_{1,2m_k-1}+\tilde{H}_{2m_k+1}\}|\geq 2|\Re\la|.\]
% e.g. when $\Re(\tilde{H}_{2m_k+1})\geq 0$, take $F_{2m_k,0}=F_{2m_k-1,1}=1$ if $\Re\la>0$ and $-1$ otherwise.
%
% we arrange for the inequality
% %
% \[
%   %
%   |\Re\{\la \ov F_{2m_k,0} + \bar{\la}F_{1,2m_k-1} + \tilde{H}_{2m_k+1}\}|
%   %
%   \ge 2|\Re\la|
%   %
% \]
% %
% to hold by taking $F_{2m_k,0} = F_{1,2m_k-1} = 1$ if
%   $(\Re\la)\Re(\tilde{H}_{2m_k+1})\ge0$ and $F_{2m_k,0} = F_{1,2m_k-1} = -1$
%   otherwise.
% %
% \David{CHECK INDICES AND CORRECT END OF PROOF}
  
% %By induction, we choose $\lambda \ov F_{2m_k,0}+\bar{\lambda}F_{1,2m_k-1}=\pm 1$ according to the sign of the real part of 
% %
% As $\theta_k\to 0$, for $k$ large enough, we  have 
% \[|Re\{\lambda \ov F_{2m_k,0}+\bar{\lambda}F_{1,2m_k-1}+\tilde{H}_{2m_k+1})e^{-i\theta_{k}}\}|\geq |\Re\la|,\]
% hence
% $|\tau_{2m_k}|\geq 2|\Re\la| (2m_k-1)!$

We set $F_{1,n_p} = -i u_p$ and $F_{n_{p+1},0} = i u_p$ with~$u_p$ to be chosen in
$\{-1,+1\}$, so
\begla
\tau_{n_p+1} = 4 u_p (\Re\la)\Re(i/\xi_p) + I_p.
\edla
%
%\color{blue}
%
Recall that $|\xi_p|\le 1/n_p!$ by~\eqref{ineqlanpnp}. Observing that
$1 = |\la^{n_p}|^2 = |\xi_p+1|^2 = |\xi_p|^2 +2\Re\xi_p + 1$,
whence
\begla
\Re\xi_p = -\tfrac12|\xi_p|^2, \quad
\Im\xi_p = v_p |\xi_p| \big(1-\tfrac14|\xi_p|^2)^{1/2}
\ens \text{with $v_p=\pm1$,}
\edla
we get
\begla
\tau_{n_p+1} = 4 u_p v_p (\Re\la)|\xi_p|\ii\big(1-\tfrac14|\xi_p|^2)^{1/2}  + I_p
\edla
and we choose
\begla
u_p\defeq v_p \ens\text{if}\ens (\Re\la)I_p\ge0,
\qquad
u_p\defeq -v_p \ens\text{otherwise.}
\edla
This way we always have
\begla
|\tau_{n_p+1}| \ge 4 (\Re\la)|\xi_p|\ii\big(1-\tfrac14|\xi_p|^2)^{1/2},
\edla
whence~\eqref{ineqaim} follows for~$p$ large enough
and, for~$F$ thus constructed, $\tau_F$ is divergent.
% when $\omega$ is odd super-Liouville.

We now conclude by the IPM alternative, thanks to
Proposition~\ref{degtaun}, that the foliation involution of the
generic element of~$\Fomcv$ is divergent.
  %
%\color{black}
%
\end{proof}

%%%%%%%%%%%%%%%%%%%%%%%%%%%%%%%%%%%%%%%%%%%%%%
%%%%%%%%%%%%%%%%%%%%%%%%%%%%%%%%%%%%%%%%%%%%%%

\begin{proof}[Proof of Proposition~\ref{propoddsuperLdense}]
We show that the set of odd super-Liouville numbers~\eqref{eqoddsuperL} is dense in~$\R$
by modifying Siegel's proof \cite[pp.~189--190]{SM} that the set of
super-Liouville numbers~\eqref{eqdefsuperL} is dense.

Set $\lambda=e^{2\pi i \omega}$. For each $n\in\Z_{>0}$, choose
$m_{\om,n}\in\Z$ so that $\theta \defeq |n\omega-m_{\om,n}|\leq 1/2$,
then
\[
  |1-\lambda^n|=|1-e^{2\pi i n\omega}|=|e^{\pi i n\omega}-e^{-\pi i
    n\omega}|=2|\sin(\pi n\omega)|=2\sin(\pi|n\omega-m_{\om,n}|).
\]
Since $\theta\leq 1/2$, it follows that $2\theta\leq \sin(\pi \theta)\leq \pi \theta$ and therefore
\[4\theta\leq |1-\lambda^n|\leq 2\pi \theta\leq 7\theta.\]
Consequently, it is enough to show that the set irrational numbers $\omega\in(0,1)$ with 
\begin{equation}\label{ome}
  %
  % |n\omega-m_{\om,n}|
  \operatorname{dist}(n\omega,\Z)
  \leq \frac{1}{7n!}
  \quad \text{for infinitely many odd~$n$}
\end{equation}
is dense in~$(0,1)$ (and thus in~$\R$, by $\Z$-periodicity). %  the unit circle.
For each irrational number $0<\omega<1$, consider its continued
fraction representation: one can associate a unique sequence of
natural numbers $r_k, k\geq1$, called partial quotients of $\omega$,
so that the sequence of fractions $p_k/q_k$, $k\geq 0$, recursively
defined according to the prescription
\begin{equation}
\left\{
\begin{aligned}
&&p_0=1,\quad q_0=1,\quad p_1=1,\quad q_1=r_1,\\
&&p_k=r_kp_{k-1}+p_{k-2},\quad q_k=r_kq_{k-1}+q_{k-2},\quad k\geq 2\\
\end{aligned}
\right.
\end{equation}
converges to $\omega$.  Moreover, from the theory of continued
fractions (see e.g.\ \cite{HW}), one has
\begin{equation}\label{omegacf1}
q_kp_{k-1}-p_kq_{k-1}=(-1)^k,\quad k\geq 2.
\end{equation}
\begin{equation}\label{omef}
|q_k\omega-p_k|< \frac{1}{q_{k+1}}<\frac{1}{r_{k+1}q_k}\leq \frac{1}{r_{k+1}}.
\end{equation}
From  (\ref{omegacf1}), and that $p_k, q_k$ are relatively prime, one deduces that for each $k\geq 2$, $q_k$ and $q_{k-1}$ could not both be even.
Conversely, corresponding to each such sequence $\{r_k\}$, there is an irrational number $\omega$  in the interval $(0,1)$ with $\{r_k\}$ as its partial quotients in its continued fraction representation.

Now let $\alpha$ be an arbitrary irrational number in $(0,1)$ with $\{s_k\}$ the partial quotients in its continued fraction representation.

For $\ell$ an arbitrary fixed natural number, one defines 
\[
  r_k\defeq s_k \quad (0<k\leq \ell),
  %
%   \qquad r_{k+1} \defeq 7q_k!+\epsilon_{k+1} \quad (k\geq\ell).
%
  \qquad r_k \defeq 7q_{k-1}!+\epsilon_k \quad (k>\ell),
  \qquad q_{k} \defeq r_kq_{k-1}+q_{k-2} \ens\text{for all~$k$,}
\]
%
% where $\epsilon_k$ are chosen to be $1$ or $0$ so that
% $q_k$ is odd by the recursive relation $q_{k}=r_kq_{k-1}+q_{k-2}$.
where the $\epsilon_k$'s are chosen to be~$1$ or~
$0$ so that~$q_k$ is odd. % by the recursive relation $q_{k}=r_kq_{k-1}+q_{k-2}$.
This is possible since $q_{k-1}$ and $q_{k-2}$ can not both be even.
Denoting by~$\om$ the irrational number determined by these
partial quotients~$r_k$, we have
\[|\alpha-\omega|\leq |\alpha-\frac{p_{\ell}}{q_{\ell}}|+|\omega_\ell-\frac{p_{\ell}}{q_{\ell}}|\leq 2q_{\ell}^{-2}\leq 2\ell^{-2}\]
while, on the other hand, by~\eqref{omef},
there are infinitely many % pairs $n=q_k, n=p_k$ with $n$ odd satisfy (\ref{ome}) .
odd $n=q_k$ that satisfy~\eqref{ome}.
Since~$\ell$ can be chosen arbitrarily large, the corresponding
numbers $\om_\ell=\om$ tend to~$\omega$. This shows
that the set of odd super-Liouville numbers is dense in the unit
inverval.
\end{proof}

\smallskip

\subsection{On general involutions $\tau$}

In Section \ref{realaxis} the involution $\tau_F$ was defined by
looking at the intersection of the foliation with the real axis
$\{\, (z,w)=(z,z) \mid z\in\R \,\}$. More generally, we could have replaced this axis by any
real analytic curve $\mathcal{C}$ in the $z$ plane, image of an
analytic map 
\[
  c \col (\R,0)\to (\C,0)=(\R^2,0), \quad
  u\mapsto c(u)=\sum_{j\ge 1}c_ju^j,\quad c_1\not=0\in\C,
\]
and defined an involution~$\tau$ related to the curve $\cal C$ as follows:
% \Alain{}
defining the one-variable formal series $\Lambda=\Lambda_{\cal C}$ by
\[
  \Lambda_{\cal C}(u)\defeq L(c(u),\ov c(u)),
  \quad\text{where}\ens \ov c(u)=\sum_{j\ge 1}\ov{c_j} \, u^j,
  % \; \hbox{in the notation of Appendix A},
  %
  \]
and checking that $\Lambda_{\cal C}$ has real coefficients, one defines 
$\tau(u)=\tau_{F,\cal C}(u)=-u+O(u^2)$ as the unique formal
involution different from Identity such that $\Lambda\circ \tau=\Lambda$. 
Precisely, $\tau$ can
be defined explicitly by $\tau=\ell^{-1}\circ \sigma\circ \ell$, where
$\ell(u)=\Lambda^{1/2}(u)=(c_1\ov c_1)^{1/2}u+O(u^2)$ is any square root of $\Lambda(u)$.
Exactly as in Section \ref{realaxis} such $\tau$ has real coefficients and depends
only on $F$ and $\cal C$,
whence there exists a balanced solution
$L_{F,{\cal C}}$, depending only on $F$ and~$\cal C$, such that
\[L_{F,{\cal C}}(c(u),\ov c(u))=-u\tau_{F,\cal C}(u),\]
whose divergence implies divergence of all other $F$-admissible solutions $L$.
\smallskip

\smallskip

\medskip
\noindent {\bf Remark.} Although the convergence or divergence of the balanced solution $L_{F,{\cal C}}$ does not depend on the curve $\cal C$,  the convergence or divergence of the involution $\tau_{F,\cal C}$ may depend on the choice of $\cal C$. In particular, if the formal curve ${\cal C}=\{(c(u),\ov {c}(u))\}$ is such that $c(-u)=-c(u)$ and $L(c(u),\ov{c}(u))=
L(-c(u),-\ov{c}(u))$, this implies that 
$\tau_{F,\cal C}(u)=-u$. Such a symmetric formal curve is not in general unique but, as a formal consequence of the uniqueness of the invariant foliation of $F$, any such curve has the same property with respect to any solution $L$ of equation \eqref{conj} and a balanced solution $L_F$ is obtained by requiring that $L_F(c(u),\ov{c}(u))=u^2$. 

% \newpage

\begin{center}
\hskip0.5cm
\includegraphics[scale=.21]{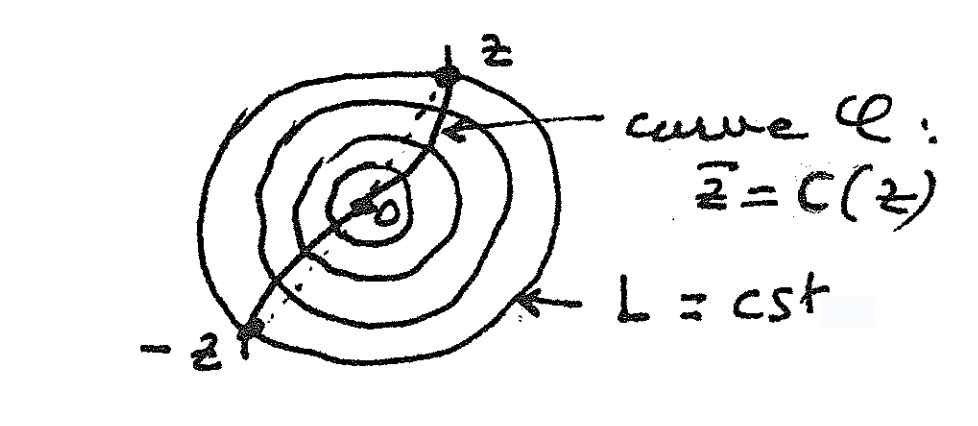}
\end{center}

%\Qiaoling{
%
The next lemma shows the role played by a general curve $\cal C$  in
making a link between the involutions of two formal foliations in
$\gL$. 
% }

%   \Alain{}\marginlabel{\red{in Lemma \ref{compare} and Corollary
%   \ref{corcomp} one could compare more generally $(L,\cal C)$ and
%   $(L',{\cal C}')$ ?}}

\begin{lemma}\label{compare}
For any two $L, L'\in \gL$, the equation $L'(u,u)=L(c(u),\ov{c}(u))$
has a  solution  $c(u)=\sum_{i\geq 1}c_iu^i$, $c_1\neq 0$. As a
consequence,  for any two $F,F'\in\gF$, there exists a curve ${\cal
  C}=\{(c(u),\ov{c}(u))\}$ such that $\tau_{F'}=\tau_{F,\cal C}$;
in particular, there exists a formal curve $\cal C$ such that
    $\tau_{F,{\cal C}}=\sigma$ for given $F\in\gF$.
\end{lemma}

\begin{proof}  We first remark that if $L\in\gL$, then $L(c(z),\ov{c}(w))\in \gL$.
Let $L(z,w)=zw+\sum_{i+j\geq 3}L_{ij}z^iw^j$, and write $L'(z,z)=z^2+\sum_{k\geq 3} a_kz^3$. The equation $L'(u,u)=L(c(u),\ov{c}(u))$ is expressed as
\[u^2+\sum_{k\geq 3}a_ku^k=(\sum_{k\geq 1}c_ku^k)(\sum_{k\geq 1}\overline{c_k}u^k)+\sum_{i+j\geq 3} L_{ij}(\sum_{k\geq 1}c_ku^k)^i(\sum_{k\geq 1}\overline{c_k}u^k)^j\]
By inductions on the degrees of monomials on both sides of the equation, the coefficients $c_k$ can be solved recursively: at degree $n=2$,  $c_1\overline{c_1}=1$; at degree $n$, 
\[c_{n-1}\overline{c_1}+c_1\overline{c_{n-1}}+A_n(c_1,\cdots,c_{n-2})=a_n,\]
where $A_n$ is a function only depends on the former coefficients.
Notice that the solution is not unique.
%We solve recursively
%\[c_{n-1}=a_n-\sum_{i=0}^n L_{ij}\sum_{k_1+\cdots k_j=n-i}c_{k_1}\cdots c_{k_j},\quad n\geq 2.\]

%$(2)$ For any two $F,F'\in \gF$, let $L, L'$ be two admissible solutions for $F,F'$ respectively.  By $(1)$, there exists $c(z)$ such that $L'(z,z)=L(z,c(z))$. Let  ${\cal C}=\{(z,c(z))\}$, it follows that $\tau_{F'}=\tau_{F,\cal C}$ by the definition of $\tau$. Furthermore, by Proposition \ref{conj2},  $\tau_{F'}$ and $\tau_F$ are formally conjugate. 

\end{proof}

% \Qiaoling{
%
When $F$ and $F'$ are conjugate, a curve $\cal C$ as in Lemma
\ref{compare} is just the image of the real axis by the conjugacy :
%
% }

\begin{center}
\hskip0.5cm
\includegraphics[scale=.65]{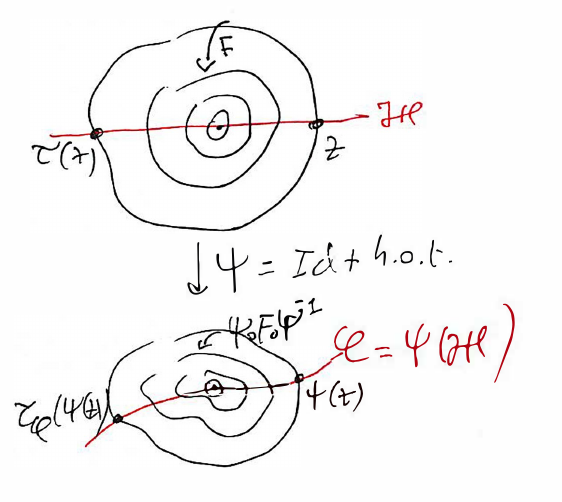}
\end{center}

%\blue{ THE FOLLOWING PARAGRAPH HAD BEEN INSERTED BY QIAOLING ON
%  2024/10/16,  IN SECTION~\ref{realaxis} right after Proposition~\ref{LFbasic}}
%
%\blue{ One can define similar involution by replacing the intersecting line $\{(z,z)\}$ by a general (analytic) curve }
%%
%\begin{equation}\label{curve}{\cal C}=\{(z,c(z))|c(z)=\alpha z+O(z^2),\quad \alpha\in \mathbb{C},\,\alpha\neq 0\}.
%  %
%\end{equation}
%% 
%\blue{  More precisely, an involution $\tau=\tau_{F,{\cal C}}$ is defined uniquely by the invariance 
%\[L(\tau(z),c(\tau(z))=L(z,c(z)),\]
%or $\tau=\ell^{-1}_{{\cal C}}\circ \sigma\circ \ell_{{\cal C}}$, where $\ell_{{\cal C}}(z)=\sqrt{L(z,c(z))}$.
%A balanced solution $L_{F,{\cal C}}$ is then defined by\[L_{F,{\cal C}}(z,c(z))=-z\tau_{F,\cal C}(z),\]
%whose divergence implies divergence of all other $F$-admissible
%solutions $L$. }
%
%\bigskip
%
%\blue{DAVID: Think more and make a decision on that subsection...}
%
%\bigskip
%
%\blue{I restored the following corollary, that had been commented by
%  Qiaoling some time ago...}

\begin{corollary}\label{corcomp} Let $F_1=\Psi^{-1}\circ F\circ \Psi$, where $\Psi$ is a local analytic diffeomorphism tangent to identity, then 
\[\tau_{F_1}=\tau_{F,\cal C},\quad \text{where}\quad  {\cal C}=\{(\Psi(u,u),\tilde{\Psi}(u,u))\}.\]
\end{corollary}

\begin{proof}
  Let $L$ be an admissible solution for $F$, then $L'=L\circ\Psi$ is
  an admissible solution for $F_1$, in particular
  \[L'(z,z)=(L\circ \Psi)(z,z)=L(\Psi(z,z),\tilde{\Psi}(z,z)).\]
\end{proof}

\section{The case of symmetric dynamical systems} \label{secodddiv}

In this section, we focus on the set of odd formal diffeomorphisms
\begin{equation}
  \Fomodd \defeq \{F\in \Fom \mid F(-z)=-F(z)\},
\end{equation}
whose expansions $F(z) = \la z + \sum_{j+k\ge2} F_{jk} z^j \bar z^k$
are characterized by the property
\begla
j+k\;\text{even} \enspace\Longrightarrow\enspace
F_{jk}=0,
\edla
and we prove Theorem~\ref{thmB} for the case $F \in \Fomoddcv
= \Fomodd \cap \Fomcv$.

\smallskip 

Suppose first that $F\in\Fomodd$.
As observed in Lemma~\ref{Fodd}, all the $F$-admissible series~$L$ are
even.
It is straightforward that the formal diffeomorphism~$\Phi$
constructed in the proof of Lemma~\ref{Phi}(i) so that
$L=\nu\circ\Phi$ is odd.
We thus have, by Lemma~\ref{exist}, at least one odd formal geometric
normalization for~$F$.
However note that, since according to Lemma~\ref{Phi}(ii) formal geometric
normalizations may differ by any factor $e^{2\pi i \beta(z)}$, not all
of them are odd.
This is in contrast with the case of formal normalizations:

% \Qiaoling{Better define $\Fomodd$ to be the set of $F$ who has an odd representative in its analytic conjugacy class, so that $\Fomodd$ is invariant under conjugacy?.}

\begin{lemma} \label{odd}
  If $F\in \Fomodd$, then all its formal normalizations are odd. In
  particular, if~$F$ is in addition area-preserving, then any
  formal area-preserving conjugacy to its Birkhoff normal form is odd.
\end{lemma}

\begin{proof}
Pick an odd geometric normalization~$\Phi$, for instance by using
Lemma~\ref{Phi}(i) as indicated above, and apply to
$G \defeq \Phi\circ F\circ \Phi^{-1}$ the theory of \cite{CSSW}:
it is straightforward that the basic normalization~$\Phi^*$ of~$G$ is
odd (due to its uniqueness).
We thus get an odd normal form
$N \defeq \Phi^*\circ\Phi\circ F \circ(\Phi^*\circ\Phi)^{-1}$.
Now, by Lemma~3 of \cite{CSSW}, any further normalization~$H$ of~$N$
is odd.
This yields the conclusion since all formal normalizations of~$F$ are
of the form $H\circ \Phi^*\circ \Phi$.
%
%   Let $F\in \Fomodd$, then $F$ has only monomials of odd degrees.
% %\[F(z)=\lambda z+ \sum_{i+j\geq 2 F_{ij}z^i\bar{z}_j,\quad F_{ij}=0\quad \text{if $i+j$ is even\}.\]
% Suppose terms of degrees $2k+1$, $1\leq k\leq n-1$ of $F$ are put in normal form by odd diffeomorphisms, we still denote by $F$:
% \[F(z)=\lambda z+\sum_{1\leq k\leq n-1} c_kz|z|^{2k}+\sum_{i+j=2n+1}F_{ij}z^i\bar{z}^j + h.o.t.,\]
% Take a change of variables $\zeta=\Psi(z)$, where $\Psi(z)=z+\sum_{i+j=2n+1}\Psi_{ij}z^i\bar z^j$ is an odd diffeomorphism, then
% \[\Psi\circ F\circ \Psi^{-1}(\zeta)=\lambda z+\sum_{1\leq k\leq n-1} c_kz|z|^{2k}+\sum_{i+j=2n+1}(F_{ij}+\Psi_{ij}(\lambda^i\bar{\lambda}^j-\lambda))\zeta^i{\bar{\zeta}}^j+h.o.t.
% \]
% Hence nonresonant monomials $\zeta^i\bar{\zeta}^j$, $i\neq j$, $i+j=2n+1$ can be eliminated. By induction, there exists an odd diffeomorphism ((resp. formal odd diffeomorphisms) that conjugate $F$ into its normal form up to any finite order (resp.  normal form ). 
% Moreover, if  $F$ is in addition area-preserving, the odd diffeompshims $\Psi$ above can be chosen  area-preserving. See for example \cite{C2}, Proposition 4.
%
% Now by Lemma 3 of \cite{CSSW}, the formal normalizations between any two normal forms of $F$ are of the form $H(z)=z(1+h(|z|^2)$ which is odd, hence all formal normalizations of $F$ are odd once there exists an odd one.
%
\end{proof}

\subsection{Adaptation of Genecand's theorem to odd area-preserving maps}

To show generic divergence of geometric normalizations in the class
$\Fomoddcv$ for any irrational number $\omega$, we first need an
adaptation of Genecand's Theorem~\ref{area} to this special class.
Recall that, given an analytic area-preserving $F_0\in\Fom$ and an
integer $N\ge 4$, the space $S(F_0,N)$ of local analytic
area-preserving maps with same $(N-1)$-jet as~$F_0$ is parametrised by
generating functions $u(x,y')$ according to \eqref{eqdefgenu}--\eqref{eqdefugen} and given the
product topology induced by the Taylor coefficients of~$u$.
We now assume that~$F_0$ is odd and consider the closed subspace
\begla
\Sodd(F_0,N) \defeq S(F_0,N)\cap \Fomodd.
\edla

\begin{theorem} \label{area2}
If~$F_0$ has non-trivial Birkhoff normal form at order~$N-1$, then a
generic $F\in \Sodd(F_0,N)$ has non-degenerate heteroclinic orbits and non-degenerate Mather
sets in every neighborhood of the origin,
which implies local non-integrability.
\end{theorem}

\begin{proof}
  Clearly,
\begla
\Sodd(F_0,N) = \big\{\,  F_0\circ \Phi^u \mid
u(x,y') = \sum_{k+\ell>N \atop k+\ell\,\text{even}} u_{k\ell} x^k y'^\ell \in\R\{x,y'\} \,\big\}.
\edla
In what follows, we adapt Genecand's proof of Theorem~\ref{area},
  indicating the modifications to be performed in the arguments found in~\cite{G}, but referring
  to~\cite{G} for the details.
The key observation is that, for odd area-preserving diffeomorphisms, periodic points come in pairs. 
%
% Given $\phi_0$ an odd analytic area-preserving diffeomorphism, let
% $S=\Sodd(\phi_0;N)$ denote the set of odd area-preserving analytic
% diffeomorphisms having the same $N-1$ jet at $0$ as $\phi_0$. For
% $\phi\in S$, define the generating function $u$ by means: \David{ENGLISH}
% \[\phi_0^{-1}\circ \phi(\xi,\eta)=(\xi',\eta')\quad \text{iff}\quad \xi'=\xi+u_2(\xi,\eta'),\quad \eta=\eta'+u_1(\xi,\eta').\]
%
% Topology of $S$ is defined by the product topology of the Taylor
% coefficients of generating functions.
%
We claim that:
\begin{quote}
  Given $F\in S=\Sodd(F_0;N)$, an open disk $D$ containing $0$
and an open neighbourhood $B$ of $F$ in $S$, there exists $F_1\in B$
with a transversal heteroclinic orbit
% loop of  transversal heteroclinic orbits
contained in~$D$.
\end{quote}

We first change coordinates.
According to Lemma~\ref{odd}, any formal area-preserving conjugacy
of~$F_0$ to its Birkhoff normal form is odd.
Truncating at order~$N-1$ and making use of the addendum to
Proposition~\ref{propAreaP}, we get a local analytic odd area-preserving
change of variable $z=x+iy \mapsto \ze=\xi+i\eta$ such that, in the
new coordinate,~$F_0$ and hence~$F$ itself are in Birkhoff normal form
at order $N-1$ and still odd:
\begla
F(\zeta)=\zeta \, e^{2\pi i \beta(|\zeta|^2)} + \cO_N % P(\zeta,\bar{\zeta}),
\quad \text{with}\ens
\beta(R)=\om+b_1R+\cdots+ b_s R^s,
\quad 2s+1\le N-1.
\edla
%
% $P(\ze,\bar\ze)=\cO_N$ 
%
For the sake of simplicity we assume $b_1>0$.

% \blue{CHANGE $\phi$ into~$F$ everywhere...}

\medskip

1. In restriction to the annulus
$A_{\rho} \defeq \{ \rho \sqrt{1/2} \le |\zeta| \le \rho\sqrt{3/2} \}
\subset D$ with $\rho>0$ small enough,
we define new coordinates $(x,y) \in (\R/\Z) \times (-\frac12,\frac12)$ by
\[2\pi x =\arg \zeta,\quad \rho\sqrt{y+1}=|\zeta|.\]
The restriction of~$F$ to~$A_{\rho}$ takes the form of a twist map~$f_\rho$:
\[f_\rho:(x,y)\to (x',y'),\quad (x',y')=(x+a+by+\alpha(x,y),  y+\beta(x,y))\]
with $a\defeq \om+b_1\rho^2$, $b\defeq b_1\rho^2$ and
$C^1$-norm $|\al|_1+|\be|_1=O(\rho^3)$. % $|\alpha|_1+|\beta|_1=O(\rho^3)$.
Mather's theory guarantees the existence of $p/q$ minimal periodic
orbits of~$F$ localized in the annuli with median circle
$|\zeta|=\rho$ and breadth of order $O(\rho^{3/2})$ for any rational
$p/q$ slightly above~$\om$ (Proposition~3.6 in \cite{G}).
Since $F$ is odd, such periodic orbits come in pairs.

\medskip

2. We obtain \emph{hyperbolic} periodic orbits as follows.
Let $u(\xi,\eta')$ be the (even) generating function of~$F$, \ie
$F=F_0\circ\Phi^u$.
%
%\color{blue}
%
Suppose $(Q) = ( Q_j)_{0\leq j\leq q-1} = ((\xi_j,\eta_j))_{0\leq j\leq q-1}$ 
and~$(-Q)$ form a pair of minimal $p/q$ periodic orbits of~$F$;
choosing~$q$ odd, we ensure $F^j(Q_0)\neq-Q_0$ for any~$j$ (indeed,
the contrary would imply that
  $F^{2j}(Q_0)=Q_0$, whence $q\mid2j$, whence $q\mid j$, whence $F^j(Q_0)=Q_0$: contradiction)
and thus $(Q)$ and~$(-Q)$ are two distinct orbits.
We use the nonnegative even polynomial
\[
  P(\xi,\eta') \defeq (\xi^2+\eta'^2)^n
  \prod_{0\leq j\leq q-1}
  [(\xi-\xi_j)^2+(\eta'-\eta_j')^2]
  [(\xi+\xi_j)^2 + (\eta'+\eta_j')^2]
  %
  % \quad 2n\ge N. % n\geq [N/2],
  %
\]
with an integer power~$n$ no less than $N/2$ to perturb~$u$, and consider $F_0\circ \Phi^{u+sP}$
with $s>0$ small.
%
%\color{black}
%
By the ``monotone correspondence formula'' between perturbations of the generating
function $u(\xi,\eta')$ of~$F$ and of the generating function
$h(x,x')$ of the twist map~$f_\rho$ \cite[Lemma~4.2]{G}, $(Q)$ and~$(-Q)$ become the
  only minimal $p/q$ periodic orbits for the perturbed map and they
  are hyperbolic
  (because $P\ge0$ has vanishing $1$-jet but not vanishing $2$-jet at
  each $\pm Q_j$; see the proof of Proposition~4.1 in~\cite{G}).

In the twist coordinates,
\[f_\rho(x+1/2,y)=f_\rho(x,y)+(1/2,0)
  \quad \text{in}\ens (\R/\Z)\times\R.
\]
Lifting the orbits~$(Q)$ and~$(-Q)$ to $(x,y)\in\R\times\Z$, we
  obtain minimal $p/q$ periodic states $(x)=(x_j)_{j\in\Z}$ and $(x')=(x_j+\frac12)_{j\in\Z}$ % in $\R^\Z$
  (with $x_{j+q}=x_j+p$ for all $j\in \Z$)
  such that the translates $(x_{j-k}+\ell)_{j\in\Z}$ of~$(x)$,
  $k,\ell\in\Z$, and those of~$(x')$ are the only minimal $p/q$ periodic
  states.
  All these translates are necessarily isolated, and thus ordered
  in a sequence of neighboring minimal $p/q$ periodic states,
  between which there are asymptotic minimal states.
  Correspondingly, there exist heteroclinic orbits
connecting $( Q)$ and~$(-Q)$ (instead of homoclinic orbits as in~\cite{G}).

\medskip

3. If the heteroclinic orbits just obtained are not \emph{transversal}, then we apply a
further perturbation as follows.
Suppose $F$ has a nontransversal heteroclinic orbit
$(H)=(H_j)_{j\in\Z}$ asymptotic to $(Q)$ and $(-Q)$. Obviously, $(-H)$
is a nontransversal heteroclinic orbit asymptotic to $(-Q)$ and
$(Q)$. By the proof of Proposition~5.1 in~\cite{G}, one can perturb
the generating function~$u$ of~$F$ into $u+sw$ with~$w$ an even $C^\infty$
function supported in the the union of a neighbourhood of~$H_0$ and a
neighbourhood of~$-H_0$ (thus without changing the periodic orbits
$(Q)$ and $(-Q)$),
so that the odd diffeomorphism $F_0\circ\Phi^{u+sw}$ has a
transversal heteroclinic orbit $(H^s)$ asymptotic to the $p/q$
periodic orbits~$(Q)$ and~$(-Q)$.
%
% More precisely, let $\Omega(u) $ be the orthogonal $C^1$ distance
% of the unstable curves of $(Q)$ and the stable curve of $(-Q)$ at $H_0$,
% $w$ is an even function chosen so that $d\Omega(u)$ is surjective.
%
Then the Stone-Weierstrass approximation theorem allows one to find a
suitable polynomial even perturbation
(note that if $w(\xi,\eta')$ is approximated by a polynomial
$P(\xi,\eta')$, then $w(\xi,\eta)=w(-\xi,-\eta')$ is also approximated by
$P^*\defeq P(-\xi,-\eta')$, and thus by $(P+P^*)/2$, which is even).
%
% Notice that $(-H^s)$ provides a transversal heteroclinic orbit
% asymptotic to $(-Q)$ and $(Q)$.
%
% Therefore, we get a \blue{transversal heteroclinic cycle.}
% loop of transversal heteroclinic orbits.
%
 \end{proof}

%\begin{lemma}
% Given any $\tilde{F}\in \gF_{2\omega}$, there exists a unique $F\in \Fomodd$ such that 
% \[\tilde{F}(z^2)=F(z)^2.\]
%\end{lemma}
%\begin{proof} Given $\tilde{F}(z)=\lambda^2 (z+\sum_{j+k\geq 2}\tilde{F}_{jk}z^j\bar{z}^k)$  
%\[F(z)=\lambda z(1+\sum_{j+k\geq 2}\tilde{F}_{rs}z^{2r}\bar{z}^{2s})^{1/2},\]
%satisfies the desired property.
%\end{proof}

%%%%%%%%%%%%%%%%%%%%%%%%%%%%%%%%%%%%%%%%%%%%%%
%%%%%%%%%%%%%%%%%%%%%%%%%%%%%%%%%%%%%%%%%%%%%%

\begin{corollary}\label{CorGenecandodd}
  Let $N\ge1$. Given an odd polynomial
  $J(z)=\la z+\sum_{2\le j+k\le N}J_{jk}z^j\bar z^k$ that is the
  $N$-jet of a formal area-preserving map, for any $\De>0$ there
  exists $F\in\FomJcv$ odd, area-preserving and non-integrable.
\end{corollary}

\begin{proof}
Adapt the proof of Corollary~\ref{CorGenecand}, using the addendum to
Proposition~\ref{propAreaP}.
\end{proof}
 
%%%%%%%%%%%%%%%%%%%%%%%%%%%%%%%%%%%%%%%%%%%%%%
%%%%%%%%%%%%%%%%%%%%%%%%%%%%%%%%%%%%%%%%%%%%%%

\subsection{Proof of Theorem \ref{thmB}}

Exactly as in the proof of Theorem~\ref{main}, it is enough to show
that in each case there exists an example of analytic map for which
the balanced admissible series is divergent; the IPM alternative then guarantees
the divergence of the balanced admissible series of the generic map.

\medskip

{\bf Proof of (i)}. Corollary~\ref{CorGenecandodd}
  with $J(z)=\la z$ ($N=1$)
  yields a non-integrable area-preserving $F\in\Fomoddcv$,
  whose balanced admissible series~$L_F$ is necessarily divergent
  (otherwise, it would be an analytic non-constant first integral
  of~$F$ by Proposition~\ref{integ}(ii)).

\medskip

{\bf Proof of (i')}
It is straightforward to replace $\Fomcv$ by $\Fomoddcv$
in the proof of Theorem~\ref{main}(i').

\medskip

{\bf Proof of (ii)}
We assume that $2\omega$ is super-Liouville and pick
a sequence of natural numbers $(n_p)$ such that
\[|1-(\lambda^2)^{n_p}|^{-1}\geq n_p!.\]
Given an arbitrary odd $J(z)=\lambda z+\sum_{2\le j+k\le N}J_{jk}z^j\bar z^k$,
we construct $F\in\Fomoddcv$ with $N$-jet~$J$ such that
$L_F=\sum L_{rs}z^r\bar z^s$ is divergent by modifying the proof of Theorem~\ref{main}(ii) as follows.

Without loss of generality we assume $\De=1$.
We construct the coefficients $F_{rs}$ by induction on $r+s$.
We take them as prescribed by~$J$ for $r+s\le N$, 
and for $r+s>N$ we choose $F_{rs}=0$ except for $(r,s)$ of the form
$(1,2n_p)$ or $(2n_p+1,0)$ with $n_p\ge N$, in which case the value will
be~$\pm1$.
Such~$F$ will obviously be odd.

We choose $F_{1,2n_p}$ and $F_{2n_p+1,0}$ by writing~\eqref{eqdominantF}, 
\[(1-\lambda^{2n_p})L_{2n_p+1,1}=\lambda \ov{F}_{1,2n_p}+\bar{\lambda}F_{2n_p+1,0}+G_{2n_p+1,1},\]
where $G_{2n_p+1,1}$ is determined by coefficients~$F_{jk}$ with
$j+k\le 2n_p$ and thus already chosen,
and arranging for the inequality
\begin{equation}
  |\lambda \ov{F}_{1,2n_p}+\bar{\lambda}F_{2n_p+1,0}+G_{2n_p+1,1}|
  \ge 2 |\Re \la|
\end{equation}
to hold.
For that, % since $\Re\la=\cos(2\pi\om)\neq0$,
  we may take $F_{1,2n_p} = F_{2n_p+1,0} = 1$ if
  $(\Re\la)\Re(G_{2n_p+1,1})\ge0$ and $F_{1,2n_p} = F_{2n_p+1,0} = -1$
  otherwise.
   With such~$F$, we get $|L_{2n_p+1,1}|\geq 2 n_p! |\cos(2\pi\om)|$ for
   all~$p$ large enough, therefore~$L_F$ is divergent.

\medskip

{\bf Proof of (iii)}.
Assuming that $\omega$ is non-Bruno and given $J(z)=\lambda
z+\sum_{2\leq k\leq N}J_kz^k$ an odd $N$-jet,
we construct an odd map~$F$ that is holomorphic for~$z$ in the
  disc $\{ |z|<\De\}$, has $N$-jet equal to~$J$ and is not
  analytically linearizable.

Since~$J$ is odd, the unique tangent-to-identity formal linearization
of~$J$ is odd too. Truncating it far enough, we obtain an odd polynomial $\Phi(z)=z+\sum \Phi_jz^j$  such that 
\beglab{eqlinJNodd}
  J^*(z) \defeq \Phi\circ J\circ\Phi\ii=\la z+O(z^{2N+1}).
  \edla
   Consider $F^*(z) \defeq \la z+\la z^{2N+1}$,
  $F\defeq\Phi\ii\circ F^*\circ \Phi$
  and let~$\De_0$ denote the radius of convergence of~$F$.
By Corollary~\ref{CorGe}, $F^*$ is not linearizable, hence neither is~$F$,
which implies that, for any $\De\le\De_0$,~$F$ can be viewed as an
element of~$\Fomoddcv$ (both $\Phi$ and $F^*$ are odd) for which the
unique invariant foliation~$\cF_F$ is not analytic (by
Proposition~\ref{integ}(i)) and~$L_F$ is divergent.

One easily checks, as in the proof of Theorem \ref{main} (iii), that the $N$-jet of~$F$ is~$J$.

 \medskip

 {\bf Proof of (iv)}.
 Given an odd polynomial~$J$ that is the $N$-jet of a formal
 area-preserving map, 
Corollary~\ref{CorGenecandodd} allows us to find an odd non-integrable
  area-preserving map~$F$ in~$\FomJcv$, thus with~$L_F$ divergent.
 
%%%%%%%%%%%%%%%%%%%%%%%%%%%%%%%%
%%%%%%%%%%%%%%%%%%%%%%%%%%%%%%%%

%%%%%%%%%%%%%%%%%%%%%%%%%%%%%%%%
%%%%%%%%%%%%%%%%%%%%%%%%%%%%%%%%

\section{Some questions}   \label{secquest}

%%%%%%%%%%%%%%%%%%%%%%%%%%%%%%%%
%%%%%%%%%%%%%%%%%%%%%%%%%%%%%%%%

{\bf Question 1.} Construct an explicit diffeomorphism $$F(\zeta)=e^{2\pi i\omega}\zeta(1-|\zeta|^2)e^{\pi(\zeta-\bar\zeta)}+O(|\zeta|^4)$$
where $\lambda=e^{2\pi i\omega}$ with $\omega$ super-Liouville, or even non-Bruno, with no convergent geometric normalization.

\smallskip

 \noindent{\bf Question 2.} Is it possible to specify features  characterizing an $F$ all of whose geometric normalizations diverge? Are analogues of Birkhoff zones for translated orbits (see \cite{C}) a candidate for this in spite of their dependence on a choice of local coordinates?

\smallskip

 \noindent {\bf Question 3} What is true in case $\omega$ is a Bruno number? Does the existence of weak attraction (or repulsion) influence the existence of an analytic normalization? Recall that in \cite{CSSW} it is proved that for almost all values of $a\in\R$, no analytic normalization exists for the geometrically normalized local diffeomorphism $z\mapsto e^{2\pi\omega}z(1+a|z|^{2d})e^{\pi(z-\bar z)}$, that if $\omega$ is not a Bruno number this is the case for all $a$ while if $\omega$ is a Bruno number it is unknown whether there are exceptional values of $a$ for which an analytic normalization exists.  

%\noindent {\bf Question 4} Can one construct $F$ odd such that $L_F$ is divergent, and then $\tau=\sigma$.  
\medskip

%\noindent {\bf Conjectures (David)} 
%
%1) Are the analytic conjugacy classes of $\Gamma$ and $\tau$ analytic invariants of $F$?
%
%2) If we replace the complex line $(z,z)$ by a curve $(z,c(z))$ with $c$ local analytic diffeo, then $\tau$ is replaced by an element of its analytic conjugacy class. Would this be related with the curve in the alternative proof?

%\section{Existence of smooth geometric normalizations~ ?}
\noindent {\bf Question 5.} Given any integer $k$, does there exist a geometric normalization $\Phi$ of class $C^k$ defined in a neighborhood of 0 depending on $k$? This could be the starting point of an analogy with the setting of Hopf bifurcation where the invariant curves which bifurcate from the fixed point are of class $C^k$ with $k$ tending to $+\infty$ when the curve tends to the fixed point. The case $k=0$ is a consequence of \cite{S} which asserts that any two local contractions are topologically conjugate. 
\smallskip

\noindent {\bf Question 6.} Does a $C^\infty$ geometric normalization exist in case the two following conditions are satisfied?

1) $F$ is a local contraction;

2) $\lambda=e^{2\pi i\omega}$ with $\omega$ a Bruno number.
\smallskip

\noindent A tentative proof could use the Ecalle formalism of mould calculus (see \cite{FMS}).   
\smallskip

\noindent If true, such a result would fill a gap in our study: indeed, neither in \cite{CSSW} nor in the present paper did we use the effective existence of some weak attraction or repulsion in the main results on normalization, that is, the results hold as well in the conservative case where $F$ is formally conjugate to $G$ which preserves individually each circle centered at the fixed point (an exception is section 4.2 of \cite{CSSW}).
On the contrary, such existence plays a key role in the existence of
convergent (even polynomial) normal forms (Proposition 1 of
\cite{CSSW}).

\noindent {\bf Question 7.} {Generalization to higher dimensions?}

Let $F$ be a local diffeomorphisms of  $(\R^{2d},0)$ whose derivative $dF(0)$ has its spectrum
$\lambda_1,\bar \lambda_1\cdots, \lambda_d,\bar\lambda_d$ on the unit circle: $$\lambda_j=e^{2\pi \omega_j},\; j=1,\cdots, d.$$
If the spectrum is not resonant there exists a formal conjugacy to a formal normal form 
$$N(z_1,\cdots, z_d)=\left(\lambda_1z_1\bigl(1+f_1(|z_1|^2,\cdots,|z_d|^2)\bigr),\cdots, \lambda_dz_d\bigl(1+f_d(|z_1|^2,\cdots,|z_d|^2)\bigr)\right).$$  
A natural generalization to local diffeomorphisms of  $(\R^{2d},0)$ of the notion of (non-resonant) geometric normalization could be a conjugacy to a local diffeomorphism $G$ which preserves $d$ mutually transverse foliations of $\R^d$ whose intersections are $d$-dimensional tori.

\appendix
\section{ Complex extension}   \label{appA}

%Consider 
%\[F:\C\to \C,\quad F(z)=\lambda z+\sum_{k+\ell\geq 2}F_{k\ell}z^k\bar{z}^{\ell}.\] Define its complex extension by
% \[F_{\C}:\C\to \C,\quad F_{\C}(z,w)=\lambda z+\sum_{k+\ell\geq 2}F_{k\ell}z^kw^{\ell},\]

Given a formal series 
  \begin{equation}  \label{eqfCtoC}
    f \col \mathbb{C}\to \mathbb{C}, \qquad f(z)=\sum f_{k\ell}\,    z^k\bar{z}^\ell 
  \end{equation}
  we define its complex extension, which we still denote by~$f$ (with slight abuse of notation), by
  \begin{equation}  \label{eqfCtwotoC}
    f \col \mathbb{C}^2\to \mathbb{C}, \qquad f(z,w)=\sum
    f_{k\ell}\,z^kw^\ell.
  \end{equation}
  %
%  \David{}
  %
  Of course, a formal series usually defines no function on~$\C$
  or~$\C^2$, but writing $f\col\C\to\C$ or $f\col\C^2\to\C$ is just a
  way of referring to limits of polynomial functions in the topology
  of formal convergence in $\C[[z,\bar z]]$ or $\C[[z,w]]$.
  For the conjugate formal series $\ov{f(z)}$, defined as the formal limit as $N\to\infty$ of the
  conjugate polynomial functions $\sum_{k+\ell\le N} \ov{f_{k\ell}} \,z^\ell\bar{z}^k$,
  and its complex extension, we use the notations
  \begin{equation} \label{notatifzw}
\ti{f}(z)
\defeq \ov{f(z)}=\sum \ov{f_{k\ell}} \,z^\ell\bar{z}^k,
\qquad
    \ti{f}(z,w) \defeq \sum \ov{f_{k\ell}} \,z^\ell w^k.
    \end{equation}
 The complex extension of $\nu(z)=|z|^2$ is $\nu(z,w)=zw$ and 
   \begin{equation} \label{eqnormfsq}
     |f|^2(z,w) = \nu\big(f(z,w),\tilde{f}(z,w)\big).
     \end{equation}
 
The extension of a formal diffeomorphism $F(z)=\lambda z+\sum_{k+\ell\geq
  2}F_{k\ell}\, z^k\bar{z}^{\ell} \in \Fom$ is defined as
\begin{equation}  \label{notahatF}
  \hat{F}:\C^2\to\C^2,\qquad \hat{F}(z,w)=(F(z,w),\tilde{F}(z,w)).
\end{equation}

The $\C$-linear part of the formal mapping $\hat F$ has two eigenvalues,
$\lambda$ and~$\bar{\lambda}$.
The normal forms~$\hat{N}$ of~$\hat F$ are of the form
\[
  \hat{N}(\xi,\eta)=\Big(\xi \sum\alpha_{2k}(\xi\eta)^k, \eta \sum
  \ov{\alpha}_{2k}(\xi\eta)^{2k}\Big)=:(N(\xi,\eta),\tilde{N}(\xi,\eta))
\]
%
% \marginlabel{David replaced $\hat S$ by~$\hat\Phi$}
%
A normalization that takes $\hat{F}$ to $\hat{N}$ is of the form 
\[(\xi,\eta)=\hat\Phi(z,w)=(\Phi(z,w),\tilde{\Phi}(z,w)),\]
where \[\Phi(z,w)=z+\sum_{k+l\geq 2}\Phi_{kl}z^kw^l.\]
The conjugacy equation
\[\hat\Phi\circ \hat{F}= \hat{N}\circ \hat\Phi,\]
is expressed as
\[(\Phi\circ F,\,\widetilde{\Phi\circ F})=(N\circ \Phi, \,\widetilde{N\circ \Phi}).\]
hence equivalent to 
\[\Phi\circ F=N\circ \Phi.\]

Similarly, if $\Phi$ is a formal geometric normalization of~$F$,
with $L\defeq \nu\circ\Phi$ and $\Ga\in\gG$ satisfying $L\circ
F=\Ga\circ L$,
then conjugacy by the complex extension~$\hat\Phi$ brings~$\hat F$ to the 
%
% A formal diffeomorphism~$\hat\Phi$ is called a geometric
% normalization of~$\hat F$ if it takes~$\hat{F}$ to a
%
geometric normal form 
% {\blue{ 
\begin{equation}\label{comG}
  \hat\Phi\circ \hat{F}\circ \hat\Phi^{-1}=
  (\Phi\circ F\circ \Phi^{-1},\widetilde{\Phi\circ F\circ \Phi^{-1}})
  =:(G,\tilde{G})=\hat{G}
\end{equation}
with $\nu\circ \hat{G}=\Gamma\circ \nu$, % for some $\Ga\in\gG$.
% This is the same thing as the complex extension of a geometric normalization of~$F$.
%The Commutative diagram 1 remains the same after replacing $\R$ by
%$\C$, $\R^2$ by $\C^2$, and $F,\Phi$ by $\hat F,\hat \Phi$. 
%
% Since $\nu\circ \hat\Phi=|\Phi|^2(z,w)=L(z,w)$, it follows that
%
because
 \begin{equation}\label{conj1}
L\circ \hat{F}(z,w)=\Gamma\circ L(z,w). % \quad \text{that is,} \quad L(F(z,w),\tilde F(z,w))=\Gamma(L(z,w)).
 \end{equation}
% %
% \marginlabel{David trying to interpret what was written here...}
% %
% \blue{which is exactly}
%  \[L\circ F=\Gamma\circ L,\quad \text{that is, }\quad
%    L(F(z))=\Gamma(L(z)).\]
 %
% \blue{... END OF SECTION TO BE CHECKED...}

%%%%%%%%%%%%%%%%%%%%%%%%%%%%%%%%%%%%%%%%%%%%%%%%%%%%%%%%

  \section{On formal involutions in one variable}    \label{appforminvol}

We gather here a few elementary facts on formal and convergent
involutions, mainly inspired by [A.Bennett, Annals
1915] pp. 37-38.
%
% it was proved that any nontrivial involution is conjugate to -id by
% a formal diffeomorphism.

Given a formal series without constant term, $\tau(x)\in x\C[[x]]$,
one can easily check that $\tau\circ\tau = \ID$ implies that either
$\tau=\ID$ or
\begin{equation}  \label{eqtauini}
\tau(x) = - x + O(x^2).
\end{equation}
In this appendix, we thus give ourselves $\tau(x)\in x\C[[x]]$ of the
form~\eqref{eqtauini}.
We set
\begin{equation}
  U_\tau \defeq \frac{\ID-\tau}{2} = x+O(x^2), \qquad
  V_\tau \defeq \frac{\ID+\tau}{2} = O(x^2).
\end{equation}
Notice that~$U_\tau$ may be considered as a complex formal
tangent-to-identity diffeomorphism.

\begin{lemma}   \label{leminv1}
  \begin{equation}
    \text{$\tau$ involution}
    \; \Longleftrightarrow\;
    \tau = U_\tau\ii \circ \sig \circ U_\tau
  \end{equation}
  where $\sig = -\ID$.
\end{lemma}

\begin{proof}
  $U_\tau \circ \tau = \frac{\tau-\tau\circ\tau}{2}$, hence
  $\tau$ involution $\,\Longleftrightarrow\, U_\tau\circ\tau = -U_\tau
  \,\Longleftrightarrow\, U_\tau\circ\tau = \sig\circ U_\tau$.
  \end{proof}

  \begin{corollary}
    $\tau$ convergent involution $\,\Longrightarrow\,$ $\tau$
    analytically conjugate to~$\sig$.
  \end{corollary}

  \begin{lemma}   \label{leminv2}
    Suppose that~$\tau$ is a formal involution. Then
    \medskip

    (i)\; $V_\tau\circ\tau = V_\tau$,
    \medskip

    (ii) \; $E_\tau \defeq V_\tau \circ U_\tau\ii$ is $O(x^2)$ and even,
    \medskip

    (iii)\; $U_\tau\ii = \ID + E_\tau$.
  \end{lemma}

  \begin{proof}
    (i): obvious.
    \medskip

\noindent
    (ii): Compose~(i) by $U_\tau\ii$:
    $V_\tau\circ U_\tau\ii = V_\tau\circ\tau\circ U_\tau\ii$, which is
    $V_\tau\circ U_\tau\ii\circ\sig$ by Lemma~\ref{leminv1}, \ie
    $E_\tau$ is $\sig$-invariant.

    More geometrically (if $\tau(x)$ is convergent and real), we are
    just saying that the graph $\{y=\tau(x)\}$ is symmetric with
    respect to $\{x=y\}$, hence the change of coordinates
    $(u,v)=\big(\frac{x-y}2, \frac{x+y}2\big)$ maps it to a graph that
    is symmetric  with
    respect to $\{u=0\}$.

    \medskip

\noindent
    (iii): We compute $(\ID+E_\tau)\circ U_\tau = U_\tau + E_\tau\circ
    U_\tau = U_\tau + V_\tau = \ID$.
\end{proof}

  \begin{corollary}
    If~$\tau$ is a formal involution, then %
    \[
      \tau =     (\ID+E_\tau)\circ\sig\circ(\ID+E_\tau)\ii =
      -(\ID-E_\tau)\circ(\ID+E_\tau)\ii
\;   \text{ with $E_\tau(x) = O(x^2)$ even.}
      \]
  \end{corollary}

  \begin{lemma}   \label{leminv3}
  Suppose that~$\tau$ is a formal involution. Then, for any complex
  formal diffeomorphism~$\psi(x)$,
  \begin{equation}
    \tau = \psi \circ \sig \circ \psi\ii
    \; \Longleftrightarrow\;
    \psi = U_\tau\ii\circ g\enspace \text{with $g(x)$ odd.}
  \end{equation}
  Moreover, $\psi=U_\tau\ii$ is the only conjugacy such that
  $\psi-\ID$ is even.
\end{lemma}

\begin{proof}
  Given a complex
  formal diffeomorphism~$\psi$, let $g\defeq
  U_\tau\circ\psi$.
  We have $\psi \circ \sig \circ \psi\ii=U_\tau\ii\circ g\circ\sig
  \circ g\ii \circ U_\tau$, hence, in view of Lemma~\ref{leminv1},
  $\psi \circ \sig \circ \psi\ii = \tau
    \,\Longleftrightarrow\,
  g\circ\sig\circ g\ii= \sig     \,\Longleftrightarrow\,$  $g$ is odd.

\medskip

Now, $\psi = U_\tau\ii\circ g = g + E_\tau\circ g$ by
Lemma~\ref{leminv2}(iii). Thus, under the assumption that~$g$ is odd, 
$(\psi-\ID)\circ\sig = -g +\ID +E_\tau\circ g\circ \sig = -g +\ID
+E_\tau\circ \sig\circ g = -g +\ID +E_\tau\circ g$ by Lemma~\ref{leminv2}(ii),
and 
\begin{multline*}
\text{$\psi-\ID$ even} \;\Longleftrightarrow\;
  \psi-\ID = (\psi-\ID)\circ\sig \\[1ex] \;\Longleftrightarrow\;
g -\ID+ E_\tau\circ g = -g +\ID +E_\tau\circ g \;\Longleftrightarrow\;
g=\ID.
\end{multline*}
\end{proof}
 
% \marginlabel{\red {Alain. Maybe in the last formula use
% another notation like $\gamma$ for $g$
% which in the proof appears as $g'\circ g^{-1}$?}}

\begin{corollary}
    Given two formal involutions~$\tau$ and~$\tau'$ of the
    form~\eqref{eqtauini}, they are formally conjugate.
    More precisely, given a complex formal diffeomorphism~$\psi$,
    \begin{align*}
      \tau' = \psi \circ \tau\circ\psi\ii
      %
      % & \;\Longleftrightarrow\;
      % %
      % \psi = U_{\tau'}\ii\circ g\circ U_\tau \enspace \text{with $g$ odd}\\[1ex]
      %
      & \;\Longleftrightarrow\;
 \psi = (\ID+E_{\tau'})\circ \gamma\circ (\ID+E_\tau)\ii
        \enspace \text{with $\gamma$ odd.} 
    \end{align*}
  \end{corollary}

%%%%%%%%%%%%%%%%%%%%%%%%%%%%%%%%%%%%%%%%%%%%%%%%%%%%%%%%
%%%%%%%%%%%%%%%%%%%%%%%%%%%%%%%%%%%%%%%%%%%%%%%%%%%%%%%%

  \section{On the jets of area-preserving maps}    \label{appAreaP}

%%%%%%%%%%%%%%%%%%%%%%%%%%%%%%%%%%%%%%%%%%%%%%%%%%%%%%%%
%%%%%%%%%%%%%%%%%%%%%%%%%%%%%%%%%%%%%%%%%%%%%%%%%%%%%%%%

In the course of the proof of Theorem~\ref{main}, we make use of

  \begin{proposition}  \label{propAreaP}
  Let $N\ge1$. Given a polynomial $J(z)=\la z+\sum_{2\le j+k\le N}J_{jk}z^j\bar z^k$ that is the $N$-jet of an
  area-preserving formal map $F\in\Fom$,
  one can find an area-preserving polynomial map $F_0 \in \Fom$ with $N$-jet
  equal to~$J$. % such that its complex extension~$\hat F_0$ is holomorphic in all of~$\C^2$.
\end{proposition}

This is the real two-dimensional version of \cite[Theorem~1]{LPPW}.
We reproduce their elegant proof
adapting it to our setting, for the sake of completeness, and also
because the proof of our Theorem~\ref{thmB} requires the following ``odd'' variant:

\begin{addC}
 If, moreover, the polynomial~$J$ is odd, then one can take~$F_0$ odd.
\end{addC}

Recall that from the start we have identified~$\R^2$ and~$\C$ via
$z=x+i y$.
Here we will work in coordinates $(x,y)$, with respect to which the
area-preserving condition is expressed by saying that the Jacobian
determinant is the constant~$1$.
We find it convenient to write the generic point of~$\R^2$ as a column
vector $\binom x y$.
The proof of Proposition~\ref{propAreaP} and its addendum requires two lemmas.

  \begin{lemma}  \label{lemSpan}
    For any integer $d\ge0$, the $\R$-vector space of $d$-homogeneous
    polynomials in $(x,y)$ is spanned by the $d^{\text{th}}$ powers of linear
    forms of~$\R^2$.
  \end{lemma}

  \begin{lemma}  \label{lemHamilt}
    Given $d\ge0$ integer and $a,b,c\in\R$, the polynomial map
    \begla
\Phi_{a,b}^{c,d} \col
    \begin{pmatrix} x \\[1.5ex] y \end{pmatrix}
    \mapsto
    \begin{pmatrix} x + (d+1) c b (ax+by)^d \\[1.5ex]
      y -  (d+1) c a (ax+by)^d \end{pmatrix}
    \edla
    is area-preserving.
  \end{lemma}

%%%%%%%%%%
%   \begin{lemma}  \label{lemAreaP}
%     %
%     Given non-negative integers~$p$ and~$q$ such that $p+q\ge1$, one
%     can find real-analytic functions of one variable $\ph_{p,q}$ and $\psi_{p,q}$ that extend to
%     entire functions such that the map
%     %
%     \begla
%     %
%     \Phi = \Phi_{p,q;a} \col
%     %
%     \begin{pmatrix} x \\[1.5ex] y \end{pmatrix}
%     %
%     \mapsto
%     %
%     \begin{pmatrix} x \, \ph_{p,q}(a x^p y^q) \\[1.5ex] y \, \psi_{p,q}(a x^p y^q) \end{pmatrix}
%     %
%     \edla
%     %
%     is area-preserving for any $a\in \R$ and has $(p+q+1)$-jet as follows:
%     %
%     \begla
%     %
%     \Phi\begin{pmatrix} x \\[1.5ex] y \end{pmatrix}
%     %
%     =
%     %
%     \begin{pmatrix}
%       %
%       x + \frac 1 {p+1} a x^{p+1} y^q + \cO_{p+q+2} \\[1.5ex]
%       %
%       y - \frac 1 {q+1} a x^p y^{q+1}  + \cO_{p+q+2}
%       %
%     \end{pmatrix}
%     %
%     \edla
%     %
%     (where the notation~$\cO_N$ represents terms of total degree $\ge N$).
%     %
% \end{lemma}
%%%%%%%%%%

  \begin{proof}[Proof of Lemma~\ref{lemSpan}]
    Let~$V_d$ denote the $\R$-vector space of $d$-homogeneous
    polynomials in $(x,y)$, a basis of which is $(x^p y^q)_{p,q\ge0;\,p+q=d}$,
    and let~$W_d$ denote the subspace spanned by the $d^{\text{th}}$ powers of linear
    forms of~$\R^2$.
    Suppose that a linear form $L\in (V_d)^*$ vanishes on~$W_d$. We
    thus have, for any $a,b\in \R$,
    \begla
    0 = \langle L, (ax+by)^d\rangle = \sum_{p+q=d} \binom d p a^p b^q
    \langle L, x^p y^q \rangle,
    \edla
    which implies that~$L$ annihilates all the basis elements $x^p
    y^q$, whence $L=0$. It follows that $W_d = V_d$.
  \end{proof}

  \begin{proof}[Proof of Lemma~\ref{lemHamilt}]
    A direct check is obvious, but it is worth noting
    that~$\Phi^{c,d}_{a,b}$ is nothing but the time-$c$ map of the
    Hamiltonian vector field generated by $H(x,y) = (ax+by)^{d+1}$ (of
    which $ax+by$ is a first integral).
%    
    % The Hamiltonian vector field generated by $H(x,y) \defeq
    % (ax+by)^{d+1}$ is
    % %
    % \begla
    % %
    % X_H = \frac{\pa H}{\pa y}\frac{\pa\;}{\pa x} - \frac{\pa H}{\pa
    %   x}\frac{\pa\;}{\pa y}
    % %
    % = (d+1) b(ax+by)^d \frac{\pa\;}{\pa x} - (d+1) a(ax+by)^d \frac{\pa\;}{\pa y}
    % %
    % \edla
    % %
    % and $ax+by$ is a first integral, hence its time-$c$ map is easily
    % seen to coincide with~$\Phi^{c,d}_{a,b}$, which is thus area-preserving.
    %
  \end{proof}
  
\begin{proof}[Proof of Proposition~\ref{propAreaP} and its addendum]
  We argue by induction on~$N$.
  \smallskip
  
\noindent -- For $N=1$, we have $J(z)=\la z$ and we can take for~$F_0$
the linear map
\begla
    R_\om \col
    \begin{pmatrix} x \\[1.5ex] y \end{pmatrix}
    \mapsto
    \begin{pmatrix}
      x \cos(2\pi\om) - y \sin(2\pi\om) \\[1ex]
      x \sin(2\pi\om) + y \cos(2\pi\om)
    \end{pmatrix}.
    \edla
  
\noindent -- Suppose that the property has been proved for a certain $N\ge1$ and
  let~$J$ be the $(N+1)$-jet of a formal area-preserving $F\in\Fom$.
  The property at rank~$N$ gives us a polynomial area-preserving
$F_N\in\Fom$, odd if~$J$ is odd, such that
$F_N = J + \cO_{N+1}$.
%
% \color{blue}
Comparing with $F=J+\cO_{N+2}$, we can write
$F=F_N+\cO_{N+1} = (\ID+\cO_{N+1})\circ F_N$, since
after factoring out~$R_\om$ we are dealing with tangent-to-identity
formal diffeomorphisms and can use the property
\beglab{eqkhomog}
\text{$G$ of order $\ge2$} \imp
(\ID+ \cO_{N+1})\circ(\ID+G) =
\ID+G + \cO_{N+1}.
\edla
Thus
\begla
F = F_* \circ F_N
\quad \text{with area-preserving}\ens
F_* = \ID + \cO_{N+1}.
\edla
%
%\color{black}
%
%   The property at rank~$N$ gives us a polynomial area-preserving
% $F_N\in\Fom$, odd if~$J$ is odd, such that
% %
%   $F_N = J + \cO_{N+1}$, thus
% %
% \begla
% %
% F = F_* \circ F_N
% %
% \quad \text{with area-preserving}\ens
% %
% F_* = \ID + \cO_{N+1}.
% %
% \edla
% %
% Indeed, the linear part of $F_* \defeq F\circ F_N\ii$ is~$\ID$ and the
% property
%
% \begin{multline}   \label{eqkhomog}
% %
% \text{$G$ of order $\ge2$ and $\gP^{[k]}$ $k$-homogeneous with
%   $k\ge2$} \imp \\[1ex]
% %
% (\ID+\gP^{[k]} + \cO_{k+1})\circ(\ID+G) =
% %
% \ID+\gP^{[k]} +G + \cO_{k+1}
% %
% \end{multline}
% %
% shows that
% %
% a non-trivial term of lowest degree in $F_*-\ID$ must have degree $k\ge
% N+1$,
% %
% otherwise $\ID+G \defeq F_N\circ R_\om\ii$ and $F_*\circ(\ID+G)=F\circ R_\om\ii$ would not have the
% same $N$-jet.
%
We just need to find a polynomial area-preserving map~$\Phi$
% convergent in~$\C^2$
such that
$F_* = (\ID + \cO_{N+2}) \circ \Phi$ or, equivalently (still
by~\eqref{eqkhomog} but at order $N+1$),
\begla
F_* = \Phi + \cO_{N+2}
\edla
and the desired result will be achieved with $F_0 := \Phi \circ F_N$.

To this end, we write % the $(N+1)$-homogeneous part of~$F_*$ as
\beglab{eqdefNphomog}
F_*\begin{pmatrix} x \\[2ex] y \end{pmatrix}
    =
    \begin{pmatrix}
      x + g(x,y) + \cO_{N+2}
      \\[1.5ex]
      y + h(x,y) + \cO_{N+2}
      %
      % x + \sum\limits_{m+n=N+1} g_{mn} x^m y ^n + \cO_{N+2}
      % %
      % \\[2.5ex]
      % %
      % y + \sum\limits_{m+n=N+1} h_{mn} x^m y ^n + \cO_{N+2}
      % %
    \end{pmatrix}
    \edla
    with $(N+1)$-homogeneous polynomials~$g$ and~$h$. % coefficients $g_{mn}, h_{mn} \in \R$.
    Since~$F_*$ preserves area, we have
    \begla
    \begin{vmatrix}
    1 + \dfrac{\pa g}{\pa x} + \cO_{N+1} &
    \dfrac{\pa g}{\pa y} + \cO_{N+1}
    \\[2.5ex]
    \dfrac{\pa h}{\pa x} + \cO_{N+1} &
    1 + \dfrac{\pa h}{\pa y} + \cO_{N+1}
    %
    % 1 + \sum\limits_{m+n=N+1} mg_{mn} x^{m-1} y ^n + \cO_{N+1} &
    % %
    % \sum\limits_{m+n=N+1} ng_{mn} x^m y ^{n-1} + \cO_{N+1}
    % %
    % \\[2.5ex]
    % %
    % \sum\limits_{m+n=N+1} mh_{mn} x^{m-1} y ^n + \cO_{N+1} &
    % %
    % 1 + \sum\limits_{m+n=N+1} nh_{mn} x^m y ^{n-1} + \cO_{N+1}
    % %
    \end{vmatrix} = 1,
  \edla
  and the $N$-homonegeous part of the determinant gives
$\frac{\pa g}{\pa x} = - \frac{\pa h}{\pa y}$
  %
  % whence
  % %
  % \begla
  % %
  % a_p := (p+1)g_{p+1,q} = - (q+1)h_{p,q+1}
  % %
  % \quad \text{for any $p,q\ge0$ such that $p+q=N$}
  % %
  % \edla
  % %
  % (because these are the only terms appearing in the $N$-homogeneous
  % part of the determinant,
  %
(since $0<N<2N$),
whence
\begla
g = \frac{\pa H}{\pa y},
\quad
h = -\frac{\pa H}{\pa x}
\edla
for some $(N+2)$-homogeneous polynomial $H(x,y)$.

Note that $F_*=J\circ F_N\ii + \cO_{N+2}$, thus~$\binom g h$ is
  the $(N+1)$-homogeneous part of $J\circ F_N\ii$.
  In the case where~$J$ is supposed to be odd, $g$ and~$h$ are thus
  odd, which means
  \begla
  \text{$J$ odd and $N$ odd} \imp g=h=0, \ens\text{whence}\ens H=0.
  \edla

Lemma~\ref{lemSpan} gives us $1\le M\le N+3$ and real numbers $c_j,a_j,b_j$
such that
\begla
H = H_1+\cdots +H_M,
\qquad
H_j \defeq c_j (a_j x + b_j y)^{N+2}
\quad \text{for $1\le j \le M$}
\edla
(where of course we take all the $c_j$'s equal to~$0$ when~$J$ is odd
  and~$N$ is odd).
We now use the polynomial area-preserving maps $\Phi_{a_j,b_j}^{c_j,N+1}$ of
Lemma~\ref{lemHamilt} 
%
%\color{blue} 
%
and consider
\begla
\Phi \defeq \Phi_{a_1,b_1}^{c_1,N+1} \circ \cdots\circ \Phi_{a_M,b_M}^{c_M,N+1}.
\edla
By virtue of the property
\begin{multline}   % \label{eqkhomogbis}
\text{$G$ of order $\ge2$ and $\gP$ of order $\ge N+1$} \imp \\[1ex]
(\ID+\gP+ \cO_{N+2})\circ(\ID+G) =
\ID+\gP +G + \cO_{N+2},
\end{multline}
we get $\Phi = F_* + \cO_{N+2}$.
%
%\color{black} 
%
Note that~$\Phi$ is odd if~$N$ is even;
  in the case where~$J$ is supposed to be odd, we thus have~$\Phi$ odd
  whatever the parity of~$N$.
Since~$\Phi$ is area-preserving and polynomial,
our goal is achieved with $F_0 := \Phi \circ F_N$.
%
% We now use the polynomial area-preserving maps $\Phi_{a_j,b_j}^{c_j,N+1}$ of
% Lemma~\ref{lemHamilt}. 
% %
% Property~\eqref{eqkhomog} shows that
% %
% $\Phi \defeq \Phi_{a_1,b_1}^{c_1,N+1} \circ \cdots
% \circ \Phi_{a_M,b_M}^{c_M,N+1}$ satisfies
% %
% $\Phi = F_* + \cO_{N+2}$.
% %
% Note that~$\Phi$ is odd if~$N$ is even;
%   %
%   in the case where~$J$ is supposed to be odd, we thus have~$\Phi$ odd
%   whatever the parity of~$N$.
% %
% Since~$\Phi$ is area-preserving and polynomial,
% our goal is achieved with $F_0 := \Phi \circ F_N$.
%
\end{proof}

  \section{Alternative proof for the holomorphic case}    \label{appaltpf}

%%%%%%%%%%%%%%%%%%%%%%%%%%%%%%%%%%%%%%%%%%%%%%%%%%%%%%%%
%%%%%%%%%%%%%%%%%%%%%%%%%%%%%%%%%%%%%%%%%%%%%%%%%%%%%%%%

  In this appendix, we focus on the case of a germ of complex holomorphic map
\beglab{eqFholom}
  F(z) = \la z + \sum_{k\ge2} f_{k} z^k \in \C\{z\}.
\edla
Notice that its complex extension~$\hat F$ is decoupled:
%when written in the coordinates $(z,w)$:
%
\begla
  \hat F(z,w) = \big( F(z), F^*(w) \big),
  \quad \text{where}\ens
F^*(w) \defeq \ov\la w + \sum_{k\ge2} \ov{f_{k}} w^k.
\edla

The germ~$F$ is an element of the group~$\cG$ of local
biholomorphisms of $(\C,0)$. % (in one variable!).
Let us use the notations~$\ti\cG$ for the group of formal
biholomorphisms and~$\ti\cG_1$ for the subgroup of formal
tangent-to-identity biholomorphisms,
\begla
\ti\cG_1 \defeq \big\{\, h(z) = z + \sum_{n\ge2} h_n z^n \in \C[[z]] \,\big\}.
\edla
It is well-known (and easy to check) that any $F$ in~$\cG$ (or in~$\ti\cG$) admits a unique
\emph{linearization} in~$\ti\cG_1$, \ie there is a unique $h=h_F\in\ti\cG_1$
such that
\beglab{eqlinerariz}
h\circ F(z) = \la h(z),
\edla
because $\la = e^{2\pi i\om}$ with $\om$ irrational.
It is also known (but much subtler) that
\begin{enumerate}
  \item if $\om$ is a Bruno number, then~$h_F$ is convergent for any~$F$ of the form~\eqref{eqFholom};\\[-1.5ex]
  \item if $\om$ is a non-Bruno number, then there are holomorphic germs~$F$ for which~$h_F$ is
    divergent, for instance $F(z) = \la z(1+z/d)^d$, $d\geq1$
    (Theorems~\ref{holo}--\ref{Ge}) or $F(z) = \la z e^z$.
  \end{enumerate}

  Now, the solution~$h_F$ of~\eqref{eqlinerariz} can be interpreted as
  a normalization, and \emph{a fortiori} a geometric normalization
  of~$F$. Indeed,
  \beglab{eqparticularGN}
\{ \, G(z)=\la z \;\text{and}\; \Phi(z)=h_F(z) \,\} \Imp \Phi\circ F = G\circ \Phi
\edla
with $\abs{G(z)}^2=\abs{z}^2$, hence $\Ga(R)=R$ in this case, while
\begla
\La_F \defeq \abs{\Phi(z)}^2 = h_F(z) \ov{h_F(z)} = h_F(z) h^*_F(\bz).
\edla
Here we use the notation
\beglab{eqnotahF}
h_F(z) = \sum_{n\ge1} h_n z^n, \quad
h^*_F(w) = \sum_{n\ge1} \ov{h_n} w^n, \quad
h_1=1.
\edla
Thus, \eqref{eqparticularGN} gives
a particular geometric normalization. Notice that the resonant part of
$\La_F$ is $\rho(R) = \sum_{n\ge2} \abs{h_n}^2 R^n$.

Suppose that $h_F(z)$ is divergent. Then~\eqref{eqparticularGN} gives
a particular divergent geometric normalization, we even have $\La_F$ divergent—what about the other
geometric normalizations?
All of them are divergent because of the equivalence
between Lyapunov stability and analytic linearizability
(see \cite[Chap.~III \S~25]{SM} or \cite{PM3}).
We will now give a direct proof of this fact.

\begin{proposition}   \label{propallGNdiv}
  If $h_F$ is divergent, then all the geometric normalizations~$\Phi$ of~$F$
  are divergent. In fact, all of them have $L=\abs{\Phi}^2$ divergent.
\end{proposition}

% In particular, we get examples of maps with no convergent geometric
% normalization by choosing a $\la=\ee^{2\pi\I\om}$ which does not
% satisfy the Bruno condition and
% %
% \begla
% %
% F(z) = \la (z+z^2),  \;\text{\ie}\;
% %
%                        F(r,\tht) =
%                        %
%                                    \bigg( r(1+2r\cos(2\pi\tht)+r^2)^{1/2},
%                                    %
%                                    \tht+\om+\arctan\frac{r\sin(2\pi\tht)}{1+r\cos(2\pi\tht)} \bigg)
% %
% \edla
%   %
% \begla
% %
% \text{or}\ens\qquad F(z) = \la z\ee^z,  \quad\text{\ie}\ens
% %
%                        F(r,\tht) =
%                        %
%                                    \Big( r \,\ee^{r\cos(2\pi\tht)},
%                                    %
%                                    \tht+\om+r\sin(2\pi\tht) \Big).
% %
% \edla

The proof of Proposition~\ref{propallGNdiv} will be derived from

\begin{lemma}   \label{lemthreetgtid}
  Suppose we have three tangent-to-identity formal series
  \begla
f(z) = z + O(z^2) \in \C[[z]], \ens
g(w) = w + O(w^2) \in \C[[w]], \ens
M(R) = R + O(R^2) \in \C[[R]].
\edla
Let
  \beglab{eqdefCzw}
  C(z,w) \defeq M\big( f(z) g(w) \big) = zw + O_3(z,w) \in \C[[z,w]].
  \edla
  Then
  \begla
  \text{$C(z,w)$ convergent} \Imp \text{$f$, $g$ and~$M$ are
    convergent.}
  \edla
  \end{lemma}

  \bg
  
  \begin{proof}[Proof that Lemma~\ref{lemthreetgtid} implies Proposition~\ref{propallGNdiv}]
%
%    Suppose that $h_F$ is divergent.
    %
    By Corollary~\ref{setF}, the square norm~$L$ of an arbitrary geometric
normalization~$\Phi$ of~$F$ can be written $L = M\circ \La_F$ for a
certain $M(R) = R + O(R^2) \in \R[[R]]$,
\ie $L(z)=M\big( h_F(z) h^*_F(\bz) \big)$.
Suppose that $L(z)$ is convergent. Then its holomorphic extension to~$\C^2$,
\begla
L(z,w)=M\big( h_F(z) h^*_F(w) \big),
\edla
is convergent for~$\abs z$ and~$\abs w$ small
enough and Lemma~\ref{lemthreetgtid} shows that~$h_F$ is convergent.
    \end{proof}

    \smallskip
    
    \begin{proof}[Proof of Lemma~\ref{lemthreetgtid}]
      \textbf{(a)\,} Suppose that $C(z,w)$ defined by~\eqref{eqdefCzw} is convergent.
      Let us consider the inverse~$N$ of~$M$ in the group~$\ti\cG_1$,
      so that
      \beglab{eqfgNC}
f(z) g(w) = N\big( C(z,w) \big).
      \edla
      We want to apply the logarithmic derivative \wrt~$z$,
      \begla
      \pa\log\ph \defeq \frac{\pa \ph}{\ph},
      \quad\text{with}\ens \pa \defeq \frac{\pa\,}{\pa z}
      \edla
(after which we'll apply $\frac{\pa\,\,}{\pa w}$), but
$\pa\log\ph$ is defined only if $\ph(z,w)$ has nonzero constant term,
so we must first take care of the leading terms of $f$, $g$, $N$
and~$C$. We thus set
\begla
f(z) = z f_0(z), \quad g(w) = w g_0(w), \quad
M(R) = R M_0(R), \quad N(R) = R N_0(R),
\edla
where $f_0$, $g_0$, $M_0$ and $N_0$ are formal series with a constant
term equal to~$1$.
From~\eqref{eqdefCzw} we get
\begla
C(z,w) = z w C_0(z,w)
\quad\text{with}\ens
C_0(z,w) \defeq f_0(z) g_0(w) M_0\big( f(z) g(w) \big)
\in \C[[z,w]]
\edla
and we can rewrite~\eqref{eqfgNC} as
\beglab{eqfzgzCzNz}
f_0(z) g_0(w) = C_0(z,w) N_0\big( C(z,w) \big).
\edla

\bg

\noindent \textbf{(b)\,}
Now we can apply $\pa\log$ to both sides of~\eqref{eqfzgzCzNz}. We find
\beglab{eqdefPNz}
\pa\log f_0(z) = \pa\log C_0(z,w) + P\big(C(z,w)\big) \pa C,
\quad \text{where} \ens
P(R) \defeq \frac{N_0'(R)}{N_0(R)} \in \C[[R]].
\edla
Since the \lhs\ does not depend on~$w$, when applying $\ov\pa \defeq
\frac{\pa\,\,}{\pa w}$ we get
\begla
0 = \ov\pa\big(\pa\log C_0(z,w)\big) + P'\big(C(z,w)\big) \ov\pa C \pa C
+P\big(C(z,w)\big) \ov\pa\pa C,
\edla
which we rewrite as
\beglab{eqODEPCzw}
z w D_0(z,w) P'\big(C(z,w)\big) + E_0(z,w) P\big(C(z,w)\big) = H(z,w)
\edla
with
$D_0 \defeq \frac{\ov\pa C \pa C}{zw}= (C_0+w\ov\pa C_0)(C_0+z\pa C_0)$,
$E_0 \defeq \ov\pa \pa C = C_0 + w\ov\pa C_0 + z\pa C_0
+ z w \ov\pa\pa C_0$
and
$H \vphantom{\dfrac{\ov\pa C \pa C}{zw}} \defeq - \ov\pa\big(\pa\log C_0(z,w)\big)$.

\bg

\noindent \textbf{(c)\,}
The point is that, on the one hand, since we assume $C(z,w) \in \C\{z,w\}$, we also
have
$C_0(z,w) = \frac{C(z,w)}{zw} \in \C\{z,w\}$, whence
\begla
D_0(z,w), E_0(z,w), H(z,w) \in \C\{z,w\}
\edla
and, on the other hand, $C_0$, $D_0$ and~$E_0$ have constant term equal to~$1$.
Because of this, the differential equation~\eqref{eqODEPCzw} will
entail that
\beglab{eqCVPR}
P(R) \in \C\{R\}.
\edla

Here is the proof of~\eqref{eqCVPR}: by restricting to the complex
line $\{w=z\}$ and using the principal branch of the square root
$c_0(z) \defeq C_0(z,z)^{1/2} \in \C\{z\}$, we get
\begla
C(z,z) = c(z)^2, \quad\text{where}\ens c(z) \defeq z c_0(z),
\edla
and~\eqref{eqODEPCzw} entails
\beglab{eqODEPzezh}
z^2 P'\big(c(z)^2\big) + e_0(z) P\big(c(z)^2\big) = h(z),
\quad\text{where}\ens
e_0(z) \defeq \frac{E_0(z,z)}{D_0(z,z)},\;
h(z) \defeq \frac{H(z,z)}{D_0(z,z)}.
\edla
Since $c_0(0)=1$, we can view $r=c(z)$ as an element of the group of
tangent-to-identity biholomorphisms~$\cG_1$. Let us denote its inverse
by
\begla
z = d(r) = r d_0(r) \in\cG_1
\quad \text{(with $d_0(0)=1$).}
\edla
Evaluate~\eqref{eqODEPzezh} at $z=d(r)$ and divide the resulting
equation by $d_0(r)^2$:
\beglab{eqODEPrdeux}
r^2 P'(r^2) + a_0(r) P(r^2) = b(r)
\quad %\\%[1ex]
\text{with}\ens
a_0(r)\defeq \frac{e_0\big(d(r)\big)}{d_0(r)^2},
\;
b(r)\defeq \frac{h\big(d(r)\big)}{d_0(r)^2} \;\in\C\{r\}.
\edla
Extract the even parts of~$a_0$ and~$b$:
\begla
\frac{a_0(r)+a_0(-r)}{2} \eqqcolon A_0(r^2), \quad
\frac{b(r)+b(-r)}{2} \eqqcolon B(r^2), \quad A_0(R),\;B(R) \in \C\{R\}.
\edla
The even part of~\eqref{eqODEPrdeux} is
\begla
R P'(R) + A_0(R)P(R) = B(R).
\edla
Since $A_0(0)=1$, this linear non-homogeneous ODE has a regular
singular point at $R=0$ and its unique formal solution $P(R)$ must
thus be convergent.

\bg

\noindent \textbf{(d)\,}
The definition of~$P$ given in~\eqref{eqdefPNz} yields
\begla
N_0(R) = \exp\Big( \int_0^R P(R')\,\dd R' \Big) \in \C\{R\},
\edla
thus $N\in\cG_1$, whence, on the one hand,
$M\in\cG_1$ and, on the other hand, $f(z) g(w) \in\C\{z,w\}$
by~\eqref{eqfgNC}, which concludes the proof.
      \end{proof}

%%%%%%%%%%%%%%%%%%%%%%%%%%%%%%%%%%%%%%%%%%%%%%%%%%%%%%%%%%%%%%%%%%%%%%%
%%%%%%%%%%%%%%%%%%%%%%%%%%%%%%%%%%%%%%%%%%%%%%%%%%%%%%%%%%%%%%%%%%%%%%%

\section{Corrigendum to \cite{CSSW}}

\noindent {\bf page 48, Section 2.2} 

\nidt
non uni{\bf queness} of Formal Normal Forms

\nidt
In Lemma 3, after H(z)= ...,  add {\bf and any conjugacy between $N_1$ and $N_2$ is of this form.}
\medskip

\noindent {\bf page 50}

\nidt
In the proof of Lemma 5, replace the first line by
{\bf As $|\Phi^*(z)|^2=|z|^2$, convergence of $|\Psi|^2$ implies the one of $|H|^2$ hence the one of $H$ and
$e^{2\pi i(\phi^*(z)+b(|z|^2))}$. As the terms in $|z|^{2s}$ in $\phi^*(z)+b(|z|^2)$ originate only in $b$,
the Cauchy-Hadamard formula ...}
\medskip

\noindent {\bf page 51}

\nidt
In the formula for $N_2$, replace ``such that $|\nu_2(z)|^2=...$" by {\bf such that  $\nu_2(|z|^2)=...$}
\medskip

\noindent  {\bf page 57}

\nidt
line 2, suppress {\it and let}

\nidt
line 2 of the proof  of Corollary 4, replace $z(1+\pi z)$ by $\lambda z(1+\pi z)$
\medskip

\noindent {\bf page 60, Theorem 3}

\nidt
In i) replace ``Let $g$ be as in (5.1)" by {\bf Let $g$ define a generic family $g_r$ of circle diffeomorphisms, where $g_r=g|_{|z|=r}$ sends the circle of radius $r$ to the circle of radius $r(1+f(r^2))$.}

\nidt
In ii) replace ``Let $f\not=0$ be as in (5.1)" by {\bf Whatever be $f$}
\medskip

\noindent {\bf page 61} 

\nidt
In Lemma 14, replace ``the coefficients $f_j(t)$ and $g_{jk}(t)$ are polynomial functions of $t\in \C$" by 
{\bf ``the coefficients $f_j(t)$ and $g_{jk}(t)$ are affine functions of $t\in \C$"}

\nidt
 In the proof of Lemma 14 line 2, 
replace the cohomological equation by
$$g(t;z)-n^*(t;|z|^2)+\varphi^*(t;F(t,z))-\varphi^*(t;z)=0.$$
\noindent In the proof of Lemma 14 line 5, 
replace ``since $f_j(t)$ and $g_{pq}(t)$ are polynomial in $t$" by
{\bf ``since $f_j(t)$ and $g_{jk}(t)$ are affine in $t$"}.

\nidt Last-but-one line: replace ``affine subspace'' by \textbf{set}.

%%%%%%%%%%%%%%%%%%%%%%%%%%%%%%%%%%%%%%%%%%%%%%%%%%%%%%%%

\newpage % \vskip1.5cm

\end{document}